\documentclass[reqno,12pt]{amsart}
\usepackage{amsmath,amssymb,amsthm,enumerate,mathrsfs,varioref,colonequals,amscd}
\usepackage{fullpage,url,tikz,xspace,setspace}
\usepackage{microtype}
\usepackage{calc}
\usepackage[alphabetic,lite,nobysame]{amsrefs} 

\makeatletter
\@namedef{subjclassname@2020}{\textup{2020} Mathematics Subject Classification}
\makeatother

\usepackage[all,cmtip]{xy}

\theoremstyle{plain}
\newtheorem{theorem}{Theorem}[section]
\newtheorem{corollary}[theorem]{Corollary}
\newtheorem{lemma}[theorem]{Lemma}
\newtheorem{proposition}[theorem]{Proposition}
\newtheorem{notation}[theorem]{Notation}
\newtheorem{defn}[theorem]{Definition}
\newtheorem{thm}[theorem]{Theorem}

\theoremstyle{definition}
\newtheorem{definition}[theorem]{Definition}
\newtheorem{remark}[theorem]{Remark}

\usepackage{enumerate}
\usepackage{latexsym}
\usepackage{amscd,graphicx,float,xcolor}
\usepackage[all]{xy}
\usepackage{verbatim, tikz, rotate}
\usepackage{etoolbox}
\usepackage{svg}
\usepackage[unicode=true]{hyperref}

\let\oldtocsection=\tocsection
\let\oldtocsubsection=\tocsubsection
\let\oldtocsubsubsection=\tocsubsubsection
\renewcommand{\tocsection}[2]{\hspace{0em}\oldtocsection{#1}{#2}}
\renewcommand{\tocsubsection}[2]{\hspace{1em}\oldtocsubsection{#1}{#2}}
\renewcommand{\tocsubsubsection}[2]{\hspace{2em}\oldtocsubsubsection{#1}{#2}}

\title{Sarnak's conjecture in quantum computing,
 cyclotomic unitary group coranks, and Shimura curves}
\author{Colin Ingalls, Bruce W. Jordan, Allan Keeton, \\ Adam Logan,  and Yevgeny Zaytman}

\address{School of Mathematics and Statistics, Carleton University, Ottawa, ON K1S 5B6, Canada}
\email{cingalls@math.carleton.ca}

\address{Department of Mathematics, Box B-630, Baruch College,
The City University of New York, One Bernard Baruch Way, New York
NY 10010, USA}
\email{bruce.jordan@baruch.cuny.edu}

\address{Center for Communications Research, 805 Bunn Drive,
Princeton, NJ 08540, USA}
\email{agk@idaccr.org}

\address{ICERM, 121 South Main St., 11th Floor, Providence, RI 02903, USA}
\address{The Tutte Institute for Mathematics and Computation, P.O. Box 9703,
Terminal, Ottawa, ON K1G 3Z4, Canada}
\email{adam.m.logan@gmail.com}

\address{Newton, MA 02465, USA}
\email{ykzaytman@alum.mit.edu}


\DeclareMathAlphabet{\mathcalligra}{T1}{calligra}{m}{n}

\newcommand{\bruce}[1]{{\color{purple}\sf [Bruce: #1]}}

\newcommand{\adam}[1]{{\color{green}\sf [Adam: #1]}}

\DeclareMathOperator{\gr}{\mathit{gr}}
\DeclareMathOperator{\ogr}{\overline{\gr}}

\DeclareMathOperator{\nrd}{nrd}

\newcommand{\Z}{{\mathbb{Z}}}

\newcommand{\Fr}{{\rm Fr}}

\newcommand{\N}{{\mathbb N}}
\newcommand{\A}{{\mathbb A}}
\newcommand{\BC}{{\mathbf C}}

\newcommand{\BR}{{\mathbf R}}
\newcommand{\F}{{\mathbf F}}

\renewcommand{\H}{{\mathbf H}}
\newcommand{\B}{{\mathbf B}}
\newcommand{\M}{{\mathcal M}}
\newcommand{\sA}{\mathscr{A}}
\newcommand{\sJ}{\mathscr{J}}

\newcommand{\sM}{{\mathscr M}}

\newcommand{\Gg}{{\mathcal G}}

\newcommand{\SGg}{{\mathrm S}{\mathcal G}}

\newcommand{\Pic}{{\rm Pic}\, }
\newcommand{\Q}{{\mathbf Q}}
\newcommand{\OO}{\mathcal{O}}

\newcommand{\uOO}{\,\protect\underline{\! \OO \!}\,}
\newcommand{\uOOn}{\uOO_n}

\newcommand{\Rn}{R_{\hspace*{.01in}n}}

\newcommand{\Uu}{{\mathcal{U}}}

\newcommand{\T}{{\mathcal T}}
\newcommand{\p}{{\mathfrak p}}
\newcommand{\fa}{{\mathfrak a}}
\newcommand{\ZKt}{\Z_{K_n}^{(2)}}

\newcommand{\fq}{{\mathfrak q}}
\newcommand{\fp}{{\mathfrak p}}
\newcommand{\fP}{{\mathfrak P}}

\newcommand{\fo}{{\mathscr O}}
\newcommand{\sP}{{\mathscr P}}
\newcommand{\sfO}{\,\underline{\! \OO\!}\,}

\newcommand{\ff}{{\mathfrak f}}

\newcommand{\fF}{{\mathfrak F}}
\newcommand{\be}{{\mathbf e}}

\newcommand{\bv}{{\mathbf v}}
\newcommand{\bw}{{\mathbf w}}

\newcommand{\q}{{\mathfrak q}}

\newcommand{\ZFp}{{\Z_F^{(\wp)}}}
\newcommand{\GG}{{\mathbb{G}}}

\newcommand{\Gfp}{G_{\rm{f}}^\wp}
\newcommand{\GGi}{{\GG^{\rm int}}}

\newcommand{\Gif}{{G^{\rm{int}}_{\rm{f}}}}
\newcommand{\Bi}{{B^{\rm int}}}
\newcommand{\Bit}{{B^{\rm{int},\times}}}

\newcommand{\fh}{\mathfrak{h}}
\newcommand{\fhpm}{{\fh^\pm}}
\newcommand{\YiU}{{Y^{\rm{int}}_U}}
\newcommand{\YiUt}{{\tilde{Y}^{\rm{int}}_U}}
\newcommand{\Yi}{{Y^{\rm{int}}}}

\newcommand{\Btp}{{\tilde{B}_\wp}}
\newcommand{\Af}{{\A_{\rm{f}}}}
\newcommand{\Afp}{{\A_{\rm{f}}^\wp}}
\newcommand{\I}{\mathbb{I}}
\newcommand{\If}{\I_{\rm{f}}}
\newcommand{\Ifp}{\I^\wp_{\rm{f}}}
\newcommand{\fO}{\mathfrak{O}}
\DeclareMathOperator{\Spf}{Spf}
\DeclareMathOperator{\Hi}{Hilb}

\DeclareMathOperator{\VM}{VM}
\DeclareMathOperator{\EM}{EM}
\DeclareMathOperator{\Jac}{Jac}
\DeclareMathOperator{\UU}{U}
\DeclareMathOperator{\UT}{U_{2}}
\DeclareMathOperator{\UTz}{U_{2}^{\zeta}}
\DeclareMathOperator{\PUT}{PU_{2}}
\DeclareMathOperator{\SOT}{SO_{3}}
\DeclareMathOperator{\PUTz}{PU_{2}^{\zeta}}
\DeclareMathOperator\Matt{Mat}
\newcommand{\PGL}{{\rm PGL}}
\DeclareMathOperator{\PGLT}{PGL_{2}}

\DeclareMathOperator{\PSUT}{PSU_{2}}
\DeclareMathOperator{\SUT}{SU_{2}}
\DeclareMathOperator{\SLT}{SL_{2}}
\DeclareMathOperator{\PSLT}{PSL_{2}}
\DeclareMathOperator{\GLT}{GL_{2}}

\newcommand{\Th}{3\cdot 2^s}
\newcommand{\w}{\mathrm{w}}

\newcommand{\tts}{{3\cdot 2^s}}

\newcommand{\mO}{{\mathcal O}}
\newcommand{\mN}{{\mathcal N}}
\newcommand{\f}{{\mathfrak f}}

\newcommand{\bH}{\mathbf{H}}
\DeclareMathOperator{\Nrm}{{N}}
\DeclareMathOperator\corank{corank}

\DeclareMathOperator\rank{rank}

\DeclareMathOperator\SL{SL}
\DeclareMathOperator{\Ty}{Type}
\DeclareMathOperator{\RO}{ROrd}

\DeclareMathOperator\PSU{PSU}
\DeclareMathOperator\PU{PU}

\DeclareMathOperator\HH{H}
\DeclareMathOperator\Ann{Ann}
\DeclareMathOperator\cond{cond}
\DeclareMathOperator\ATr{ATr}

\theoremstyle{theorem}

\newtheorem{prop}[theorem]{Proposition}

\theoremstyle{definition}
\newtheorem{definition1}[theorem]{Definition}
\theoremstyle{remark}
\theoremstyle{remark1}
\newtheorem{remark1}[theorem]{Remark}

\theoremstyle{theorem}

\DeclareMathOperator\Gal{Gal}
\DeclareMathOperator\Hom{Hom}
\DeclareMathOperator\Pro{P}

\DeclareMathOperator\Aut{Aut}

\DeclareMathOperator\Nm{N}
\DeclareMathOperator\Stab{Stab}

\DeclareMathOperator\Disc{Disc}
\DeclareMathOperator\Norm{Norm}

\DeclareMathOperator\Ver{Ver}
\DeclareMathOperator\Tr{Tr}
\DeclareMathOperator\Ed{Ed}
\DeclareMathOperator\dist{dist}
\DeclareMathOperator\Star{Star}
\DeclareMathOperator\Cl{Cl}

\DeclareMathOperator\val{val}

\DeclareMathOperator\lcm{lcm}

\DeclareMathOperator\Prin{Prin}

\DeclareMathOperator\core{core}

\makeatletter
\patchcmd{\@part}{\null\vfil}{}{}{}
\patchcmd{\@part}{\par}{.\,\,}{}{}
\makeatother
\newcommand{\Gn}{\mathcal{G}_n}
\newtheorem{mtheorem}[theorem]{Main Theorem}

\newtheorem{reduction}[theorem]{Reduction}

\theoremstyle{mtheorem}
\theoremstyle{remark}
\DeclareMathOperator\Pp{P}
\newcommand{\PGn}{\Pp\!\Gn}

\let\widebar=\overline
\def\O{{\mathcal O}}

\begin{document}

\maketitle
\vspace*{1in}
\begin{abstract}
Sarnak's conjecture in quantum computing concerns when
the groups $\PUT$ and $\PSUT$ over cyclotomic rings
$\Z[\zeta_{n}, 1/2]$ with $\zeta_n=e^{2\pi i/n}$, $4|n$,
are generated by the Clifford-cyclotomic gate set.  
We previously settled this using Euler-Poincar\'{e}
characteristics.
A generalization of Sarnak's conjecture
is to ask when these groups are generated by torsion elements.
An obstruction to this is provided by the corank: a group $G$
has $\corank G>0$ only if $G$ is not generated by torsion elements.
In this paper we study the corank of these cyclotomic unitary groups
in the families $n=2^s$ and $n=\tts$, $n\geq 8$, by letting them
act on Bruhat-Tits trees. The quotients by this
  action are finite graphs whose first Betti number is the corank of the
  group.  Our main result is that for the families $n=2^s$ and $n=3\cdot 2^s$
  the corank grows doubly exponentially in $s$ as $s\rightarrow\infty$; it
  is $0$ precisely when $n= 8,12, 16, 24$, and indeed the
cyclotomic unitary groups are generated by torsion elements (in fact
by Clifford-cyclotomic gates) for these $n$.  We give explicit lower
bounds for the corank in two different ways.  The first is to bound
the isotropy subgroups in the action on the tree by
explicit cyclotomy.  The second is to relate our graphs to Shimura
curves over $F_n=\Q(\zeta_n)^+$ 
via interchanging local invariants
and applying a result of Selberg and Zograf.
We show that the cyclotomy arguments give the stronger
bounds.
 In a final section we
execute a program of Sarnak to show that our results for the
$n=2^s$ and $n=\tts$ families are sufficient to give a second
proof of Sarnak's conjecture.

  \end{abstract}

\clearpage

\clearpage

\section{Introduction}
\label{sec:intro}

We begin by setting the notation to be  used throughout the
paper.
Let $n\geq 8$ with $4|n$ and put $\zeta_n=e^{2\pi i/n}$.  A local
or global 
field $L$ has ring of integers $\Z_L$.  A number field $L$ has discriminant
$\Disc(L)$ and class number $h(L)$.    If $\mathfrak{a}$ is an
ideal of a number field $L$, 
$\Z_L^{\mathfrak{a}}$ is the ring of $\mathfrak{a}$-integers
of $L$, i.e.,
\[
\Z_L^{\mathfrak{a}}\colonequals \{x\in L\mid \text{$x$
is integral for all finite primes $v\!\not | \mathfrak{a}$}\}.
\]
\begin{notation}
{\rm
\label{chicken}
 Set
\begin{align*}
K_n&\colonequals \Q(\zeta_n) \text{ with }\Z_{K_n}=\Z[\zeta_n], \\
F_n & \colonequals K_n^+=\Q(\zeta_n+\overline{\zeta}_n) \text{ with }
\Z_{F_n}=\Z[\zeta_n+\overline{\zeta}_n];\text{  we have}\\
\Z_{K_n}^{(2)}&=\Z_{K_n}[1/2]\text{  and }\Z_{F_n}^{(2)}=\Z_{F_n}[1/2].\\
\H &\colonequals \text{ the Hamilton quaternions over $\Q$, i.e.,}\\
 & \qquad\qquad \text{the rational quaternion
  algebra ramified precisely at $2$ and $\infty$}\\ 
\H_n&\colonequals \H\otimes_{\Q} F_n; \text{  the ramified primes of $\H_n$ are
precisely the infinite primes of $F_n$}\\
 & \qquad\text{since $F_n$ is totally 
real and the ramification
index of primes above $2$ is even}\\
m(\H_n) &\colonequals \text{ the \textsf{Eichler mass} of $\H_n$ with
\textsf{relative mass}} \\
M_n&\colonequals m(\H_n)/h(F_n) \text{ in case $n=2^s$ or $n=\tts$, $n\geq 8$
(see Sections \ref{maxorders} and \ref{puppy1}) .}\\
\{1,i,j,k\}&\colonequals \text{ a fixed
$\Q$-basis of
$\H$ satisfying}\\
&\qquad\qquad\text{$i^2=j^2=k^2=-1$, $ij=-ji=k$, $ki=-ik=j$, $jk=-kj=i$}\\
\widetilde{\M}_n&\colonequals \Z_{F_n}^{(2)}\langle 1, i, j, k\rangle,\text{ the
{\sf standard}  maximal $\Z_{F_n}^{(2)}$-order of $\H_n$}.
\end{align*}
For any subgroup $H\subseteq \H_n^\times$, we put $H_1:=\{h\in H\mid
\Nm_{\H_n/F_n}(h)=1\}$.
}
\end{notation}

Define the Hadamard matrix $H$ and the matrix $T_n$ by 
\begin{equation*}
H:= \frac{1}{2}\left[\begin{array}{rr}
    1+i  & 1+i\\
    1+i & -1-i\end{array}\right]
    \quad\mbox{and}\quad
T_n:=\left[\begin{array}{rr}
    1 & 0\\
    0 & \zeta_n\end{array}\right];
  \end{equation*}
we have $H,T_n\in \UT\big(\Z_{K_n}^{(2)}\big)$.
With $4|n$, the 
{\sf Clifford-cyclotomic group} \cite[Sect.~2.2]{FGKM} (resp.,
{\sf special} Clifford-cyclotomic group) is
\begin{equation*}
\Gg_n=\langle H, T_n\rangle\qquad\mbox{(resp., $\SGg_n=\Gg_n\cap 
\SUT\big(\Z_{K_n}^{(2)}\big)$)}.
\end{equation*}
It is named for being generated by the {\sf Clifford group}, the normalizer
of the group of Pauli matrices in the unitary group $\UT(\BC)$, and the
cyclotomic element $T_n$. 
Put
\begin{equation*}
\UTz\big(\Z_{K_n}^{(2)}\big)=\{\gamma \in\UT\big(\Z_{K_n}^{(2)}\big)
\mid \det\gamma\in\langle\zeta_n\rangle\};
\end{equation*}
we then have $\Gg_n\subseteq \UTz\big(\Z_{K_n}^{(2)}\big)\subseteq 
\UT\big(\Z_{K_n}^{(2)}\big)$.
In general, $\UTz\big(\Z_{K_n}^{(2)}\big)\subsetneqq \UT\big(\Z_{K_n}^{(2)}\big)$.

Various subgroups and quotient groups of
$\UT\big(\Z_{K_n}^{(2)}\big)$ and $\SUT\big(\Z_{K_n}^{(2)}\big)$ 
occur throughout this
paper.  We use the usual notation $H\leq G$ to mean that
$H$ is a subgroup of $G$ and further write $H\ll G$
in case $H\leq G$ and $[G:H]=\infty$. 
For a group $H\leq \GLT(\BC)$, we write $\Pp\!H$ for
$H$ modulo scalars.
The {\sf corank} of a group $G$, denoted $\corank G$,
is the largest integer $k$
such that $G$ surjects onto a free group of rank $k$, or $\infty$ if $G$
surjects onto free groups of arbitrarily large rank.  By analogy with
classical logic gates such as AND, NOT, etc., we often refer to
elements of $\UT\big(\Z_{K_n}^{(2)}\big)$ as {\sf gates}.

As much of this paper concerns asymptotic
growth rates, we make extensive use of standard notation:
\begin{notation}
\label{not}
{\rm
The function 
$g(n)$ is assumed positive for sufficiently large $n$.
\begin{enumerate}[$\bullet$]
\item                      
$f(n)=O(g(n))$ means that $|f|$  is bounded above by $g$
(up to a constant factor) asymptotically: 
$
\limsup_{n\rightarrow \infty}\frac{|f(n)|}{g(n)}<+\infty.
$
\item          
$f(n)=\Omega(g(n))$ means that $f$ is bounded below by g
(up to a constant factor)
asymptotically: $\liminf_{n\rightarrow \infty} \frac{|f(n)|}{g(n)}>0$.
\item          
$f(n)=\Theta(g(n))$ means that $f$ is bounded both above
and below by $g$ 
(up to a constant factor)
asymptotically: $f(n)=O(g(n))$ and $f(n)=\Omega(g(n))$.
\end{enumerate}
}
\end{notation}

Finally we turn to the contents of this paper and the
context for our results.  Single qubit quantum computation concerns
$\UT\big(\Z_{K_n}^{(2)}\big)$ and
$\SUT\big(\Z_{K_n}^{(2)}\big)$, together with their associated groups
$\PUT\big(\Z_{K_n}^{(2)}\big)$, $\PUTz\big(\Z_{K_n}^{(2)}\big)$,  and
$\PSUT\big(\Z_{K_n}^{(2)}\big)$. 
Sarnak conjectured in 2015 \cite[p.~$15^{\rm IV}$]{S} when exact synthesis of
single qubit unitaries over Clifford-cyclotomic gate sets was 
possible --this is the question of when the Clifford-cyclotomic
matrices generate the appropriate cyclotomic unitary group.
We proved Sarnak's conjecture in \cite{IJKLZ} by computing the 
the Euler-Poincar\'{e}
characteristic $\chi(\SUT(\Z_{K_n}^{(2)}))$
using a theorem of Harder and \cite{S3} and subsequently
following \cite{S1}.
The resulting theorem was then:
\begin{theorem}
  \label{rant}
\textup{(}\cite[Thm.~1.3]{IJKLZ}\textup{)}\,\,
  Suppose $4|n$, $n\geq 8$.
  \begin{enumerate}[\upshape (a)]
  \item
    \label{rant1}
We have $\Gg_n=\UTz\big(\ZKt\big)$ if and only if $n=8, 12, 16, 24$.
If $\Gg_n\neq\UTz\big(\ZKt\big)$, then we have $\Gg_n\ll\UTz\big(\ZKt\big)$.
\item
  \label{rant2}
We have $\SGg_n=\SUT\big(\ZKt\big)$ if and only if $n=8, 12, 16,24$.
If $\SGg_n\neq\SUT\big(\ZKt\big)$, then $\SGg_n\ll\SUT\big(\ZKt\big)$.
  \end{enumerate}
\end{theorem}
\noindent In \cite{S} Sarnak pointed out that the most important families
in studying the Clifford-cyclotomic group, $\PSUT\big(\ZKt\big)$, 
and $\PUT\big(\ZKt\big)$, $4|n$, are
$n=2^s$, $n=\Th$, $n\geq 8$---he even suggested a program to reduce his
conjecture (now Theorem \ref{rant}) to these two families.
We in fact execute Sarnak's program in the final Section \ref{bird}.

In this paper we consider the  question of when
$\PSUT\big(\ZKt\big)$ and $\PUT\big(\ZKt\big)$ are generated by torsion elements
for the families $n=2^s$ and $n=\tts$.  A special
case of this is when the Clifford-cyclotomic group
generates, which is the subject of Sarnak's conjecture. For
this more general problem, the techniques of Theorem \ref{rant} are
insufficient---the proof of Theorem \ref{rant} in \cite{IJKLZ} 
is specific to the
Clifford-cyclotomic group. Hence here we consider a more difficult
invariant than the Euler-Poincar\'{e} characteristic, namely,
the corank.
This is an obstruction to being generated by torsion elements:
a group $G$ is not generated by torsion elements if 
$\corank G>0$.

Assume that $n=2^s$ or $n=\tts$, $n\geq 8$, for the 
remainder of this introduction.
To study $\corank \PSUT\big(\ZKt\big)$ and $\corank \PUT\big(\ZKt\big)$,
we want our groups to act on a contractible space with an
analyzable quotient. We achieve this  by embedding them
as explicit subgroups $\Gamma_{1,n}$
and $\Gamma_{0,n}$, respectively, of $\Pp\!\H_n^\times$.
The embedding $\PSUT\big(\ZKt\big)\hookrightarrow 
\Gamma_{1,n}\subseteq\Pp\!\H_n^\times$
is straightforward and general
(see \eqref{iso}) with $\Gamma_{1,n}=\Pp\!\widetilde{\M}_{n,1}^\times$. 
The embedding $\PUT\big(\ZKt\big)\hookrightarrow \Gamma_{0,n}\subseteq 
\Pp\!\H_n^\times$ is 
delicate, but is always
possible if $n=2^s$ or $n=3\cdot 2^s$, $n\geq 8$,
and $\Gamma_{0,n}=\Pp\!\widetilde{\M}_n^\times$ (Theorem~\ref{U})
in this case. By our assumptions on $n$, 
there is one prime $\fp:=\fp_n$ of $F_n$ above $2$ and $\H_n$ is
unramified at $\fp$.  The groups $\Gamma_{1,n}=\Pp\!\widetilde{\M}_{n,1}^\times$
and $\Gamma_{0,n}=\Pp\!\widetilde{\M}_n^\times$
act on the Bruhat-Tits tree
$\Delta=\Delta_{\fp}$ of $\SLT(F_{\fp})$ with finite stabilizers
and finite quotient graphs $\gr_n$ and $\overline{\gr}_n$, respectively.
The graph $\overline{\gr}_n$ may have half-edges
 and hence is an {\sf h-graph} in the 
terminology of Section~\ref{groupgraph}.
For an h-graph $\gr$, the first Betti number is $b_1 (\gr)=\rank H_1(\gr, \Z)$.
It follows from the Bass--Serre theory of groups acting on trees that
\begin{align*}
\rank H_1\big(\PSUT\big(\ZKt\big),\Z\big) &= \corank \PSUT\big(\ZKt\big)
=b_1(\gr_n)\\
\rank H_1\big(\PUT\big(\ZKt),\Z\big )&=\corank \PUT\big(\ZKt\big)=b_1(\ogr_n).
\end{align*}

Our aim is to bound the first Betti numbers $b_1(\gr_n)$
and $b_1(\ogr_n)$ from below. We do this by relating these to the relative mass
$M_n$ of the quaternion algebra $\H_n$ with 
 correction terms $E_{1,n}$, $E_{0,n}$:
\begin{equation}
\label{deer}
b_1(\gr_n)= 1 + M_n- E_{1,n}\quad\text{and}\quad
b_1(\ogr_n)=1+M_n/2-E_{0,n}.
\end{equation}
Equivalently, then, we aim to bound $E_{1,n}$, $E_{0,n}$
from above.
The action of $\Gamma_{i,n}$ on $\Delta$ is free only if
$E_{i,n}=0$, $i=0, 1$; $E_{i,n}$ is a measure of how much the torsion
in $\Gamma_{i,n}$ affects the homology of $\Gamma_{i,n}\backslash \Delta$.
A more conceptual treatment of the $E_{i,n}$ is given in Sections
\ref{fern} and \ref{harried}.

For $n=2^s$ and $n=\tts$, $\Disc(F_n)$ grows doubly exponentially
in $s$, see \eqref{salty}.  As $\Disc(F_n)$ is the dominant term in the formula 
\eqref{stilts}, \eqref{pros}
for $M_n$, this shows that $M_n$ grows doubly exponentially in $s$:
we in fact show in Theorem \ref{soup}
that the leading term in a lower bound for $M_n$ is
\begin{equation*}
\label{quickly}
2^{3s2^{s-3}}\text{ for }n=2^s\quad\text{and}\quad
2^{6s2^{s-3}}\text{ for }n=\tts .
\end{equation*}
So to show that $b_1(\gr_n)$ and $b_1(\ogr_n)$ grow doubly 
exponentially in $s$, it suffices to show that asymptotically
the leading term of an upper bound  
\begin{equation}
\label{burst}
\text{for }E_{i,2^s} \text{ is less than }
2^{3s2^{s-3}}\quad\text{and}\quad \,\, \text{for }E_{i,\tts}\text{ is less than }
2^{6s2^{s-3}}.
\end{equation}

The hard work of bounding $E_{i,n}$ from above is done
in Sections \ref{sec: optimal} -- \ref{three}. The torsion
elements in $\Gamma_{1,n}$ and $\Gamma_{0,n}$ arise from
embedding certain quadratic $\Z_{F_n}$-orders into $\H_n$---we
determine which embeddings occur in Sections
\ref{sec: optimal} -- \ref{harried} and count these
embeddings, which gives rise to class numbers.
We then need to bound these specific class numbers from above,
which we do in Section \ref{two} for $n=2^s$ and in Section
\ref{three} for $n=\tts$. In the end we prove an upper
bound 
\begin{equation}
\label{fin}
\text{for }E_{i,2^s}\text{ with leading term } 2^{s2^{s-3}}\quad\text{and}\quad
\,\,\text{for }
E_{i,\tts}\text{ with leading term }2^{2s2^{s-3}}.
\end{equation}
These upper bounds are better than those required in 
\eqref{burst}. We are thus able to prove that
$b_1(\gr_n)=\corank\PSUT\big(\ZKt\big)$ and $b_1(\ogr_n)=
\corank\PUT\big(\ZKt\big)$ grow in the families $n=2^s$
and $n=\tts$ as the relative mass of $\H_n$ does, which
is doubly exponential in $s$ as $s\rightarrow \infty$.

\begin{mtheorem}
\label{main}
Suppose $n=2^s$ or $n=3\cdot 2^s$, $n\geq 8$.
\begin{enumerate}[\upshape (a)]
\item
  \label{main1}
We have
\begin{align*}
\corank\PSUT\big(\ZKt\big)=&2^{\Theta(2^s \cdot s)}=2^{\Theta(n\log n)}\\
\corank\PUT\big(\ZKt\big)=&2^{\Theta(2^s \cdot s)}=2^{\Theta(n\log n)}.
\end{align*}
\item
  \label{main2}
The corank of $\PUT\big(\ZKt\big)$
or $\PSUT\big(\ZKt\big)$ is equal to
$0$ if and only if $n=8, 12, 16, 24$. In particular, 
$\PUT\big(\ZKt\big)$ and  $\PSUT\big(\ZKt\big)$ are {\bf not} generated by
torsion elements if $n>24$.
\end{enumerate}
\end{mtheorem}
\noindent When combined with the results of Section \ref{bird},
Main Theorem \ref{main}\eqref{main2} gives a second
proof of Sarnak's Conjecture (Theorem \ref{rant}).

In Section \ref{sec:bounding-below} we give a second
approach to bounding $b_1(\gr_n)$, $b_1(\ogr_n)$ from below,
or, equivalently, to bounding $E_{i,n}$ from above.
Whereas our method to prove Main Theorem \ref{main} in
Sections \ref{two}--\ref{three} involves the
nonarchimedean uniformization of 
$\gr_n$ and $\ogr_n$ at the finite prime $\p$ of $F_n$,
the alternative 
concerns archimedean uniformization 
of Shimura curves
at an infinite prime
$\infty_1$ of $F_n$. This analytic approach
uses
a result of Zograf \cite{zograf} which in turn relies on the famous
inequality $\lambda_1\geq 3/16$ of Selberg \cite{Selberg} as well as work of Yang and Yau
\cite{yy}. Deducing a lower bound on $b_1(\gr_n)$ from analytic results
on Riemann surfaces is possible because of {\em ``\,interchanging local
invariants''} for Shimura curves as in 
\v{C}erednik \cite{cer} and Varshavsky \cite{var}: we use this 
theory with $\fp$
 and $\infty_1$ being the primes interchanged.  Interchanging
local invariants is shown to imply that the Betti number  $b_1(\gr_n)$ or
$b_1(\ogr_n)$ is  the genus $g$ of a specific Shimura curve over
the totally real field $F_n$. For this Shimura curve, a result
of Shimizu \cite{shi} computes the area $A$ of the fundamental domain
as either $M_n$ or $2M_n$.  Selberg-Zograf is the inequality
\begin{equation*}
\label{minty}
A<\frac{64}{3}(g+1),
\end{equation*}
and so applied to our Shimura curve allows us to bound the genus
$g$ which is 
$b_1(\gr_n)$ or $b_1(\ogr_n)$ from below in terms of $M_n$.
We content ourselves with working out the simplest application
of Selberg-Zograf to our problem in Section \ref{sec:bounding-below},
which is sufficient to prove Main Theorem \ref{main} 
for $\PSUT\big(\Z_{K_{2^s}}^{(2)}\big)$ and $\PUT\big(\Z_{K_{\tts}}^{(2)}\big)$.

Not surprisingly, our $\infty_1$-bounds coming from
the  general Selberg-Zograf result are weaker than
our $\fp$-bounds coming from 
specific cyclotomic arguments.  
The dominant term in the upper bound \eqref{dinner} at $\infty_1$
 for $E_{1,2^s}$ is
$2^{3s2^{s-3}}$ and the dominant term in the upper bound \eqref{fin} at $\fp$ is
 $2^{s2^{s-3}}$, which is much smaller.  We illustrate
this in Figure 5, where already for $s=8$
the Selberg-Zograf upper bound for $E_{1,2^8}$ is $2.75\cdot 10^{114}$
and the cyclotomic upper bound
 is $7.83\cdot 10^{47}$. However, the cyclotomic $\fp$-bounds 
are specific to the $n=2^s$ and $n=\tts$ families
and to the quaternion algebra $\H_n$,  whereas the
Selberg-Zograf $\infty_1$-bounds would apply to Shimura curves much more
generally.

\thanks{It is a pleasure to thank Peter Sarnak,
  both for his interest and for helpful discussions.
We are also indebted to David Sherman and Shou-Wu Zhang.
  Lastly, we heartily thank the referee for alerting us to the
  possible relevance of Selberg-Zograf 
to our paper, which led us to
  write Section \ref{sec:bounding-below}.}


\section{Graphs and groups acting on trees}
\label{groupgraph}

The standard reference for graphs constructed as  quotients of trees
by group actions is Serre \cite{S2}, which we use freely. However, group actions
arising from totally definite quaternion algebras can invert
an oriented edge, necessitating the modified definitions
of Kurihara \cite[Defn.~3-1]{K}. Following \cite{S2}, a {\sf graph} $\gr$
has vertices $\Ver(\gr)$ and oriented edges $\be$ with initial vertex
$o(\be)$ and terminal vertex $t(\be)$ along with their opposite edges
$\bar\be$ with $\be\neq \bar\be$.
We say {\sf h-graph} when we allow {\sf half-edges}, edges $\be$
with $\be=\bar\be$, as in \cite{K}.  (A graph is an h-graph with no
half-edges.)  We then call edges $\be$ with
$\overline{\be} \ne \be$ {\sf regular edges}. Write $\Ed_r(\gr)$ and
$\Ed_h(\gr)$ for the set of regular and half-edges of $\gr$,
respectively, and $\Ed(\gr) := \Ed_r(\gr) \amalg \Ed_h(\gr)$ for the
set of all edges.  
In this section we extend results of \cite{S2} to 
h-graphs.

\begin{remark1}\label{notloops}
  {\rm
Half-edges $\be$ originate and terminate at the same vertex $o(\be)=t(\be)$.  
They are contractible and do not contribute to the homology of the graph.
They are not self-loops, which are edges $\be$ with $o(\be) = t(\be)$ 
but $\overline \be \ne \be$.
}
\end{remark1}

\begin{definition1}
  \label{gr}
  {\rm
    Let $\gr$ be a finite connected h-graph with $v=v(\gr)=\#\Ver(\gr)$
    vertices. Set
\begin{align*}
e_r=e_r(\gr)  = & \frac{\#\Ed_r(\gr)}{2},\quad
e_h =e_h(\gr)= \frac{\# \Ed_h(\gr)}{2},\quad\text{and}\\
\nonumber &\quad e=e(\gr)=e_r(\gr)+e_h(\gr).
\end{align*}
  }
  \end{definition1}
\noindent The fundamental
group
$\pi_1(\gr)$ is free and its rank is equal to that of $\HH_1(\gr, \Z)$.
The first Betti number 
$b_1(\gr)$ is  $b_1(\gr)=\rank\HH_1(\gr, \Z)$.
By Euler's formula
  $b_1(\gr)=1+e_r-v$.   

An {\sf inversion} of a graph $\gr$ is a pair 
$(g,\be) \in \Aut(\gr) \times \Ed(\gr)$ such that $g\be = \overline \be$.
If $\Gamma\leq \Aut(\gr)$ acts without inversions, then
the natural quotient
$\Gamma\backslash \gr$ with $\Ver(\Gamma\backslash \gr)=\Gamma\backslash\!\!\Ver(\gr)$
and $\Ed(\Gamma\backslash \gr)=\Gamma\backslash\!\Ed(\gr)$ is again a graph.
However, if $\Gamma$ acts on $\gr$ with inversions, 
the natural quotient $\Gamma\backslash \gr$ is only an
h-graph.

For a group $G$, let $G_f$ be the subgroup of $G$ generated by elements of 
finite order and recall that $\corank G$ was defined in the introduction.
Note that $G_f$ is a characteristic, and therefore normal, subgroup of $G$.
Every set of generators of $G$ must include at least $\corank G$
elements of infinite order.  If $\corank G > 0 $, then $G_f \ll G$.

The results below follow from Bass--Serre theory \cite{S2}.
\begin{theorem}
  \label{reddish}
 Suppose that $\Gamma$ is a group acting with finite stabilizers and 
 without inversions on a tree $\Delta$.
 Then there is a short exact sequence
 $$0 \rightarrow \Gamma_{\! f} \rightarrow \Gamma\rightarrow
 \pi_{1}(\Gamma\backslash \Delta) \rightarrow 0 . $$ 
\end{theorem}
\begin{proof}
By \cite[Cor.~1 to Thm.~13]{S2}, there is a surjection
$\Gamma \twoheadrightarrow \pi_1(\Gamma \backslash \Delta)$ with kernel
$H$ generated by the vertex stabilizer groups $\{ \Gamma_{\!\bv}\mid \bv \in
\Ver(\Delta)\}$.
By assumption $H \subset \Gamma_{\! f}$. Clearly, the image of every torsion element of
$\Gamma$ is trivial in the free group $\pi_1(\Gamma \backslash \Delta)$.
\end{proof}
\begin{corollary}\label{cor:surj-fg}
 Suppose that $\Gamma$ is a group acting with finite stabilizers and 
 without inversions on a tree $\Delta$. Then $\corank \Gamma =
 b_1(\Gamma \backslash \Delta)$.
\end{corollary}

Let $\Gamma$ be a group acting faithfully on a tree $\Delta$, possibly
with inversions.  Define $\Delta^\prime$ to be the barycentric subdivision
of $\Delta$. Then $\Gamma$ acts faithfully on $\Delta^\prime$ without inversions.
Moreover the h-graph $\Gamma\backslash \Delta$ is homotopic to
the graph $\Gamma\backslash \Delta^\prime$.
Hence
$$
b_1(\Gamma\backslash \Delta^\prime) =
e(\Gamma\backslash \Delta^\prime)-v(\Gamma\backslash \Delta^\prime) +1=
e_r(\Gamma\backslash \Delta) - v(\Gamma\backslash \Delta) + 1=b_1(\Gamma\backslash \Delta)\, .
$$
\begin{corollary}
  \label{reddish1}
   Suppose that $\Gamma$ is a group acting on a tree $\Delta$ with finite
   stabilizers, but possibly with
  inversions.  Then $$\corank \Gamma = b_1(\Gamma\backslash \Delta)
  = 1 + e_r(\Gamma\backslash\Delta) -v(\Gamma\backslash \Delta)\,
  .$$
\end{corollary}

\subsection{Graph mass}\label{sec:wm}
\begin{definition1}\label{defn: mass}
Let the group $\Gamma$ act on a set $S$.  For $v \in S$, let $\Gamma_v$ be the stabilizer
of $v$ in $\Gamma$.  For each $s \in \Gamma\backslash S$, choose a
representative $\tilde s \in S$.  The {\sf mass} $m(s)$ of
$s$ is $1/\#\Gamma_{\tilde s}$ if $\Gamma_{\tilde s}$ is finite,
or $0$ if it is infinite (clearly this does not depend
on the choice of $\tilde s$).  The {\sf  mass} 
$m(\Gamma\backslash S)$ of $\Gamma\backslash S$ 
is $\sum_{s \in \Gamma\backslash S} m(s)$.
\end{definition1}

For a subgroup $H$ of a group $G$, write $H\unlhd G$ if $H$ is a
{\em normal} subgroup of $G$.
Mass is multiplicative in coverings in the following sense:

\begin{theorem}\label{thm: mass_mult_gen}
Let $\Gamma_0$ act on a set $S$ and let $\Gamma \le \Gamma_0$ be a subgroup of
finite index $d$.  Then $m(\Gamma\backslash S) = dm(\Gamma_0\backslash S)$.
\end{theorem}

\begin{proof}
 Replacing $\Gamma$ by its normal core, 
$\core_{\Gamma_{0}}(\Gamma):=\cap_{\gamma\in\Gamma_{0}}\gamma \Gamma\gamma^{-1}$,
  in
  $\Gamma_0$ \cite[p. 16]{Rob} and applying
  the result to
  $\core_{\Gamma_{0}}(\Gamma)\unlhd \Gamma_{0}$
and $\core_{\Gamma_{0}}(\Gamma)\unlhd\Gamma$,
  we can assume $\Gamma$ is normal in $\Gamma_0$.
  Note that $[\Gamma_{0}:\Gamma]=d$ implies that 
 $d|[\Gamma_{0}:\core_{\Gamma_{0}}(\Gamma)]|d!$ by \cite[1.6.9]{Rob}.

  Now, let $\pi: \Gamma\backslash S \to\Gamma_0\backslash S$ 
be the induced map.
Suppose  $s \in \Gamma_0\backslash S$: 
then $s$ is a union of orbits $w_1, \ldots,
w_r$ of $\Gamma$.  If ${\Gamma_0}_{\tilde s}$ is infinite, then the
same is true of the $\Gamma$-stabilizers of all elements of $s$, so
the contribution to the mass on both sides is $0$ and $s$ may be
ignored.

Now $\Gamma_0/\Gamma$ acts  transitively on $\pi^{-1}(s)=\{ w_1, \ldots,
w_r\}$. Summing up the stabilizers gives
$$ d = \#\left(\Gamma_0/\Gamma\right) = \sum_{i= 1}^r
\#\left(\Gamma_0/\Gamma\right)_{w_i} = \sum_{i= 1}^r
\#\left({\Gamma_0}_{\widetilde{w}_i}\right)/\#\left(\Gamma_{\widetilde{w}_i}\right)\,
. $$
Since
$\#{\Gamma_0}_{\widetilde{w}_i} = \#
{\Gamma_0}_{\tilde s}$, $1\leq i\leq r$,
$$ d =\#{\Gamma_0}_{\tilde s} \sum_{i=1}^r\frac{1}{\#
  \Gamma_{w_i}} =\frac{1}{m(s)}\sum_{w\in\pi^{-1}(s)} m(w) ,$$
so that
$\sum_{\bw\in\pi^{-1}(s)}m(w)= d m(s)$.

Summing over $s\in\Gamma_0\backslash S$ gives the desired result.
\end{proof}

We now apply these results to quotient graphs.  
\begin{definition1}\label{defn: star}
 The {\sf star} of a vertex $\bv\in \Ver(\gr)$
  in an h-graph $\gr$ is 
$$\Star(\bv) := \{ \be \in \Ed(\gr) \mid o(\be) = \bv\}\, .$$
\end{definition1}

\begin{definition1}
\label{onion}
Let $\Gamma$ act on a tree $\Delta$ with $\gr := \Gamma
\backslash \Delta$. Then $\Ver(\gr)$ is identified with
$\Gamma\backslash\Ver(\Delta)$ and the mass $m(\bv)$ is
in the sense of the group $\Gamma$ acting on the set $\Ver(\Delta)$.
  The
 {\sf vertex mass} $\VM(\gr)$ of $\gr$ is
$$\VM(\gr) \colonequals \sum_{\bv\in\Ver(\gr)}m(\bv).$$
 Likewise $\Ed(gr)$ is identified with $\Gamma\backslash\Ed(\Delta)$
with $m(\be)$ in the sense of the group $\Gamma$ acting on $\Ed(\Delta)$.
 The
 {\sf edge mass} $\EM(\gr)$ of $\gr$  is
 \begin{equation*}
\EM(\gr) \colonequals \frac{1}{2}\sum_{\be\in\Ed(\gr)}m(\be)=
\frac{1}{2}\sum_{\bv \in
  \Ver(\gr)}\sum_{\be \in \Star(\bv)}m(\be)\, .
\end{equation*}
 \end{definition1}

\begin{definition1}\label{defn: breakup}
  Let $\gr=\Gamma\backslash\Delta$ be an h-graph. Define
  \begin{align*}
    \Ver_1(\gr) &:= \{\bv\in \Ver(\gr)\mid  m(\bv) = 1\},\\
    \Ver_{<1}(\gr) &:= \{\bv \in  \Ver(\gr)\mid  m(\bv) <1 \}.
    \end{align*}
Set $v_1(\gr) = \#\Ver_1(\gr)$ and $v_{<1}(\gr) = \#\Ver_{<1}(\gr)$ so
that
\[
v(\gr) = v_1(\gr) + v_{<1}(\gr).
\]
\end{definition1}
\begin{lemma}\label{lemma: lower bound e}
  For an h-graph $\gr=\Gamma\backslash\Delta$ we have
  $$e_r(\gr) \ge \EM(\gr) -
  e_h(\gr)\, .$$
\end{lemma}
\begin{proof}
Since $e_r(\gr) = \#\Ed_r(\gr)/2$ and $e_h(\gr) = \#\Ed_h(\gr)/2$, 
$$
\EM(\gr) =\frac{1}{2}\bigg( \sum_{\be \in \Ed_r(\gr)} m(\be) +
\sum_{\be\in \Ed_h(\gr)} m(\be)\bigg)\le e_r(\gr) + e_h(\gr)\, . 
$$\qedhere
\end{proof}
\begin{theorem}\label{theorem: regmass}
Let $\Gamma$ act on a $k$-regular tree $\Delta$ with finite
vertex and edge stabilizer groups; set $\gr := \Gamma\backslash \Delta$.
If $\bv\in\Ver(\gr)$, then
\begin{equation*}
  k = \sum_{\be \in \Star(\bv)} \frac{m(\be)}{m(\bv)} = \frac{1}{m(\bv)}
  \sum_{\be\in\Star(\bv)}m(\be).
  \end{equation*}
\end{theorem}
\begin{proof}
The stabilizer group $\Gamma_{\tilde{\bv}}$ acts on $\Star(\tilde{\bv})$ with
orbits in bijection with $\Star(\bv)$. Explicitly, if $\be \in
\Star{\bv}$ then all the edges in $\Gamma_{\tilde{\bv}} \cdot
\tilde{\be}$ map to $\be$ in $\gr$. This orbit has size 
\[
\#
\Gamma_{\!\tilde{\bv}}\cdot \tilde{\be} = \#
\Gamma_{\!\tilde{\bv}}/\Gamma_{\!\tilde{\be}} = \#\Gamma_{\!\bv}/\Gamma_{\!\be}
= m(\be)/m(\bv).
\]
 Summing over all the orbits gives
the desired result.
\end{proof}

\begin{theorem}\label{thm:three_to_two}
Let $\Gamma$ act on a $k$-regular tree $\Delta$ 
and let $\gr=\Gamma\backslash\Delta$. Then
$ k\VM(\gr) = 2\EM(\gr)$. 
\end{theorem}
\begin{proof}
Fix $\bv \in \Ver(gr)$.  If $\Gamma_{\bv}$ is infinite, then its contribution
to $\VM(\gr)$ is $0$; since the degree of $v$ is finite, the same holds for
the edges of $\Star(\bv)$.  If $\Gamma_{\bv}$ is finite, then we may apply
Theorem~\ref{theorem: regmass}:
$$k\VM(\gr) =\sum_{\bv \in \Ver(\gr)}k m(\bv) = \sum_{\bv \in
  \Ver(\gr)}\sum_{\be \in \Star(\bv)}m(\be)= 2\EM
(\gr).$$
Summing over $\bv$ gives the desired result.
\end{proof}

Graph mass is multiplicative in coverings:
\begin{theorem}\label{thm: mass_mult}
Let $\Gamma_0$ act on a tree $\Delta$ with finite stabilizer groups. Suppose
$\Gamma \leq\Gamma_0$ with $[\Gamma_0:\Gamma]=d$. Set
$\gr_0 := \Gamma_0\backslash \Delta$ and $\gr := \Gamma \backslash \Delta$. 
Then
$$\VM(\gr) = d\VM(\gr_0)\quad\text{ and }\quad
\EM(\gr) = d\EM(\gr_0)\, .$$
\end{theorem}
\begin{proof}
This follows immediately from Theorem~\ref{thm: mass_mult_gen}.
\end{proof}

To understand what happens with inversions we will need the following
definitions and lemmas.

\begin{definition1}Let $\Gamma $ act on a tree $\Delta$ with quotient
  h-graph $\gr :=
  \Gamma \backslash \Delta$. For each $\be \in \Ed(\gr)$ fix
an edge $\tilde{\be} \in \Delta$ lifting it and define
$\widetilde{\Gamma}_{\be}$ to be the subgroup of $\Gamma$
preserving the set $\{\widetilde{\be}, \widebar{\widetilde{\be}}\}$.
Clearly, $\Gamma_{\be} = \Gamma_{\widebar{\be}} \le \widetilde{\Gamma}_{\be}$.
\end{definition1}

\begin{lemma}\label{lem: invert group}
Let $\Gamma $ act on a tree $\Delta$ with quotient graph $\gr :=
\Gamma \backslash \Delta$. If $\be \in \Ed_r(\Gamma
\backslash \Delta)$, then $\widetilde{\Gamma}_{\be}=\Gamma_{\!\be}$. If $\be \in \Ed_h(\gr)$ then 
$\widetilde{\Gamma}_{\be}/\Gamma_{\!\be} \simeq \Z/2\Z\, .$
\end{lemma}
\begin{proof}
 Suppose $\tilde{\be}\in \Ed(\Delta)$ covers
 $\be \in \Ed(\Gamma\backslash \Delta)$. If $\be \ne \widebar{\be}$
 then $\widetilde{\be}$ is not inverted by $\Gamma$ so
 $\widetilde{\Gamma}_\be = \Gamma_\be$.
 
If $\be = \overline{\be}$, then the elements in $\Gamma$ inverting
$\widetilde{\be}$ are those of  $\widetilde{\Gamma}_e$ that are not
in $\Gamma_{\be}$. Consider any two such:  $\gamma_1, \gamma_2$. Then
$\gamma_i(\tilde{\be}) = \overline{\tilde{\be}}$ and
$\gamma_i(\overline{\tilde{\be}}) = \tilde{\be}$ for $i=1,2$ and
$\gamma_1\gamma_2 \in \Gamma_{\!\tilde{\be}} = \Gamma_{\!\be}$. Hence
$\widetilde{\Gamma}_{\be}/\Gamma_{\!\be} \simeq \Z/2\Z .$
\end{proof}

\begin{definition1}
  \label{dashing}
Let $\Gamma $ act on a tree $\Delta$. Let $\Delta_\Gamma$ be the tree
obtained from $\Delta$ by barycentric subdivision of precisely those
edges which are inverted by $\Gamma$. 
\end{definition1}

\begin{lemma}\label{lem: subdivide}
Let $\Gamma $ act on a tree $\Delta$. Let $\bv \in
\Ver(\Gamma\backslash \Delta_\Gamma)$ not appearing in the 
h-graph $\gr := \Gamma \backslash \Delta$. Choose a vertex $\tilde{\bv} \in
\Delta_{\Gamma}$ covering $\bv$ and let $\widetilde{\be}$ be the edge in
$\Delta$ whose subdivision produced the vertex $\widetilde{\bv}$. If
$\be$ is the half-edge in $\gr$ lying under $\widetilde{\be}$, then
$\Gamma_\bv \simeq \widetilde{\Gamma}_\be$ and $m(\bv) = m(\be)/2$. 

\end{lemma}
\begin{proof}
That $\Gamma_{\bv} \simeq  \Gamma_{\widetilde{\bv}} =
\widetilde{\Gamma}_{\widetilde{\be}}\simeq \widetilde{\Gamma}_\be$ is
obvious. The statement on masses follows from Lemma~\ref{lem: invert group}.
  \end{proof}

\begin{definition1}
  \label{costume}
  {\rm 
  Suppose the group $\Gamma$ acts without inversions
  on the tree  $\Delta$ with finite quotient graph.  Suppose
  further that the
 vertex and edge stabilizers are finite. Then the 
  {\sf equivariant Euler characteristic} of \cite[Sect. IX.7]{Br}
  is
  \[
  \chi_{\Gamma}(\Delta)=\VM(\gr)-\EM(\gr).
  \]
  }
\end{definition1}

Finally we can relate vertex and edge mass to the Euler-Poincar\'{e}
characteristic $\chi$ of \cite{S3}:

\begin{theorem}
  \label{EP}
  Suppose $\Gamma$ is a group containing a torsion-free subgroup
  of finite index \textup{(}so ``virtually torsion-free''\textup{)} acting
  on a tree $\Delta$ \textup{(}possibly with
  inversions\textup{)} 
  with finite vertex and edge stabilizer groups and finite quotient
  h-graph $\gr:=\Gamma\backslash \Delta$.  Then
  the Euler-Poincar\'{e} characteristic $\chi(\Gamma)$ is defined and 
  \[
  \chi(\Gamma):=\VM(\gr)-\EM(\gr).
  \]
\end{theorem}
\begin{proof}
  Suppose first that the action of $\Gamma$ on $\Delta$ is without
  inversions.
  Then by \cite[Prop.~7.3(e$'$)]{Br} we have
  \[
  \chi(\Gamma)=\chi_{\Gamma}(\Delta):=\VM(\gr)-\EM(\gr).
  \]

  Now suppose $\Gamma$ acts with inversions on $\Delta$. We have 
$\gr = \Gamma\backslash \Delta$; set  $\gr_\Gamma := \Gamma \backslash
  \Delta_\Gamma$ with $\Delta_\Gamma$ as in  Definition~\ref{dashing}.
  Since $\Gamma$ acts without inversions on
  $\Delta_\Gamma$, by the discussion above,  we have $\chi(\Gamma) = \VM(\gr_\Gamma) - \EM(\gr_\Gamma)$.
  For \mbox{$\be \in \Ed_h(\gr)$} set $\bv_\be$ to be the vertex in
  $\gr_\Gamma$ arising from the barycentric subdivison of an edge
  above $\be$ in $\Delta$. We have $\Gamma_{\bv_{\be}}= 
\Gamma_{\be}$.
Moreover, $\Ver(\gr_\Gamma)$ is the disjoint
  union of $\Ver(\gr)$ and $\{\bv_\be\mid \be \in \Ed_h(\gr)\}$. Also
  for $\be \in \Ed_h(\gr)$ let $\be^\prime$ be its opposite edge in
  $\gr_\Gamma$ so that $\Ed(\gr_\Gamma)$ is the disjoint union of
  $\Ed_r(\gr), \Ed_h(\gr)$, and $\{\be^\prime \mid \be \in
  \Ed_h(\gr)\}$. We also have $\Gamma_{\be^\prime} = \Gamma_{\be}$. By Lemma
\ref{lem: subdivide},
  $$ \VM(\gr_\Gamma) = \VM(\gr) + \sum_{\be \in
    \Ed_h(\gr)}m(\bv_{\be}) = \VM(\gr) + 1/2\sum_{\be \in \Ed_h(\gr)}m(\be), \,$$
   while
   $$ \EM(\gr_\Gamma) = \EM(\gr) + 1/2\sum_{\be
     \in\Ed_h(\gr)}m(\be^\prime) = \EM(\gr) + 1/2\sum_{\be \in
     \Ed_h(\gr)}m(\be)\, .$$
   Thus,
   $$\VM(\gr) - \EM(\gr) = \VM(\gr_\Gamma) - \EM(\gr_\Gamma) =\chi(\Gamma)\, .$$
 
\end{proof}

\section{Unitary groups over
   \texorpdfstring{$\Z_{K_n}^{(2)}=\Z[\zeta_n, 1/2]$}{Z[\unichar{"03B6}n,1/2]} for
  \texorpdfstring{$n=2^s$}{n=2\unichar{"5E}s} and \texorpdfstring{$n=3\cdot2^s$}{n=3\unichar{"B7}2\unichar{"5E}s}}
\label{sec': background}

This section contains the foundational material from number theory
needed for our Main Theorem~\ref{main}.

\subsection{Classifying orders in  quadratic extensions of
Dedekind domains}
\label{quad}

In this subsection we give a simple and useful classification for orders
in relative quadratic extensions of Dedekind domains;
we use this repeatedly in Section~\ref{puppy}.
This subsection is self-contained, with notation independent of
the rest of the paper, and may be of independent interest.

\begin{definition1}
\label{sad}
{\rm
Let $\OO_M$ be a Dedekind domain with field of fractions $M$.
An {\sf extension of Dedekind domains} is the inclusion of $\OO_M$
into its integral closure in a finite extension $K/M$; this integral
closure is a Dedekind domain.  If
$[K:M]=2$, we say that the extension is {\sf quadratic}.
}
\end{definition1}

Let $\OO_L/\OO_M$ be a quadratic extension of Dedekind domains with
$\OO_L$ having field of fractions $L$.
We will classify $L$-orders $\OO\subseteq\OO_{L}$ containing $\OO_M$.
\begin{definition1}
  \label{cond}
  {\rm 
  \begin{enumerate}[\upshape (a)]
  \item
    \label{cond1}
         The {\sf conductor} $\cond(\OO)=:\fF$ of $\OO$ is 
\[
\fF:= \Ann_{\OO_{L}}(\OO_{L}/\OO)= \{x \in \OO_{L} \mid x\OO_L \subset \OO\}.
\]
Then $\fF$ is an $\OO$-ideal and an $\OO_{L}$-ideal.
\item
  \label{cond2}
         We define the  {\sf $\OO_{M}$-conductor}
          $\cond_{\OO_{M}}(\OO)=:\ff$ of $\OO$ to be
the $\OO_{M}$-ideal
$$\ff:= \Ann_{\OO_{M}}(\OO_{L}/\OO)=\fF\cap\OO_{M}. $$
\end{enumerate}
                 }
        \end{definition1}

\begin{lemma}\label{lem:conductor}
  The conductor of $\OO$ 
  is the largest $\OO_{L}$-ideal contained in $\OO$.
\end{lemma}

\begin{proof}
  First, the conductor is contained in $\OO$, for if $x \in \fF$ then
  $x \cdot 1 \in \OO$ by definition.  To see that it is the largest such
  ideal, let $I$ be an ideal of $\OO_L$ properly containing it: say we
  have $x \in I \setminus \fF$.  By definition there exists
  $y \in \OO_L$ such that $xy \notin \OO$; then $xy \in I$,
  so $I \not \subset \OO$.
\end{proof}

\begin{definition1}
  Let $I$ be a nonzero ideal of $\OO_M$.  Define the order 
  $\OO_I \subseteq \OO_L$ to be $\OO_M + I\OO_L$.
\end{definition1}

\begin{remark}
  It is clear that $\OO_I$ is closed under addition and has finite index in 
  $\OO_L$, so proving that it is closed under multiplication shows that it is
  an order.  To do so, let 
  $m + \sum_{j=1}^n i_j \ell_j$ and $m' + \sum_{k=1}^{n'} i'_k \ell'_k$ belong to
  $\OO_I$.  Their product is 
  $$mm' + \sum_j (m'i_j)\ell_j + \sum_k (mi'_k)\ell'_k + \sum_{j,k} (i_ji'_k) \ell_j \ell'_k \ ,$$ 
  and the first term is in $\OO_M$ while all the remaining terms belong to
  $I\OO_L$.
\end{remark}

\begin{proposition}\label{prop:distinct}
  Every $L$-order containing $\OO_M$ is one of the $\OO_I$.  Further,
  $\cond_{\OO_M}(\OO_I)=I$, and $\OO_J \subseteq \OO_I$ if and only
  if $J \subseteq I$.  In particular $\OO_I = \OO_J$ if and only if $I = J$.
\end{proposition}

\begin{proof}
In this proof $I$ denotes a nonzero ideal of $\OO_M$.
  Every $L$-order containing $\OO_M$ is an $\OO_M$-submodule of $\OO_L$,
  and submodules of $\OO_L$ containing $\OO_M$ are naturally in bijection with
  submodules of $\OO_L/\OO_M$.  The $\OO_M$-module $\OO_L/\OO_M$  is a torsion-free $\OO_M$-module (if 
  $x \ne 0 \in \OO_M$ and $xy \in \OO_M$ then $y \in M \cap \OO_L = \OO_M$),
  so it is projective, and clearly the rank is $1$.  Thus by standard theory
  of Dedekind domains its nonzero submodules are exactly those of the form
  $I\OO_L/\OO_M$, which correspond to the submodules $\OO_I \subseteq \OO_L$.
  Further, the correspondence between submodules preserves inclusion, which
  shows that $\OO_J \subseteq \OO_I$ if and only if $J \subseteq I$.

  It certainly holds that $I\OO_L \subset \OO_I$, but no larger 
  $\OO_L$-ideal is contained in $\OO_I$, for if $J$ is an ideal
  contained in $\OO_I$ then $\OO_J \subseteq \OO_I$.
  Finally, as an $\OO_M$-module $\OO_L/\OO_I$ is isomorphic to 
  $(\OO_L/\OO_M)/I(\OO_L/\OO_M)$.  To see that this is isomorphic to 
  $\OO_M/I\OO_M$, tensor the exact sequence
  \[0 \to I(\OO_L/\OO_M) \to \OO_L/\OO_M \to \OO_L/\OO_I \to 0\]
  with $(\OO_L/\OO_M)^{-1}$ to obtain $0 \to I \to \OO_M \to \OO_L/\OO_I \to 0$.
  It follows that the annihilator of $\OO_L/\OO_I$ is exactly $I$, i.e.,  
  that $\cond_{\OO_M}(\OO_I) = I$.
\end{proof}

\begin{remark1}
{\rm As this proof suggests, the conductor
cannot characterize $\OO_M$-orders of $\OO_L$ when $[L:M] > 2$.  For example,
the $\Q(\root 3 \of m)$-orders spanned by 
$1, m \root 3 \of m, m \root 3 \of {m^2}$ and 
$1, m \root 3 \of m, \root 3 \of {m^2}$
have the same conductor.
}
\end{remark1}

\begin{corollary}\label{cor:F-f}
  We have $\fF(\OO)=\ff\OO_{L}(\OO)$.
\end{corollary}

\begin{proof}
  By Proposition~\ref{prop:distinct} we may assume that $\OO = \OO_I$ for
  some $I$.  It is now clear that $\ff(\OO) = I$ and $\fF(\OO) = \OO_L I$.
\end{proof}

We restate our classification as a theorem.

\begin{theorem}\label{thm:cond}
  There is an order-preserving bijection between $L$-orders containing
  $\OO_M$ and nonzero ideals of $\OO_M$, given by $\OO \to \cond_{\OO_M}(\OO)$, whose
  inverse is $I \to \OO_I$.
\end{theorem}

We will now interpret these results slightly differently.
Let $\Gal(L/M)=\langle \sigma\rangle \cong \Z/2\Z$.
\begin{definition1}
  \label{anti}
\begin{enumerate}[\upshape (a)]
\item
  The {\sf trace map} $\Tr: \OO_L \to \OO_M$ is defined by 
    $\Tr(x) = \Tr_{L/M}(x) = x+x^\sigma$.  Define the {\sf antitrace}
    $\ATr:\OO_L \to \OO_L$ by $\ATr(x) = x-x^\sigma$.  Set $\sA := \ATr(\OO_L)$.
\item
  Let $\OO_{L,0} = \{x \in \OO_L \mid \Tr(x) = 0\}$.  Note that
    $\sA \subseteq \OO_{L,0}$.
  \end{enumerate}
\end{definition1}

\begin{proposition}\label{prop:atr-order}
  For every nonzero ideal $I \subseteq \OO_M$ we have
  \[
  \OO_I = \{x \in \OO_L \mid \ATr(x) \in I\sA\}.
  \]
  Also $\sA$ is a projective
  $\OO_M$-module of rank $1$.
\end{proposition}

\begin{proof}
  If $x = m + \sum_j i_j \ell_j \in \OO_I$, then 
  $$\ATr(x) = m - m^\sigma + \sum_j (i_j \ell_j - i_j \ell_j^\sigma) =
  \sum_j i_j(\ell_j - \ell_j^\sigma) \in I\sA .$$  
  Conversely, if $\ATr(x) = \sum_j i_j \ATr(\ell_j)$ with all $i_j \in I$ then 
  $x = k + \sum_j i_j \ell_j$, where $k \in \ker \ATr = \OO_M$.

  For the statement about $\sA$, consider $\ATr$ as a map of
  \mbox{$\OO_M$-modules} $\OO_L \to \OO_L$.  Its kernel is $\OO_M$, a
  saturated submodule of rank $1$, so its image is torsion-free and of
  rank $1$ less than that of $\OO_L$.
  We know that $\OO_L$ is a projective $\OO_M$-module since it is a
  finitely generated torsion-free module over the Dedekind domain
  $\OO_M$. But a submodule of a projective module over a Dedekind domain
  is projective, by
  Kaplansky's theorem on
  hereditary rings \cite[Theorem 2.24]{Lam},
and hence $\sA$ is a
  projective $\OO_M$-module. 
\end{proof}

\begin{corollary}\label{cor:distinct-atr}
  If $\OO, \OO'$ are $\OO_M$-orders of $\OO_L$ with $\ATr(\OO) = \ATr(\OO')$,
  then $\OO = \OO'$.
\end{corollary}

\subsection{Cyclotomic rings}
\label{cyclo}

Suppose $4|n$, $n\geq 8$, with
$\zeta_n=e^{2\pi i /n}$, $K_n=\Q(\zeta_n)$ and $F_n=K_n^+$.
Let $\fp:=\fp_n\subseteq\Z_{F_n}$ be
a prime
above $2$ in $F_n$ with associated valuation
$\val_{\fp}$.  Let  $\fP:=\fP_n\subseteq
\Z_{K_n}$ be a prime above $\fp$.
In case $n=2^s$ or $n=3\cdot 2^s$ we have that
$\fp$ and $\fP$ are unique;
$e(\fP/\fp)=2$ if $n=2^s$ and
$f(\fP/\fp)=2$ if $n=3\cdot 2^s$.
For an order $\OO$ in a number field $F$, we denote the class group
of $\OO$ by $\Cl(\OO)$ and its class number by $h(\OO)=\#\Cl(\OO)$.
If $F$ is totally real, the narrow class group is denoted
$\Cl^+(\OO)$ with its narrow class number
$h^+(\OO)$.

We begin with  an elementary lemma:
\begin{lemma}
  \label{norms}
 Let $m\geq 3$ be an integer with $\zeta_m$ a primitive $m^{\rm th}$ root of $1$
  and $\Nm:=\Norm_{\Q(\zeta_m)/\Q}$. Then
\begin{align*}
\Nm(1+\zeta_m) &=\begin{cases} p\quad\mbox{if $m=2p^t$ for a prime $p$}\\
1\quad\mbox{otherwise},
\end{cases}\\
\Nm(1-\zeta_m) &= \begin{cases} p\quad\mbox{if $m=p^t$ for a prime $p$}\\
1\quad\mbox{otherwise}.
\end{cases}
\end{align*}
\end{lemma}
\begin{proof}
  For an integer $d$, let $\Phi_d(x)\in\Z[x]$ be the
  $d^{\rm th}$ cyclotomic polynomial, so that $x^m-1 = \prod_{d|m} \Phi_d(x)$.
  The proof  is a simple induction using 
  \begin{enumerate}[\upshape (a)]
    \item
      $1=\prod_{1\ne d|m} \Phi_d(-1)$ if $m$ is odd,
    \item
      $m/2=\prod_{1,2\ne d|m} \Phi_d(-1)$ if $m$ is even, and
      \item
        $m=\prod_{1\ne d|m} \Phi_d(1)$.\qedhere \end{enumerate}\end{proof}

Computing norms will show that the ideal $\fp$ in our case is principal.
\begin{definition1}
  \label{gen}
        {\rm
          If $n=2^s\geq 8$, set $p_n=2+\zeta_n+\zeta_n^{-1}=
          \Nm_{K_n/F_n}(\zeta_n+1)$.  If $n=3\cdot 2^s\geq 12$, set
          $p'_n=1+\zeta_n+\zeta_n^{-1}$.
        }
        \end{definition1}
\begin{proposition}
  \label{prin}
  \begin{enumerate}[\upshape (a)]
    \item
\label{prin1}
      Assume $n=2^s$. Then $\Nm_{F_n/\Q}(p_n)=2$. The element
      $p_n$ is a totally positive generator of $\fp_{n}$.
    \item
      \label{prin2}
  Assume $n=3\cdot 2^s$. Then $\Nm_{F_n/\Q}(p_n')=
  \Nm_{F_n/\Q}(p_n'-2)=-2$ and $\fp_n=(p_n')$. All units
  of $F_n$ have norm $1$ and there is no element
  of $\Z_{F_n}$ of norm $2$.
  \end{enumerate}
  \end{proposition}
\begin{proof}
  \eqref{prin1} for $n=2^s$ is standard,
  following from
  \[
  \Nm_{K_n/\Q}(\zeta_n+1)=2 ,
\]
  cf. Lemma~\ref{norms}.
  We prove \eqref{prin2} for $n=3\cdot 2^s$ by
  induction on $s$.  It holds in the case $s = 3$ (or indeed
  $s = 2$) by direct computation.  Inductively suppose that it holds for
  $s = k$.  Now $F_{3\cdot 2^k}$ is the fixed field of the automorphism of
  $F_{3\cdot 2^{k+1}}$ that takes $\zeta_{3 \cdot 2^{k+1}}+\zeta_{3 \cdot 2^{k+1}}^{-1}$
  to its negative.  Hence
  $$\begin{aligned}
    &\Nm_{F_{3\cdot 2^{k+1}}/F_{3\cdot 2^k}} (p'_{3\cdot2^{k+1}}) = (1 + \zeta_{3 \cdot 2^{k+1}} +
    \zeta_{3 \cdot 2^{k+1}}^{-1})(1 - \zeta_{3 \cdot 2^{k+1}} -
    \zeta_{3 \cdot 2^{k+1}}^{-1}) \\
    &\quad= 1 - (\zeta_{3 \cdot 2^{k+1}} + \zeta_{3 \cdot 2^{k+1}}^{-1})^2
    = -1 - \zeta_{3 \cdot 2^k} -\zeta_{3 \cdot 2^k}^{-1} = -p'_{3\cdot 2^k},\\
    \end{aligned}$$
  and $F_{3\cdot 2^k}$ has even degree so
  $\Nm_{F_{3\cdot 2^k}/\Q} (-p'_{3\cdot 2^k}) = \Nm_{F_{3\cdot 2^k}/\Q} (p'_{3\cdot2^k}) = -2$ 
by inductive
  hypothesis.  So $\Nm_{F_{3\cdot 2^{k+1}}/\Q} (p'_{3\cdot 2^{k+1}}) = -2$ by transitivity of norm.
  Similarly,
\[
\Nm_{F_{3\cdot 2^{k+1}}/F_{3\cdot 2^k}} (p'_{3\cdot 2^{k+1}}-2) = p'_{3\cdot 2^k}
\]
and again the result follows.

All units of $\Z_{\Q(\sqrt{3})}$ have norm $1$, so the same is true for units of
$\Z_{F_n}$.  If $q_n \in \Z_{F_n}$ were an element of norm $2$, then 
$q_n/p_n'$ would be a unit of $\Z_{F_n}$ of norm $-1$ (since there is a unique
prime ideal above $2$ in $\Z_{F_n}$), so no such element exists.
\end{proof}
Put $\Z_{K_n,1}^{(2),\times}=\{x\in \Z_{K_n}^{(2),\times}\mid \Nm_{K_n/F_n}(x)=1\}$.
\begin{lemma}
  \label{notsquares}
  Suppose $n=2^s$ or $n=3\cdot 2^s$, $n\geq 8$.
  \begin{enumerate}[\upshape (a)]
  \item
    \label{notsquares1}
    $\Z_{K_n,1}^{(2),\times}=\mu_n$.
  \item
    \label{notsquares2}
  
Put $\zeta=\zeta_n$.
  Then
\[
  \Z_{F_n}^{(2),\times}\ni \Nm_{K_n/F_n}(1+\zeta)=(1+\zeta)(1+\overline{\zeta})
\notin \big(\Z_{F_n}^{(2),\times}\big)^2.
\]
\end{enumerate}
\end{lemma}
\begin{proof}
For \eqref{notsquares1} there is one prime $\fP$ in $K_n$ above
  $2$, so if $\alpha\in \Z_{K_n}^{(2)}$ with $\alpha\overline{\alpha}=1$, then
  $\alpha\in\Z_{K_n}$, and so $\alpha\in\mu_n$ since $\alpha\overline{\alpha}=1$.

For \eqref{notsquares2}, first let $n=2^s$.  Then
$
(1+\zeta_n)(1+\overline{\zeta}_n)=p_n\notin 
  (\Z_{F_n}^{(2),\times})^2
$
  since $\Nm_{F_n/\Q}(p_n)=2$ by Proposition~\ref{prin}\eqref{prin1}.

  Now let $n=3\cdot 2^s$ and $\zeta=\zeta_n$.  We have
  $\Nm_{K_n/\Q}(1+\zeta)=1$ by Lemma~\ref{norms}, so 
 $\Nm_{K_n/F_n}(1+\zeta)\in \Z_{F_n}^\times\subseteq \Z_{F_n}^{(2),\times}$.  Suppose
  $\Nm_{K_n/F_n}(1+\zeta)=(1+\zeta)(1+\overline{\zeta})=w^2$ for
  $w\in \Z_{F_n}^{(2),\times}$.
Then $\Nm_{K_n/F_n}((1+\zeta)/w)=1$ and $(1+\zeta)/w\in \Z_{K_n}^{(2),\times}$.
By \eqref{notsquares1} we have $(1+\zeta)/w\in\mu_n$, say $(1+\zeta)/w=\zeta^k$
for an integer $k$,
so $1+\zeta=w\zeta^k$.  Look at arguments:
\[
\pi/n=\arg(1+\zeta)=\arg(w\zeta^k)=2\pi k/n,
\]
so $k=1/2$, a contradiction.  Hence for $n=3\cdot 2^s$,
$\Nm_{K_n/F_n}(1+\zeta)\notin \big(\Z_{F_n}^{(2),\times}\big)^2$.\qedhere
\end{proof}

Let $\Z_{F_n,+}^\times$, $\Z_{F_n,+}^{(2),\times}$
denote the totally positive elements of $\Z_{F_n}^\times$, 
$\Z_{F_n}^{(2),\times}$, respectively.

\begin{theorem}
  \label{fields}
  \begin{enumerate}[\upshape (a)]
  \item
    \label{fields1}
  Assume $n=2^s$. Then $\Z_{F_n,+}^\times= \big(\Z_{F_n}^\times\big)^2$
  and $h^+(F_n)=h^+(\Z_{F_n})$, $h(F_n)=h(\Z_{F_n})$ are odd.  Further,
  $\Z_{F_n,+}^{(2), \times}/\big(\Z_{F_n}^{(2),\times}\big)^2 \simeq \Z/2\Z$.
\item
  \label{fields2}
  Assume $n=3\cdot 2^s$.  Then  $h(F_n)=h(\Z_{F_n})=h^+(\Z_{F_n}^{(2)})$
  is odd.  We have $h^+(F_n)=h^+(\Z_{F_n})=2h(\Z_{F_n})=2h(F_n)$ and hence
  $h^+(\Z_{F_n})=h^+(F_n)\equiv 2\pmod{4}$.  Again,
  $\Z_{F_n,+}^{(2),\times}/\big(\Z_{F_n}^{(2),\times}\big)^2 \simeq \Z/2\Z$.
    \end{enumerate}
\end{theorem}
We begin by proving a lemma.
\begin{lemma}
  \label{generalized-weber}
  Let $F/\Q$ be a totally real Galois extension of degree $2^d$.
  Suppose that there is $\alpha \in F$ such that $N(\alpha) < 0$
  and the ideal $(\alpha)$ is
  fixed by $\Gal(F/\Q)$.  Then for every sequence $S = (\pm 1, \pm 1, \dots, \pm 1)$
  of $2^d$ signs whose product is $1$, there is a unit of $\Z_F$ whose signs
  in the real embeddings of $F$ are given by $S$.
\end{lemma}

\begin{proof}
  Let $G = \Gal(F/\Q)$, let $M$ be the product of $\langle \pm 1 \rangle$ indexed
  by the real places of $F$, and let $G$ act on $M$ by permuting the places.
  This induces an action of the group ring $\F_2[G]$ on $M$.  Since $G$
  acts transitively on the real places, an element with exactly one $-1$
  generates $M$.  But $\#\F_2[G] = \#M$, so $M$ is a free module of rank $1$.
  
  Because $G$ is a 2-group, the nilradical of $\F_2[G]$ is the
  augmentation ideal \mbox{$\langle g-1\mid g \in G\rangle$} by 
\cite[Thm.~1.2]{J}.  Since all maximal ideals contain the nilradical and
  the augmentation ideal is clearly maximal, it is the unique maximal
  ideal.  Thus any element of $M$ whose product is $-1$ is the image
  of a generator by a unit of $\F_2[G]$, and hence also generates $M$.

  To conclude, consider $s_\alpha \in M$, the image of $\alpha$ by the
  sign map.  Let $\gamma \in \F_2[G]$ be such that $\gamma(s_\alpha) = S$
  and lift $\gamma$ to a $\Gamma \in \Z[G]$ whose image under the
  augmentation map is $0$.
  Then $\alpha^\Gamma$ has the signs $S$; furthermore, since the ideal
  $(\alpha)$ is fixed by Galois, $\alpha^\Gamma$ generates the ideal
  $(\alpha)^0$, so it is the desired unit.
\end{proof}
\begin{proof}[Proof of Theorem \textup{\ref{fields}}]
  To prove \eqref{fields1}, let $n=2^s$. Then the class number
  $h(F_n)$ is odd \cite[Thm.~10.4(b)]{W} and all totally
  positive units in $\Z_{F_n}^\times$ are squares by Weber's Theorem \cite{We}.
  Hence the narrow class number $h^+(F_n)$ is odd and
  $\Z_{F_n,+}^{(2),\times}/\big(\Z_{F_n}^{(2),\times}\big)^2
    \simeq \Z/2\Z$.

  We now prove \eqref{fields2}.  Let $n=3\cdot 2^s$.  The class number 
$h(F_n)$
  is odd, as we see by applying \cite[Thm.~10.4]{W} to $F_n/\Q(\sqrt{3})$.  
  Since all units of $\Z_{F_n}$ have norm $1$ by Proposition~\ref{prin}\eqref{prin2}, the
  narrow class number $h^+(F_n)$ is even.
  Apply Lemma~\ref{generalized-weber} with $F=F_n$ and $\alpha=p_n'$
  to conclude that 
  $h^+(F_n)= 2h(F_n)$.
  Finally, the unique prime above $2$ in $F_n$ is principal but
  not generated by a totally positive element, so once $2$ is inverted
  there are units of all signatures and so $h^+(\Z_{F_n}^{(2)})=h(F_n)$.
  Let $d = [F_n:\Q]$.
  Then $\Z_{F_n}^{(2),\times}/\big(\Z_{F_n}^{(2),\times}\big)^2\simeq (\Z/2\Z)^{d+1}$ 
with generators
  the classes of the fundamental units, $-1$, and $p_n'$. Since 
  $\Z_{F_n}^{(2)}$ has units of all signatures,
  there is an exact sequence with $R=\Z_{F_n}^{(2)}$:
  \[
  1\longrightarrow \frac{R_+^\times}{(R^\times)^2}\longrightarrow
  \frac{R^\times}{
    (R^\times)^2}\stackrel{\rm sig}{\longrightarrow} \F_2^d\longrightarrow 1 .
  \]
  It follows that $R_+^\times/\big(R^\times\big)^2\simeq \Z/2\Z$.
\end{proof}

\subsection{The Hamilton quaternions}
\label{quat}

Suppose $4|n$, $n\geq 8$.

\begin{definition1}
  \label{def:max-oon}
  {\rm
The {\sf standard} maximal $\Z_{F_n}^{(2)}$-order of $\H_n$ is
\[
\widetilde{\M}_n:=\Z_{F_n}^{(2)} \langle 1,i,j,k\rangle =\Z_{F_n}^{(2)}\langle 1, i, j, (1+i+j+k)/2\rangle .
\]
There is an ideal of $\Z_{F_n}$ whose square is the ideal $(2)$.
In case this ideal is principal with $(2)=(\alpha)^2$, we define a 
maximal $\Z_{F_n}$-order $\M_n$ with
\[
\{1,i,j,(1+i+j+k)/2\}\subseteq \M_n\subseteq \widetilde{\M}_n .
\]
by
\begin{equation}
\label{explicit}
\M\colonequals \M_n=\Z_{F_n}\left\langle 1, \frac{1+i}{\alpha}, \frac{1+j}{\alpha},
\frac{1+i+j+k}{2}\right\rangle,
\end{equation}
where we take $\alpha=\sqrt{2}$ if $8|n$ and $\alpha=1+\sqrt{3}$ if $n=12$.
In particular this explicitly defines $\M_n$ if $n=2^s$ or $n=3\cdot 2^s$,
$n\geq 8$.
}

\end{definition1}

We now introduce the Bruhat-Tits tree $\Delta=\Delta_{\fp}$ for $\SLT(F_{\fp})$
where $F\colonequals F_n$.
We suppose that $n=2^s$ or $n=3\cdot 2^s$ with $n\geq 8$.
 Recall that $\fp:=\fp_n\subseteq \Z_{F}\colonequals \Z_{F_n}$ 
is the unique prime ideal
 above $2$ in $F$; $\fp$ is totally ramified.
 The quaternion algebra $\H\colonequals \H_n$ is unramified at $\fp$ and hence
 at all finite primes of $F_n$.
  The vertices $\Ver(\Delta)$ are given by
  (cf. \cite[Sect.~4]{K}):
  \begin{align}
    \label{ver}
    \Ver(\Delta) &=\PGLT(F_{\fp})/\PGLT(\Z_{F_\fp})\\
 \nonumber   &= \H_\fp^\times/F_{\fp}^\times\M_{\fp}^\times\\
 \nonumber  &=\{\mbox{maximal orders $\M' \subset\H$};
 \mbox{ $\M_{\fq}'=\M_{\fq}$ for all
  primes $\fq\neq \fp$}\}.
  \end{align}
  In the last identification, an element
  $x\in \H_\fp^\times/F_{\fp}^\times\M_{\fp}^\times$
  corresponds to a maximal order $\M'\subseteq \H$ such that
  $\M_{\fp}'=x\M_{\fp}x^{-1}$ and $\M_{\fq}'=\M_{\fq}$ for all
  primes $\fq\neq\fp$. In particular, $x\in \widetilde{\M}_n^\times
  /\Z_{F}^{(2),\times}\M_n^\times$ corresponds to 
the maximal order $\M^x:=x\M x^{-1}$.

  Let $v', v''\in\Ver(\Delta)$ with corresponding
  maximal orders $\M', \M''$ as above. Then $v'$ and $v''$ are
  connected by an edge in $\Delta$ if and only if $\M'\cap\M''$ is an
  Eichler order of level $\fp$, i.e., $\M'\cap \M''$ is
  $H_\fp^\times$-conjugate
  to the order $\left[\begin{smallmatrix} \Z_{F_\fp} &
      \Z_{F_\fp} \\ \fp\Z_{F_\fp} & \Z_{F_\fp}\end{smallmatrix}\right]$
  by the identification $\H_{\fp}=\Matt_{2\times 2}(F_{\fp})$.
  The group $\PGLT(F_{\fp})$ acts on $\Delta$.

Following
Kurihara \cite{K} (who in turn
follows Ihara \cite{I}), define discrete subgroups
$\Gamma_{0}, \Gamma_{1}$ of $\PGLT(F_{\fp})$ arising from
the definite quaternion algebra $\H$.  Let
\[
\H^{\times}_{n,1}=\{\gamma \in \H_n^\times \mid \Norm_{\H_n/F_n}(\gamma)=1\}
\quad\text{and}\quad
 \M_{n,1}^\times=\M_n^\times\cap \H_{n,1}^\times
\]
 with a similar notation
locally at a prime $v$.  Put
\begin{align}
\label{gamma}
  \Gamma_{0,n}=\Gamma_0 & =\Gamma_{0}(\widetilde{\M}_n)\
  =\widetilde{\M}_n^\times/\Z_{F}^{(2),\times}=:\Pp\!\widetilde{\M}_n^\times\\[.02in]
\nonumber \Gamma_{1,n}=\Gamma_1& =\Gamma_{1}(\widetilde{\M}_n)
  =\widetilde{\M}_{n,1}^\times/\pm 1 =:\Pp\!\widetilde{\M}_{n,1}^\times .
\end{align}
The groups  $\Gamma_{0}$ and $\Gamma_{1}$ are discrete
cocompact subgroups of $\PGLT(F_\fp)$ and hence act on the
Bruhat-Tits tree $\Delta :=\Delta_{\fp}$ of $\SLT(F_{\fp})$.
A basic result \cite[Corollary, p. 75]{S2} on this action is:
\begin{lemma}
  \label{action}
  Suppose $\tilde{v}\in\Ver(\Delta)$ and $\gamma\in\GLT(F_{\fp})$.
    Then \[\dist(\tilde{v},\gamma\tilde{v})\equiv \val_{\fp}(\det\gamma)\pmod
    {2}.
    \]
    \end{lemma}
\begin{remark}
\label{dale}
This lemma implies that the group $\Gamma_{1}$ acts on $\Delta$ without
inversions.  On the other hand,
$\Gamma_{0}$ generally does invert edges of $\Delta$.
\end{remark}
\begin{definition1}
  \label{grr}
        {\rm
Set $\gr_n:=
\Gamma_{1}\backslash\Delta$ and 
$\overline{\gr}_n=\Gamma_0\backslash\Delta$.
Then $\gr_n$ is a finite graph and $\overline{\gr}_n$ is a 
finite h-graph.}
\end{definition1}

\begin{remark}
  Since $\Gamma_1\vartriangleleft \Gamma_0$, the covering
\[
\pi:=\pi_n:\gr_n\rightarrow \overline{\gr}_n
\]
is Galois with
covering group $\Gamma_{0}/\Gamma_{1}$.
\end{remark}

\begin{theorem}
  \label{index}
 Suppose $n=2^s$ or $n=3\cdot 2^s$. Put $R=\Z_{F_n}^{(2)}$.
   \begin{enumerate}[\upshape (a)]
  \item
    \label{index1}
There are isomorphisms
  \[\Gamma_{0}/\Gamma_{1}\simeq \frac{R_+^\times}{(R^\times)^2}
  \simeq \Z/2\Z.
  \]
\item
  \label{index2}
 The covering $\gr_n\rightarrow \overline{\gr}_n$ is
 \'{e}tale if $n=2^s$.
 \end{enumerate}
\end{theorem}
We require a preliminary result.
\begin{lemma}
  \label{ideal}
  Let $F\colonequals F_n$, $\H\colonequals \H_n$, and let
 $I$ be a fractional $\Z_{F}^{(2)}$-ideal of $\widetilde{\M}_n$ with
  $x\in \Nm_{\H/F}(I)$ totally positive.  Then there exists $\alpha\in I$ with
  $\Nm_{\H/F}(\alpha)=x$.
\end{lemma}
\noindent This is a generalization of
\cite[Cor.~III.5.9]{V}; we give a proof here.
\begin{proof}
  By \cite[Thm.~III.4.1]{V} there exists an $\alpha\in \H$ with
  norm $x$.  For each prime $v\neq \fp$ let
  \[
  U_v=\{ \beta_v\in \H_{v,1}^\times \mid \beta_v\in\alpha^{-1}I
  \otimes_{\Z_{F}^{(2)}}\Z_{F_v}\}.
  \]
  Now $U_v$ is clearly open in $\H_{v,1}^\times$ and it is nonempty since the
  norm map is locally surjective on ideals.  So by Strong Approximation there
  exists $\beta\in(\prod_{v\neq \fp}U_v)\cap \H_1^\times$.  Then $\alpha\beta\in I $
  and $\Nm_{\H/F}(\alpha\beta)=x$.
\end{proof}
\begin{proof}[Proof of Theorem \textup{\ref{index}}]
  The norm map $\Nm=\Nm_{\H_n/F_n}$ induces an  injective map
  \begin{equation*}
  \overline{\Nm}:\Gamma_0/\Gamma_1=\Pp\!\widetilde{\M}_n^\times
  /\Pp\!\widetilde{\M}_{n,1}^\times\longrightarrow
  \frac{R_+^\times}{(R^\times)^2} .
 \end{equation*}
  It is surjective by Lemma~\ref{ideal} taking $I=\widetilde{M}_n$ 
which is an ideal of norm $\Nm_{\H/F}(I)=\Z_F^{(2)}$.  We have
  $R_+^\times/(R^\times)^2\simeq \Z/2\Z$ by
  Theorem~\ref{fields}, proving \eqref{index1}.

  Now note that if $n = 2^s$ we have shown that there exists
  $\gamma\in\widetilde{\M}^\times$ with $\Nm_{\H/F}(\gamma)=p_n$
  as in Definition~\ref{gen}. Then 
  $\Gamma_0/\Gamma_{1}=\langle [\gamma]\rangle \simeq \Z/2\Z$,
  and Lemma~\ref{action} gives that $\gamma v\neq v$
  for every $v\in \Ver(\gr_n)$.  This proves \eqref{index2},
  since $\Gamma_0/\Gamma_1$
  is generated by $[\gamma]$.
\end{proof}

\subsection{Unitary groups}
\label{uni}

We now consider unitary groups over the cyclotomic rings $\Z_{K_n}^{(2)}$.
\begin{proposition} 
  \label{rink}
  Suppose that $n=2^s$ or $n=3\cdot 2^s$, $n\geq 8$.
 Then 
\[
\PUT\big(\Z_{K_n}^{(2)}\big)/\PSUT\big(\Z_{K_n}^{(2)}\big)
\cong \mu_n/\mu_n^2\cong \Z/2\Z .
\]
\end{proposition}
\begin{proof}
  If $A\in\UT\big(\Z_{K_n}^{(2)}\big)$, then 
$\alpha\colonequals \det A\in \Z_{K_n, 1}^{(2),\times}$ and
    hence $\alpha\in\mu_n$ by Lemma~\ref{notsquares}\eqref{notsquares1}.
  The proposition then follows from the exact sequence
  \[
  1\longrightarrow \PSUT\big(\Z_{K_n}^{(2)}\big)\longrightarrow 
\PUT\big(\Z_{K_n}^{(2)}\big)
 \stackrel{\det}{\longrightarrow}\mu_n/\mu_n^2\cong \Z/2\Z\longrightarrow
 1
 \]
 (in which the determinant map is surjective because 
  $\begin{pmatrix} \mu_n&0\\0&1\\ \end{pmatrix} \subseteq 
\PUT\big(\Z_{K_n}^{(2)}\big)$).
\end{proof}
The following observation is easy to check:
\begin{proposition}
  \label{standard}
  The map $\SUT(K_n)\rightarrow \H_{n,1}^\times$ defined by
  \[
  \left[\begin{array}{ll}
      r+s\sqrt{-1} & t+u\sqrt{-1}\\
      -t+u\sqrt{-1} & r-s\sqrt{-1}\end{array}\right]\mapsto
  r-ui-tj-sk ,
  \]
  where $r^2+s^2+t^2+u^2=1$, is an isomorphism.
\end{proposition}
The map in Proposition~\ref{standard} restricts to an isomorphism
\[
\Psi_n:\SUT\big(\Z_{K_n}^{(2)}\big)\stackrel{\simeq}{\longrightarrow}
\widetilde{\M}_{n,1}^\times ,
\]
with an induced isomorphism
\begin{equation}
  \label{iso}
  \overline{\Psi}_n:\PSUT\big(\Z_{K_n}^{(2)}\big)=
\SUT\big(\Z_{K_n}^{(2)}\big)/\langle\pm 1\rangle
  \stackrel{\simeq}{\longrightarrow}\Pp\!\widetilde{\M}_{n,1}^\times:=
  \widetilde{\M}^\times_{n,1}:=\widetilde{\M}_{n,1}^\times/
  \langle \pm 1\rangle .
  \end{equation}

We now ask whether there is an isomorphism for $\PUT$ analogous
to the isomorphism \eqref{iso} for $\PSUT$.
\begin{definition1}
  \label{class}
        {\rm
          For $A\in\UT(K_n)$, denote by $[A]$ its class in $\PUT(K_n)$.  Similarly
          for $a\in \H_n^\times$, denote by $[a]$ its class in $\Pp\!\H_n^\times$.
        }
\end{definition1}
\noindent Suppose $A\in\UT(K_n)$ and $\alpha=\det(A)$; then $\alpha\overline{\alpha}=1$.
By Hilbert's Theorem 90 we have $\alpha=\overline{\beta}/\beta$ with
$\beta\in K_n^\times$.  Consider $A'=\beta A$.  We have
$\det A'=\beta^2\alpha=\beta\overline{\beta}\in F_n$.
Hence $A'$ is of the form
\[
A'=\left[\begin{array}{rr}
    r+s\sqrt{-1} & t+u\sqrt{-1}\\
    -t+u\sqrt{-1} & r-s\sqrt{-1}\end{array}\right].
\]
We then define  for $[A]\in \PUT\big(\Z_{K_n}^{(2)}\big)$
\begin{equation}
  \label{var}
\varphi_n([A])=[r-ui-tj-sk]\in \Pp\!\H_n^\times .
\end{equation}
  Note that on $\PSUT\big(\Z_{K_n}^{(2)}\big)$ our map $\varphi_n$ agrees with
  $\overline{\Psi}_n$ and gives an injection 
\[
\PUT\big(\Z_{K_n}^{(2)}\big)
\hookrightarrow
\Pro\!\H_n^\times .
\]
  \begin{theorem}
    \label{U}
    Suppose $n=2^s$ or $n=3\cdot 2^s$, $n\geq 8$.  Then the map $\varphi_n$
    defined in \eqref{var} gives an isomorphism
    \[
    \varphi_n:\PUT\big(\Z_{K_n}^{(2)}\big)\stackrel{\simeq}
{\longrightarrow}\Pp\!\widetilde{\M}_n^\times 
=\Gamma_{0}(\widetilde{\M}_n)=\Gamma_{0}\subseteq \Pp\!\H_n^\times.
    \]
  \end{theorem}
  \begin{proof}
    1. $\varphi_{n}\big(\Pp\!\UT\big(\Z_{K_n}^{(2)}\big)\big)
\subseteq \Gamma_{0}$.\\
    Let $A\in\UT\big(\Z_{K_n}^{(2)}\big)$. Then 
$\alpha=\det A\in \Z_{K_n,1}^{(2),\times}
=\mu_n$
    by Proposition~\ref{notsquares}\eqref{notsquares1},
    say $\alpha=\zeta^k$ for $\zeta=\zeta_n$.
    But then we can take $\beta=(1+\zeta)^k$ to give
    $\alpha=\beta/\overline{\beta}$.  If $n=2^s$, then $\Nm_{K_n/\Q}(1+\zeta)=
    2$.  If $n=3\cdot 2^s$, then $\Nm_{K_n/\Q}(1+\zeta) =1$ by
    Lemma~\ref{norms}.
    In either case we see that $\beta=(1+\zeta)^k\in \Z_{K_n}^{(2),\times}$.
    So $A'=\beta A\in\UT\big(\Z_{K_n}^{(2)}\big)$ and so from \eqref{var} we have
    $\varphi_{n}([A])\in \Pp\!\widetilde{\M}_n^\times=\Gamma_{0}$.\\
    2. $\varphi_{n}\big(\Pp\!\UT\big(\Z_{K_n}^{(2)}\big)\big)=\Gamma_{0}$.\\
      From the diagram (in which $1$'s on the top and bottom are omitted)
  $$\begin{CD}
    \PSUT\big(\Z_{K_n}^{(2)}\big) & @>{\simeq}>> & \Gamma_{1}\\
  @VVV &  & @VVV  \\
  \PUT\big(\Z_{K_n}^{(2)}\big) & @>{\varphi_n}>> & \Gamma_{0}\\
  @VVV & & @VVV   \\
  \PUT(\Z_{K_n}^{(2)})/\PSUT(\Z_{K_n}^{(2)})\simeq \Z_{K_n,1}^{(2),\times}/
  \big(\Z_{K_n,1}^{(2),\times}\big)^2 & @>>> & \Gamma_{0}/\Gamma_{1}
  \simeq \Z_{F_n,+}^{(2),\times}/\big(\Z_{F_n}^{(2),\times}\big)^2& ,
  \end{CD}$$

  \noindent we see that $\varphi_n$ is an isomorphism if and only if
  \begin{equation}
    \label{ream}
   \frac{\Z_{K_n,1}^{(2),\times}}{\big(\Z_{K_n,1}^{(2),\times}\big)^2} \stackrel{\simeq}
      {\longrightarrow} \frac{\Z_{F_n,+}^{(2),\times}}
{\big(\Z_{F_n}^{(2),\times}\big)^2}
  \end{equation}
  is an isomorphism.  By Theorem~\ref{fields},
the right-hand side of  \eqref{ream} is isomorphic
to $\Z/2\Z$.
We have $\Z_{K_n,1}^{(2),\times}/\big(\Z_{K_n,1}^{(2),\times}\big)^2 
\cong \mu_n/\mu_n^2\cong \Z/2\Z$
by Proposition~\ref{notsquares}.
The map in (\ref{ream}) is induced by $[\zeta]\mapsto [(1+\zeta)
  (1+\overline{\zeta})]$ where $\zeta:=\zeta_n$.
But this is an isomorphism by Lemma~\ref{notsquares}\eqref{notsquares2}.
  \end{proof}
  \begin{remark1}
    \label{acts}
          {\rm
            By Theorem~\ref{U}, for $n=2^s$ or $n=3\cdot 2^s$, $n\geq 8$,
           \[
           \PUT\big(\Z_{K_n}^{(2)}\big)\cong \Pp\!\widetilde{\M}_n^\times 
=\Gamma_{0}\big(\widetilde{\M}_n\big)=\Gamma_{0}\subseteq \Pp\!\H_n^\times
           \]
           acts on the Bruhat-Tits tree $\Delta$, and likewise for
                   \[
                   \PSUT\big(\Z_{K_n}^{(2)}\big)\cong \Pp\!\widetilde{\M}_{n,1}^\times =\Gamma_{1}\big(\widetilde{\M}_n\big)=\Gamma_{1}\subseteq \Pp\!\H_n^\times.
                   \]
                   Hence in the notation of Definition~\ref{grr}
                   \begin{equation}
                     \label{grrr}
                     \PSUT\big(\Z_{K_n}^{(2)}\big)\backslash\Delta
=\Gamma_{1}\backslash
                     \Delta=\gr_n\quad
                   \mbox{ and }\quad
                  \PUT\big(\Z_{K_n}^{(2)}\big)\backslash\Delta
=\Gamma_{0}\backslash\Delta=
                    \overline{\gr}_n
                   \end{equation}
          }
\end{remark1}

We can use this remark to bound the mass $m$ as in Definition \ref{onion}
of vertices and edges in $\gr_n$ and $\ogr_n$.
The following is classical.
Note that if $m \notin \{3,4,6\}$ then $K_m^+ \subset K_n$ implies
that $K_m \subset K_n$ and hence that $m|\lcm(2,n)$.

\begin{lemma}\label{lem:psu-pu-order}
  Let $m,n$ be positive integers and let $K_n = \Q(\zeta_n)$ as usual.
  \begin{enumerate}[\upshape (a)]
  \item 
\label{lem:psu-pu-order1}
If $\PSUT(K_n)$ contains an element of order $m$, then
    $K_{2m}^+ \subset K_n$.
  \item 
\label{lem:psu-pu-order2}
If $\PUT(K_n)$ contains an element of order $m$, then
    $K_m^+ \subset K_n$.
  \end{enumerate}
\end{lemma}

\begin{proof} 
\eqref{lem:psu-pu-order1}:  Let $M$ be such an element and 
lift it to $M'$ in $\SUT(K_n)$.
  Since $M$ has order $m$, the order of $M'$ must be $m$ or $2m$.
  Thus the eigenvalues of $M'$ must be $\zeta_{2m}^i, \zeta_{2m}^{-i}$ for some
  $i$ with $\gcd(i,2m) \le 2$.  If $m$ is even, then $i$ must be odd,
  for otherwise the eigenvalues raised to the $m/2$ power would be $-1$
  and $M$ would have order dividing $m/2$.  Since the trace of
  $M'$ is contained in $K_n$, we find that $K_{2m}^+ \subset K_n$.
  If $m$ is odd, we do not know whether the eigenvalues are of order
  $m$ or $2m$.  In the first case we conclude that $K_m^+ \subset K_n$, and
  in the second that $K_{2m}^+ \subset K_n$; but $K_m^+ = K_{2m}^+$ if $m$
  is odd, so the result follows.\\
\eqref{lem:psu-pu-order2}: This is part of  \cite[Prop.~1.1]{beauville}.

\end{proof}
  
\begin{proposition}
\label{none}
Let $n=2^s$ or $n=3\cdot 2^s$.
\begin{align*}
S(n)&=\{12, 24, 60\}\cup \{m\mid m\in\N\text{ and } 2m|n\} \text{ and}\\
S'(n)&=\{12, 24, 60\}\cup \{m\mid m\in\N\text{ and } m|n\}.
\end{align*}
\begin{enumerate}[\upshape (a)]
\item
\label{none1}
Suppose $\mathbf{v}\in \Ver(\gr_n)$ and $\mathbf{e}\in \Ed(\gr_n)=\Ed_r(\gr_n)$.
Then $1/m(\mathbf{v}), 1/m(\mathbf{e})\in S(n)$.
In particular, if $n\geq 30$, then $1\geq 
m(\mathbf{v}),\, m(\mathbf{e})\geq 1/2n$ 
and we have
\begin{align*}
\VM(\gr_n)\leq & v(\gr_n)\leq 2n\VM(\gr_n),\\
\EM(\gr_n)\leq e_r(\gr_n)& =e(\gr_n)\leq 2n \EM(\gr_n).
\end{align*}
\item
\label{none2}
Suppose $\mathbf{v}\in \Ver(\ogr_n)$ and 
$\mathbf{e}\in \Ed_r(\ogr_n)$. Then 
 $1/m(\mathbf{v}), 1/m(\mathbf{e})\in S'(n)$.
In particular, if $n\geq 30$, then
$1\geq m(\mathbf{v}),\, m(\mathbf{e})\geq 1/2n$ and we have
\begin{align*}
\VM(\ogr_n)\leq & v(\ogr_n)\leq 2n\VM(\ogr_n),\\
\EM(\ogr_{\Th})\leq e_r(\ogr_{\Th})& =e(\ogr_{\Th})\leq 2n \EM(\ogr_{\Th}),\\
\EM(\ogr_{2^s})-e_h(\ogr_{2^s})\leq & e_r(\ogr_{2^s})\leq 2n\EM(\ogr_{2^s}).
\end{align*}

\end{enumerate}
\end{proposition}
\begin{proof}
\eqref{none1}: The stabilizer group of $\mathbf{v}\in \Ver(\gr_n)$
or $\mathbf{e}\in\Ed(\gr_n)$ is a finite subgroup of 
$\PSUT\big(\Z_{K_n}^{(2)}\big)$.
There is a well-known classification of 
 finite subgroups of $\PSUT(\BC)\cong \SOT$ (\cite[Thm.~3.6, Chap.~I]{V}):
 they are cyclic $\Z/m\Z$, dihedral $D_{m}$ of order $2m$,
 or one of $A_4$, $S_4$, or $A_5$.
 By Lemma \ref{lem:psu-pu-order}, it follows that no finite subgroup can have
 order greater than $2n$ if $n>30$.
 Note that $e_h(\gr_n)=0$ by Remark \ref{dale}.\\
\eqref{none2}:  The stabilizer group of $\mathbf{v}\in \Ver(\ogr_n)$
or $\mathbf{e}\in\Ed_r(\ogr_n)$ is a finite subgroup of 
$\PUT\big(\Z_{K_n}^{(2)}\big)$. 
Such groups are certainly subgroups of $\PGLT(\BC)$, which are classified
in \cite[Prop.~16]{S4}: the answer is the same as for $\PSUT(\BC)$.
Again, combining Lemma \ref{lem:psu-pu-order}\eqref{lem:psu-pu-order2}
(\cite[Prop.~1.1]{beauville}) with our remark preceding
Lemma \ref{lem:psu-pu-order} proves that
if $\Z/m\Z$ or $D_m$ is a subgroup of $\PUT(\Z_{K_n}^{(2)})$, then
$m|n$.  Hence for $n\geq 30$ the order
of a finite subgroup of $\PUT]\big(\Z_{K_n}^{(2)}\big)$ containing $\Z/m\Z$
or $D_m$ is at most $2n$. 
Note that $e_h(\ogr_{\Th})=0$ by Theorem
\ref{thm:invert}\eqref{thm:invert1}.
\end{proof}


  \subsection
{\texorpdfstring{Vertices of
    \except{toc}{\boldmath{$\overline{\gr}_n$}}\for{toc}{
$\overline{\gr}_n$}
    and maximal orders in \except{toc}{\boldmath{$\H_{n}$}}\for{toc}
{$\H_n$}}{Vertices of grn and maximal orders in Hn}}
\label{maxorders}

  \begin{definition1}
    \label{def:class}
Fix $n$. By an {\sf order} $\mN$ we will mean an
($\Z_F=\Z_{F_n}$)-order in $\H=\H_n$.
Denote by $\Cl(\mN)$ the set of left ideal classes of $\mN$. For a left $\mN$-ideal $I$, denote by $[I]$ its class in $\Cl(\mN)$
and by $\RO(I)$ its right order.
For a $\Z_F$-ideal $\fa$, denote by $[\fa]$ its class in $\Cl(\Z_F)$.

The class group
$\Cl(\Z_F)$ acts on $\Cl(\mN)$ by
$[\mathfrak{a}]\cdot [I] := [I\frak{a}] = [\mathfrak{a}I]$ for a
$\Z_F$-ideal $\mathfrak{a}$ and a left $\mN$-ideal $I$.
Define the {\sf relative} class set $\Cl_{\rm rel}(\mN)$
to be the set of orbits of $\Cl(\Z_F)$ on $\Cl(\mN)$:
\begin{equation*}
  \Cl_{\rm rel}(\mN)=\Cl(\mN)/\Cl(\Z_F)
  \end{equation*}
with {\rm relative} class number $h_{\rm rel}(\mN)=\#\Cl_{\rm rel}(\mN)$.
The relative class number $h_{\rm rel}(\M')$ is the same for all maximal
orders $\M'\subseteq \H_n$, so we write simply $h_{\rm rel}(\H_n)$.

\end{definition1}

  Let $\T(\H_n)$ be the set of {\sf types}, i.e., isomorphism classes
  of maximal orders, of $\H_n$.  If $\M$ is a maximal order, let
  $\Ty(\M)$ be its type. 
The explicit maximal order $\M=\M_n\subseteq\H=\H_n$ was given in
  \eqref{explicit}.
\begin{theorem}
  \label{thm:Mn}
  Let $n = 2^s$ or $3\cdot 2^s$, $n\geq 8$, and set $\M \colonequals \M_n$.  
The class group $\Cl (\Z_F)$ acts freely on the set $\Cl(\M)$.  The map
  $I\mapsto \Ty(\RO(I))$ is an $h(\Z_F)$-to-$1$ map from $\Cl(\M)$
  onto $\T(\H)$ which is constant on the $\Cl(\Z_F)$-orbits.
  So $\Cl_{\rm rel}(\M)$ may be identified with $\T(\H)$ and
  $\#\T(\H)=h_{\rm rel}(\H)$.

Let $\bv=\M'\in\Ver(\Delta)$ be a maximal order of $\H$ as in \eqref{ver}.
Let $\overline{\bv}=[\M']\in\Ver(\Gamma_{0}\backslash\Delta)$
be the vertex of $\overline{\gr}:=\overline{\gr}_n = \Gamma_{0}\backslash \Delta$ below
$\bv$.
  The map $\overline{\bv}\mapsto \Ty(\M')$ is a bijection between
  $\Ver(\overline{\gr})$ and $\T(\H)$.
  Hence $\Ver(\overline{\gr})$ can be identified with $\Cl_{\rm rel}(\M)$
  and $\#\Ver(\overline{\gr})=h_{\rm rel}(\H)$.
\end{theorem}
\begin{proof}
  We will prove Theorem~\ref{thm:Mn} through a series of propositions.
  First note that by the Skolem-Noether Theorem \cite[Thm.~I.2.1]{V}, 
  maximal isomorphic orders of $\H=\H_n$
  are  conjugate. For an element $\gamma\in \H^\times$ and
  a maximal order $\M'$ in $\H$, put $(\M')^\gamma = \gamma\M'\gamma^{-1}$.
  If $h\in\H^\times$ and $\fq$ is a prime of $F$, $h_{\fq}$ denotes $h$
  viewed as an element of $\H_{\fq}$.  Similarly, if $x\in F^\times$,
  $x_{\fq}$ denotes $x$ viewed as an element of $F_{\fq}^\times$.
  Our first proposition shows that every isomorphism class of maximal
  orders in $\H$ is represented by vertices in the tree $\Delta$ as in
  \eqref{ver}:
  \begin{proposition}
    \label{max}
Let $n=2^s$ or $n=3\cdot 2^s$, $n\geq 8$.
Suppose $\M'\subseteq \H:=\H_n$ is a maximal order.  Then there exists
    a maximal order $\M''\subseteq\H$ such that $\M''\cong \M'$ and
    $\M''_{\fq}=\M_{\fq}$ for all $\fq\neq \fp$.
    \end{proposition}
  \begin{proof}
    For $\fq\neq\fp$, $\M'_{\fq}\simeq\M_{\fq}$ since any two maximal orders are
    locally isomorphic.  Hence $\M'_{\fq}=\gamma_{\fq}\M_{\fq}\gamma_{\fq}^{-1}$
    for $\gamma_{\fq}\in \H_{\fq}^\times$.  We can take $\gamma_{\fq}=1$ for all
    but finitely many $\fq$ since two orders differ at only
    finitely many places.  Consider the ideal
    $I=\prod_{\fq\neq \fp}\fq^{\val_{\fq}(\Nm_{\H_{\fq}/F_{\fq}}(\gamma_{\fq}))}$ of
    $\Z_F^{(2)}=\Z_{F_n}^{(2)}$.
    Since the narrow class number
    of $\Z_{F}^{(2)}$ is odd by Theorem~\ref{fields}, 
the class of $I$ is a square
    in the narrow class group $\Cl^+\big(\Z_{F}^{(2)}\big)$.
    Hence there exists $\alpha\in \Z_{F,+}^{(2),\times}$
    such that $\alpha I=J^2$ for some ideal $J$ of $\Z_F^{(2)}$.
    Let $\beta\in\H^\times$ with $\Nm_{\H/F}(\beta)=\alpha$ (which exists by
    Lemma~\ref{ideal}).
    Consider the
    maximal order $\M(\beta):=\beta\M'\beta^{-1}$.
    We have
    \[
    \M(\beta)_{\fq}=\beta\gamma_{\fq}\M_{\fq}\gamma_{\fq}^{-1}\beta^{-1}=
    (\beta\gamma_{\fq})\M_{\fq}(\beta\gamma_{\fq})^{-1}.
    \]
    Notice that $\val_{\fq}\Nm_{\H_{\fq}/F_{\fq}}(\beta\gamma_{\fq})$ is even,
    say equal to $2t$.  Suppose  $F_{\fq}$ has uniformizer $\pi_{\fq}$.  Then
    $\pi_{\fq}^{-t}\M_{\fq}\pi_{\fq}^t=\M_{\fq}$ since $\pi_{\fq}$ is in
    the center $F_{\fq}$ of $\H_{\fq}$.  So
    \begin{equation*}
      \M(\beta)_{\fq}=\beta\gamma_{\fq}\pi_{\fq}^{-t}\M_{\fq}\pi_{\fq}^t
      \gamma_{\fq}^{-1}\beta^{-1}.
      \end{equation*}
    Notice that $\Nm_{\H_{\fq}/F_{\fq}}(\beta\gamma_{\fq}\pi_{\fq}^{-t})\in 
\Z_{F_\fq}^\times$
    as its $\fq$-valuation is $0$.  Hence there exists
    $\eta_{\fq}\in\M_{\fq}^{\times}$ with 
\[
    \Nm_{\H_{\fq}/F_{\fq}}(\eta_{\fq})=
    \Nm_{\H_{\fq}/F_{\fq}}(\beta\gamma_{\fq}\pi_{\fq}^{-t}).
    \]
    But then $\tau_{\fq}:=\beta\gamma_{\fq}\pi_{\fq}^{-t}\eta_{\fq}^{-1}$
    has norm $1$, and $\M_{\fq}^{\tau_{\fq}}=\M(\beta)_{\fq}$.
So by Strong Approximation there exists $\tau\in\H_{1}^{\times}$ such
    that $\tau=\tau_{\fq}u_{\fq}$ with
    $u_{\fq}\in\M_{\fq}^\times$ for all $\fq\neq \fp$.
    Now $\M(\beta)_{\fq}=\tau\M_{\fq}\tau^{-1}$.  
We see that $\M'' =\tau^{-1}\M(\beta)\tau$
    satisfies the statement of the proposition.
  \end{proof}

  A maximal order $\M'$ belongs to $\Ver(\Delta)$ when $\M'_{\fq}=\M_{\fq}$
  for all primes $\fq\neq\fp$ of $\Z_F$.  The group $\Gamma_{0}=
  \Pp\!\widetilde{\M}_{n}^\times$ acts on $\Delta$ by
  \begin{equation*}
    \gamma\cdot\M'=(\M')^\gamma =\gamma\M'\gamma^{-1}\quad
    \mbox{ for }\gamma\in\Gamma_{0},\,\M'\in\Ver(\Delta).
    \end{equation*}
  Let $[\M'] \in \Ver(\Gamma_0 \backslash \Delta)$ be the image of 
  $\M' \in \Ver(\Delta)$ by the natural quotient map.
  \begin{proposition}
    \label{maxo}
Let $n=2^s$ or $n=3\cdot 2^s$, $n\geq 8$.
    Suppose $\M', \M''\in\Ver(\Delta)$ are maximal orders.
    Then $\M'\cong\M''$ if and only if
\[
    [\M']=[\M'']\in\Ver(\Gamma_{0}\backslash
    \Delta=\overline{\gr}_n).
    \]
    Hence $\Ver(\overline{\gr}_n)$ is the set of  isomorphism classes
    of maximal orders in $\H=\H_n$.
  \end{proposition}
  \begin{proof}
    If $[\M']=[\M'']$, then there exists $\gamma\in\M[1/2]^\times
    =\widetilde{\M}^\times$
    so that $(\M')^\gamma=\gamma\M'\gamma^{-1}=\M''$, so  $\M'\cong\M''$.

    Now suppose $M'\cong\M''$.  Then $\M''=(\M')^\gamma=\gamma\M'\gamma^{-1}$
    for some $\gamma\in\H^\times$.  So for each $\fq\neq\fp$,
    $\gamma_{\fq}=s_{\fq}u_{\fq}$ with $s_{\fq}\in F_{\fq}^\times$,
    $u_{\fq}\in\M_{\fq}^\times$.  In particular,
    $\val_{\fq}(\Nm_{\H_{\fq}/F_{\fq}}
    (\gamma_{\fq}))$ is even and hence the principal ideal
    $(\Nm_{\H/F}(\gamma))$ of $\Z_{F_n}^{(2)}$ is the square of an ideal $J$:
    $(\Nm_{\H/F}(\gamma))=J^2$. The ideal $J$ is principal since $J^2$
    is principal and $h(\Z_{F_n}^{(2)})$ is odd by Theorem~\ref{fields}:
    \begin{equation}
      \label{sun}
      \Nm_{\H/F}(\gamma)=uj^2\quad\mbox{ with }\quad u\in \Z_{F_n}^{(2),\times},
      \,j\in F^\times.
      \end{equation}
    Let $\gamma':=j^{-1}\gamma$; then $\gamma'_{\fq}=j^{-1}s_{\fq}u_{\fq}$.

    We claim that $j^{-1}s_{\fq}\in\M_{\fq}^\times$: We have
    \[
    \Nm_{\H_{\fq}/F_{\fq}}(\gamma_{\fq})=s_{\fq}^2\Nm_{\H_{\fq}/F_{\fq}}(u_{\fq})=
    uj^2
    \]
    using \eqref{sun}. Taking $\val_{\fq}$ we get
    $2\val_{\fq}(s_{\fq})=2\val_{\fq}(j)$.  Hence $\val_{\fq}(j^{-1}s_{\fq})=0$
    and $j^{-1}s_{\fq}\in\M_{\fq}^\times$.  So $\gamma'_{\fq}\in\M_{\fq}^\times$
    for all $\fq\neq\fp$.  Hence $\gamma'\in\M[1/2]^{\times}$
    with 
\[
\M''=(\M')^{\gamma'}\quad\text{and}\quad[\M']=[\M'']\in\Ver(\Gamma_{0}\backslash
    \Delta).
\]
  \end{proof}
    \begin{lemma}
      \label{done}
      Let $\mN$ be a maximal $\Z_{F_n}$-order of $\H_n$ with $4|n$ and $n\geq 8$.
      Let $I$ be a two-sided $\mN$-ideal in $\H_n$.  Then $I = \mN\fa$
      for some ideal $\fa$ of $\Z_{F_n}$.
    \end{lemma}
    \begin{proof}
      First note that if $J$ is a $\Z_F\colonequals\Z_{F_n}$-ideal then
      $J=\cap_{\fq} (J\otimes_{\Z_F}\Z_{F_\fq})$, see
      \cite[Lemma 9.4.3]{voight}.
Now the only
$2$-sided ideals in matrix rings are generated by scalars
\cite[Thm.~3.1]{Lam2}
and $\H_n$ is unramified at all finite primes $\fq$ \cite[Prop.~3.1]{IJKLZ2}.
      Hence $I\otimes_{\Z_F}\Z_{F_\fq}=\mN_{\fq}S_{\fq}$ with $S_{\fq}$
      a scalar ideal for all primes $\fq$.  But then
      \begin{equation*}
        I = \cap_{\fq}\mN_{\fq}S_{\fq}
        = \mN\left(\cap_{\fq}S_{\fq}\right)
        =\mN\fa ,
      \end{equation*}
      where the $\Z_F$-ideal $\fa$ is given by $\fa=\cap_{\fq}S_{\fq}$.
    \end{proof}
  \begin{proposition}
    \label{freee}
    Let $n=2^s$ or $n=3\cdot 2^s$, $n\geq 8$.
    The action of \mbox{$\Cl(\Z_F=\Z_{F_n})$} on $\Cl(\M=\M_n)$ is free.
    Let $I_1, I_2, \ldots, I_h$ be representatives for
    $\Cl(\M)$ with right orders 
    $\M_1, \M_2,  \ldots , \M_{h}$, respectively,
    with $h=h(\H=\H_n)$.  Then $\M_i\cong\M_j$ if and only if
    $[I_i]=[I_j\fa]$ for some ideal $\fa$ of $\Z_F$, $1\leq i,j\leq h$.
    Hence the set $\T(\H)$ of isomorphism classes of maximal orders in $\H$
    is in natural bijection with $\Cl_{\rm rel}(\M)$.
  \end{proposition}
  \begin{proof}
    If $[I_i]=[I_j\fa]$ for a $\Z_F$-ideal $\fa$, then obviously $\M_j$
    is the right order of $I_j\fa$ and the maximal orders $\M_i, \,\M_j$
    are isomorphic. To see that $\Cl(\Z_F)$ acts freely
    on $\Cl(M)$, suppose $I_i\fa=I_i\alpha$ for $\alpha\in\H^\times$.
    Then taking norms of ideals
    we have $\Nm(I_i)\Nm(\alpha)=\Nm(I_i)\Nm(\fa)=\fa^2\Nm(I_i)$. Hence
    $(\Nm(\alpha))=\fa^2$ and $\fa^2$ is principal.  But the class number
    $h(\Z_F)$ is odd by Theorem~\ref{fields}, and so $\fa$ is principal.

    Now suppose that $\M_i \cong \M_j$, so that $\M_i = \M_j^\alpha$ for some
    $\alpha$.  Then $I_i^{-1}I_j\alpha^{-1}\subseteq \H_n$
    is a two-sided $\M_i$-ideal. 
    Now by Lemma~\ref{done}, we have $I_i^{-1}I_j\alpha^{-1}=\M\fa$ and so
    $[I_i]=[I_j\fa^{-1}]$, concluding the proof of Proposition~\ref{freee}.
        \end{proof}

  Putting Propositions~\ref{max}, \ref{maxo}, \ref{freee} together then
  proves Theorem~\ref{thm:Mn}.

\end{proof}

\section{Optimal embeddings, mass, and a torsion invariant}
\label{puppy}
In this section we fix $n \ge 8$ of the form $2^s$ or $3 \cdot 2^s$.
We then frequently omit subscript $n$'s; in particular, $\Z_F\colonequals
\Z_{F_n}$
and $\H\colonequals\H_n$.

\subsection{Mass}
\label{puppy1}

\begin{definition1}
The {\sf mass} of an $\uOO$-order $\mN\subseteq \H:=\H_n$, denoted
$m(\mN)$, is defined to be $$m(\mN):=\frac{1}{\#[\mN^\times:
    \Z_F^\times]}\, .$$
Let $\mN\subseteq \H_n$ be a maximal order. Let 
$I_1, I_2, \ldots,  I_h$ be representatives for the left ideal classes
of $\mN$ with right orders 
$\mN_1, \ldots, \mN_{h}$, respectively.
The {\sf Eichler mass} of $\H$, denoted $m(\H)$, is 
$m(\H)= \sum_{i= 1}^h m(\mN_i)$.
The Eichler mass $m(\H)$ is independent of the maximal order
$\mN\subseteq \H$ used to define it.
\end{definition1}

Note that 
\begin{equation}
\label{smelt}
[F_n:\Q]=\begin{cases} 2^{s-2}\quad\text{if $n=2^s$}\\
2^{s-1}\quad\text{if $n=3\cdot 2^s$}.
\end{cases}
\end{equation}
Eichler's mass formula \cite[Cor.~V.2.3]{V}
for the totally definite quaternion algebra
$\H_n$ gives:
\begin{theorem}
  \label{mass_g}
$m(\H_n)  = 2^{1-[F_n:\Q]}|\zeta_{F_n}(-1)|
h(\Z_{F_n})$ .
\end{theorem}

\begin{definition1}
  \label{masss}
For $n=2^s$ or $n=3\cdot 2^s$,
set 
\begin{equation}
\label{stilts}
M_n\colonequals 2^{1-[F_n:\Q]}|\zeta_{F_n}(-1)|
= \begin{cases}2^{1-2^{s-2}}|\zeta_{F_{n}}(-1)|\quad
  \text{ if }n = 2^s \\  2^{1-2^{s-1}}|\zeta_{F_{n}}(-1)|\quad
  \text{ if } n = 3\cdot 2^s\end{cases}\, .
\end{equation}
  Then $m(\H_n)= M_n h(\Z_{F_n})$.

We can rewrite the formula for $M_n$ using the functional equation
\cite[Thm.~7.3]{N} to obtain
\begin{equation}
\label{pros}
|\zeta_{F_n}(-1)| = \zeta_{F_n}(2)\Disc(F_n)^{3/2}
(2\pi^2)^{-[F_n:\Q]} .
\end{equation}
By \cite[Prop.~2.7, Lemma 4.19]{W}, 
\begin{equation}
\label{salty}
\Disc(F_n)=\begin{cases}  2^{(s-1) 2^{s-2} - 1}\quad\text{if $n=2^s$}\\
2^{(s-1)2^{s-1}} 3^{2^{s-2}}\quad\text{if $n=3\cdot 2^s$}.
\end{cases}
\end{equation}
Substituting \eqref{smelt}  and \eqref{salty} into \eqref{pros} and using
\eqref{stilts}, we obtain
\begin{equation}
\label{stilts2}
M_n =\begin{cases}
\left(2^{(3s-7)\cdot 2^{s-3}-1/2}\right) \pi^{-2^{s-1}} \zeta_{F_n}(2)=
2^{(3s-7-4\log_2\pi)2^{s-3}-1/2}\zeta_{F_n}(2)\quad\text{if $n=2^s$}\\
\left(2^{(3s-7)\cdot 2^{s-2}+1}\right) 3^{3 \cdot 2^{s-3}} \pi^{-2^s} \zeta_{F_n}(2)\\
\qquad\qquad\qquad\qquad =
2^{1+(6s-14+3\log_2 3-8\log_2\pi)2^{s-3}}\zeta_{F_n}(2)
\quad\text{if $n=3\cdot 2^s$}.
\end{cases}
\end{equation}
\end{definition1}

We give the first six values of $M_n$ in Figures 1 and 2 below.

It will be important for us that we can bound $M_n$.  A first
step is to bound $\zeta_{F_n}(2)$.

\begin{figure}[ht]
\begin{center}
\label{table1}
\begin{tabular}{l|l}
$s$ & $M_{2^s}$\\
\hline 
$3$ & $1/24$\\
$4$ & $5/48$\\
$5$ & $1455/32 = 45 +15/32$\\ 
$6$ & 
$36089088885 + 15/64$\\
$7$ & 
$388002862466852235269386970423082656 + 47/128$\\
$8$ & 
$12590881347932996758326842620820139275560721327409501583885395489575954$\\
 & \qquad $869483072245958812094894349923+ 47/256$
\end{tabular}
\end{center}
\caption{The mass $M_n$ for $n=2^s$, $3\leq s\leq 8$.}
\end{figure}
\begin{figure}
\begin{center}
\label{table2}
\begin{tabular}{l|l}
$s$ & $M_{\Th}$\\
\hline 
$2$ & $1/12$\\
$3$ & $1/8$\\
$4$ & 
$22 + 13/16$\\ 
$5$ & 
$3235132981 + 13/32$\\
$6$ & 
$1100354993054815942169953280797061 + 61/64  $\\
$7$ & 
$358377997932455414885081307668626153899866730607775353622534814156729$\\
 & \qquad $ 75427396299786979146454217+ 29/128$
\end{tabular}
\end{center}
\caption{The mass $M_n$ for $n=3\cdot2^s$, $2\leq s\leq 7$.}
\end{figure}
\vspace*{.3in}

\begin{lemma}
\label{padthai}
We have
\[
1<\zeta_{F_n}(2)\leq \zeta_{\Q}(2)^{[F_n:\Q]}=(\pi^2/6)^{[F_n:\Q]} =
\begin{cases}
2^{(2\log_2\pi-1-\log_2 3)2^{s-2}}<2^{(1.5) 2^{s-3}}\quad\text{if $n=2^s$}\\
2^{(2\log_2\pi-1-\log_2 3)2^{s-1}}<2^{(3) 2^{s-3}}
\quad\text{if $n=3\cdot 2^s$.}

\end{cases}
\]
\end{lemma}
\begin{proof}
For a number field $F$, let $|F|$ be the set of primes of $F$.
Let $p\in |\Q|$ factor in $F_n$ as $p\uOO =\wp_1^e\cdots \wp_r^e$ with
$\wp_i$ having residue field $\F_{p^f}$.  Then $\zeta_{F_n}(s)$ is given
by an Euler product
\begin{equation*}
\label{frost}
\zeta_{F_n}(s)=\prod_{p\in |\Q|}E_p(s)\quad\text{where}\quad
E_p(s)=(1-p^{-fs})^{-r}.
\end{equation*}
Note that 
\[
1<E_p(2)=(1-p^{-2f})^{-r}\leq (1-p^{-2})^{-fr}\leq (1-p^{-2})^{-[F_n:\Q]}
\]
since 
\[
1-p^{-2f}\geq 1-p^{-2}\geq (1-p^{-2})^f .
\]
Taking the product over $p\in |\Q|$ then gives
$ 1<\zeta_{F_n}(2)\leq \zeta_{\Q}(2)^{[F_n:\Q]}$.
\end{proof}
\begin{theorem}
\label{soup}
\begin{enumerate}[\upshape (a)]
\item
\label{soup1}
$\, 2^{(3s-12.1)2^{s-3}-1/2}>M_{2^s} >2^{(3s-13.7)2^{s-3}-1/2}$.
\item
\label{soup2}
$\, 2^{(6s-19.4)2^{s-3}+1} >M_{3\cdot 2^s}>2^{(6s-22.5)2^{s-3}+1} $.
\end{enumerate}
\end{theorem}
\begin{proof}
Combine \eqref{stilts2} and Lemma \ref{padthai} with the inequalities
\[
13.6<7+4\log_2\pi<13.7\quad\text{and}\quad 22.4<14-3\log_23+8\log_2\pi<22.5 .
\qedhere
\]

\end{proof}


\begin{theorem}   
\label{masst}
  Let $n=2^s$ or $3\cdot 2^s$, $n\geq 8$,  and $\gr_n
=\PSUT\big(\Z_{K_n}^{(2)}\big)\backslash\Delta_n$, $\widebar{\gr}_n
=\PUT\big(\Z_{K_n}^{(2)}\big)\backslash\Delta_n$ as in \eqref{grrr} .
\begin{enumerate}[\upshape (a)]
    \item
\label{founded1}
We have 
\begin{eqnarray*}
  \VM(\gr_n) &= 2M_n,
&\EM(\gr_n) = 3M_n , \\
\VM(\widebar{\gr}_n) &= M_n,
&\EM(\widebar{\gr}_n)= (3/2)M_n .
\end{eqnarray*}
\item
\label{founded}
\textup{(\cite[Thm.~6.6(ii),(iv)]{IJKLZ})}
  We have the Euler-Poincar\'{e} characteristics 
  \[\chi\big(\PSUT\big(\Z_{K_n}^{(2)}\big)\big)= -M_n\quad\mbox{and}\quad
  \chi\big(\PUT\big(\Z_{K_n}^{(2)}\big)\big)=-(1/2)M_n.\]
\end{enumerate}
\end{theorem}
\begin{proof}
From Theorem~\ref{thm:Mn} and Theorem \ref{mass_g}
$$\VM(\widebar{\gr}_n) = \frac{m(\H_n)}{h(\Z_{F_n})} =  M_n\, .$$
As $\Delta_n$ is regular of valence
$3 = \Norm(\mathfrak{p}_n) +1$, 
the rest follows from Theorem~\ref{thm:three_to_two}
and Theorem~\ref{thm: mass_mult} applied to the double
cover $\pi: \gr
\rightarrow \widebar{\gr}$ of Definition~\ref{grr}.

Theorem~\ref{masst}(\ref{founded}) follows from Theorem~\ref{EP}.
\end{proof}

\begin{proposition}\label{prop:bounds}
Let $n=2^s$ or $3\cdot 2^s$, $n\geq 8$, $s\geq 2$,  with $\gr_n
=\PSUT\big(\Z_{K_n}^{(2)}\big)\backslash\Delta_n$  and $\widebar{\gr}_n
=\PUT\big(\Z_{K_n}^{(2)}\big)\backslash\Delta_n$ as in \eqref{grrr}.
\begin{enumerate}[\upshape (a)]
  \item
\label{bounds1}
Then
\begin{eqnarray*}
v_1(\gr_n) &\le 2M_n,  
&e(\gr_n) \ge 3M_n\, , \\
v_1(\widebar{\gr}_n) &\le M_n ,
&e_r(\overline{\gr}_n) \ge \frac{3}{2}M_n - e_h(\overline{\gr}_n),
\end{eqnarray*}
with $v_1, v_{<1}$ as in Definition \textup{\ref{defn: breakup}}
and $e_r$ as in Definition
\textup{\ref{gr}}.
\item
\label{bounds2}
  We have
\[
v(\gr_n)=v_{1}(\gr_n)+v_{<1}(\gr_n)\le 2M_n+v_{<1}(\gr_n).
\]
\end{enumerate}
\end{proposition}
\begin{proof}
By Lemma~\ref{action} $\PSUT\big(\Z_{K_n}^{(2)}\big)$ acts on $\Delta_n$ without
inversions, thus
$e_h(\gr) = 0$ and $e(\gr) = e_r(\gr)$. Apply Theorem~\ref{masst}, 
Lemma~\ref{lemma: lower bound e}, and the fact that the
mass of an edge or vertex is at most 1. 
\end{proof}

Although most of this paper is devoted to asymptotic {\em lower} bounds
for $b_1(\gr_n)$ and $b_1(\overline{\gr}_n)$, it is
much easier to give asymptotic {\em upper} bounds.

\begin{proposition}
  \label{upper}
  Let $n=2^s$ or $n=3\cdot 2^s$ with $n\geq 8$.
  Then $b_1(\gr_n)=2^{O(n\log n)}$ and $b_1(\overline{\gr}_n)
  =2^{O(n\log n)}$.
\end{proposition}
\begin{proof}
First observe that the stabilizer group for a vertex or noninverted
edge is a discrete subgroup of $\PUT(\BC)$.  By the well-known
classification of finite subgroups of $\PUT(\BC)$, it must be either
cyclic $\Z/m$, dihedral $D_m$, or one of $A_4$, $S_4$, or $A_5$.  Now,
we saw in Lemma \ref{lem:psu-pu-order} that to have either $\Z/m$ or
$D_m$ in $\PSUT\big(\Z_{K_n}^{(2)}\big)$ for $m > 6$ requires that $\Z_{K_n}^{(2)}$
contain the $2m^{th}$ roots of unity, i.e., that $2m$ divide $n$
(similarly for them to be in $\PUT\big(\Z_{K_n}^{(2)}\big)$ requires that 
$m$ divide
$n$); hence, for $n\ge60$ its order is at most $n$ for 
$\PSUT\big(\Z_{K_n}^{(2)}\big)$
and $2n$ for $\PUT\big(\Z_{K_n}^{(2)}\big)$. 
Hence by Proposition~\ref{none} and Theorem \ref{masst}\eqref{founded1} we have
$$b_1(\overline{\gr}_n)\leq b_1(\gr_n)\le
e_r(gr_n)\le 2n\EM(\gr_n)=(2n)(3M_n)=
6 n2^{1-[F_n:\Q]}\lvert\zeta_{F_n}(-1)\rvert\,\,
\text{for $n\geq 60$}.$$

Now by the functional equation 
\cite[Thm.~7.3]{N} we have
$$|\zeta_{F_n}(-1)| = \zeta_{F_n}(2)\Disc(\Z_{F_n})^{3/2}
(2\pi^2)^{-[F_n:\Q]} .$$

Notice that $\Disc(\Z_{F_n})=2^{O(n\log n)}$
whereas all the other factors involved are bounded by $2^{O(n)}$.  This
is clear for all factors except for $\zeta_{F_n}(2)$, but we showed
in the proof of  Lemma \ref{padthai} that
$\zeta_{F_n}(2) \le \zeta_{\Q}(2)^{n/4}$.  
Therefore, we get $b_1(\gr_n)=2^{O(n\log n)}$
and $b_1(\overline{\gr}_n)=2^{O(n\log n)}$ as desired.
\end{proof}

\subsection{A torsion invariant}
\label{fern}

In this section we define an invariant which measures how torsion in a group
acting on a contractible space affects the Euler characteristic of the quotient.
\begin{definition}
\label{cream}
For a group $\Gamma$ of finite homological type 
(cf.~\cite[Chap.~9, Sect.~6]{Br}),
$\chi(\Gamma)$ is the Euler-Poincar\'{e} characteristic as
in \cite[Chap.~IX, Sect.~7]{Br}.  
Suppose $\Gamma$ acts
on a contractible space $X$ with finite isotropy subgroups.
Define the invariant
\begin{equation*}
\label{git}
E(\Gamma, X)\colonequals\chi(\Gamma\backslash X)-\chi(\Gamma),
\end{equation*}
where $\chi(\Gamma\backslash X)$ is the Euler characteristic
of the quotient space $\Gamma\backslash X$.  If $\Gamma$ is torsion-free,
then $E(\Gamma,X)=0$.  However, if the action of $\Gamma$ on $X$
has nontrivial finite isotropy subgroups, then $E(\Gamma,X)$ can be 
nonzero---it measures the correction to the Euler characteristic of 
$\Gamma\backslash X$ due to the isotropy subgroups of the action.
\end{definition}

Let $n=2^s$ or $n=3\cdot 2^s$, $n\geq 8$, with $\mathfrak{p}_n$
the unique prime above $2$ in $F_n$.
The groups $\PSUT\big(\Z_{K_n}^{(2)}=\Z[\zeta_n, 1/2]\big)$ and 
$\PUT\big(\Z_{K_n}^{(2)}\big)$
act on the tree $\Delta\colonequals\Delta_{\mathfrak{p}_n}$.
It will be convenient to adopt the following abbreviated notation.
\begin{definition1}
\label{tail}
Set 
\[
E_{1,n}\colonequals E\big(\PSUT\big(\Z_{K_n}^{(2)}\big)  ,\Delta\big)
\qquad\text{and}\qquad
E_{0,n}\colonequals  E\big(\PUT\big(\Z_{K_n}^{(2)}\big)  ,\Delta\big).
\]
\end{definition1}

For the group $\PSUT\big(\Z_{K_n}^{(2)}\big)$ 
we have by Theorem \ref{masst}\eqref{founded}
\begin{align}
\label{shade2}
E_{1,n}&=\chi(\gr_{n})+M_{n}
=1-b_1(\gr_{n})+M_{n}\\
\nonumber
&=v(\gr_n)-e(\gr_n)+M_n
= v_1(\gr_n)+v_{<1}(\gr_n)-e(\gr_n)+M_n\\
\label{shade1}
&\leq 2M_n+v_{<1}(\gr_n)-3M_n+M_n=
v_{<1}(\gr_n)\text{ by Prop.~\ref{prop:bounds}}.
\end{align}

In the case of $E_{0,n}$ we break into
cases as to whether $3|n$.  First suppose $n=2^s$ with $s\geq 3$.
Since the double cover $\gr_{2^s}\rightarrow \ogr_{2^s}$ is
\'{e}tale by Theorem \ref{index}\eqref{index2},
\begin{equation}
\label{dealt}
2\chi(\ogr_{2^s})=\chi(\gr_{2^s})+e_h(\ogr_{2^s}).
\end{equation}
Combining \eqref{dealt} with the identity 
$2\chi\big(\PUT\big(\Z_{K_n}^{(2)}\big)\big)=
\chi\big(\PSUT\big(\Z_{K_n}^{(2)}\big)\big)=-M_{2^s}$ of 
Theorem \ref{masst}\eqref{founded}, we
obtain
\begin{align}
\label{dealt2}
E_{0,2^s}=& 1-b_1(\ogr_{2^s})+M_{2^s}/2\\
\label{dealt4}
&=( E_{1,2^s}+
e_h(\ogr_{2^s}))/2\\
&\leq (v_{<1}(\gr_{2^s})+e_h(\ogr_{2^s}))/2\text{ using \eqref{shade1}}.
\nonumber
\end{align}

If $n=3\cdot 2^s$ with $s\geq 2$, then $\ogr_{3\cdot 2^s}$ has no
half-edges by Theorem \ref{thm:invert}\eqref{thm:invert1}, but 
the double cover $\gr_{3\cdot 2^s}\rightarrow \ogr_{3\cdot 2^s}$ does
have $v_r(\gr_{3\cdot 2^s}/\ogr_{3\cdot 2^s})$ ramified vertices.
Hence,
\begin{equation*}
\label{molt}
2\chi(\ogr_{3\cdot 2^s})=\chi(\gr_{3\cdot 2^s})+v_r(\gr_{\Th}/\ogr_{\Th}).
\end{equation*}
Since $2\chi\big(\PUT\big(\Z_{K_n}^{(2)}\big)\big)=
\chi\big(\PSUT\big(\Z_{K_n}^{(2)}\big)\big)=M_{\Th}$ by 
Theorem \ref{masst}\eqref{founded}, we
have
\begin{equation}
\label{lava1}
E_{0,\Th}=1-b_1(\ogr_{\Th})+M_{\Th}/2
=(E_{1,\Th}+
v_r(\gr_{\Th}/\ogr_{\Th}))/2.
\end{equation}

\subsection{Naive bounds on the torsion invariant and Betti numbers}
\label{package}

In this section we make our first approach at the main problem of this paper --
to bound $E_{i,n}$, $i=0,1$, $b_1(\gr_n)$, and
$b_1(\ogr_n)$ for $n=2^s$ and $n=3\cdot 2^s$.
Here we see what we can say naively -- without counting optimal embeddings.
Proposition \ref{none}  is key to saying anything in this direction.

\begin{theorem}
\label{resolve}
Let $n=2^s$ or $n=3\cdot 2^s$.
We have the inequalities
\begin{align}
\label{resolve1}
(3-6\cdot2^s)2^{(3s-12.1)2^{s-3}-1/2}\leq & E_{1,2^s}
\leq 1+2^{(3s-12.1)2^{s-3}-1/2}\\
\nonumber
(3-6\cdot 3\cdot2^s)2^{(6s-19.4)2^{s-3}+1}\leq & E_{1,\tts}
\leq 1+2^{(6s-19.4)2^{s-3}+1}
\end{align}
and
\begin{align}
\label{resolve3}
0\leq & b_1(\gr_{2^s})\leq 1+(6\cdot 2^s-2)2^{(3s-12.1)2^{s-3}-1/2}\\
\nonumber
0\leq & b_1(\gr_{\Th})\leq 1+(18\cdot 2^s-2)2^{(6s-19.4)2^{s-3}+1} .
\end{align}
In particular, $E_{1,n}
=2^{O(n\log n)}$ and 
$b_1(\gr_n)=2^{O(n\log n)}$ as in  Proposition \textup{\ref{upper}}.
\end{theorem}
\begin{proof}
Since $e(\gr_n)=e_r(\gr_n)$ by Remark \ref{dale} and 
$\chi\big(\PSUT\big(\Z_{K_n}^{(2)}\big)\big)=-M_n$ by Theorem 
\ref{masst}\eqref{founded}, we have 
\begin{equation*}
\label{rested}
E_{1,n}=1-b_1(\gr_n)+M_n=v(\gr_n)-e(\gr_n)+M_n.
\end{equation*}
Using the bounds on $v(\gr_n)$ and $e(\gr_n)$
in Proposition \ref{none}\eqref{none1} 
as well as Theorem 
\ref{masst}\eqref{founded1} and $b_1(\gr_n)\geq 0$ gives
\begin{align}
\label{chinese}
 \VM(\gr_n)-2n\EM(\gr_n)+M_n  &\leq E_{1,n}
\leq 1+M_n\\
(3-6n)M_n=2M_n-2n(3M_n)+M_n&\leq E_{1,n}\leq 1+M_n.
\nonumber
\end{align}
Since $b_1(\gr_n)=1+M_n-E_{1,n}$, we get the following bounds on $b_1(\gr_n)$
from \eqref{chinese}:
\begin{equation}
\label{swift}
0\leq b_1(\gr_n)\leq 1+M_n-(3-6n)M_n=1+(6n-2)M_n.
\end{equation}
Finally using the bounds for $M_n$ in Theorem \ref{soup}
gives the bounds in Theorem \ref{resolve}.
\end{proof}

\begin{theorem}
\label{Resolve}
Let $n=2^s$ or $n=3\cdot 2^s$.
We have the inequalities
\begin{align}
\label{Resolve1}
(3-6\cdot 2^s)2^{(3s-12.1)2^{s-3}-3/2}\leq & E_{0,2^s}
\leq 1+2^{(3s-12.1)2^{s-3}-3/2}\\
\nonumber
(3-18\cdot 2^s)2^{(6s-19.4)2^{s-3}}\leq & E_{0,\tts}
\leq 1+2^{(6s-19.4)2^{s-3}}
\end{align}
and
\begin{align}
\label{Resolve3}
0\leq & b_1(\ogr_{2^s})\leq 1+(3\cdot 2^s-1)2^{(3s-12.1)2^{s-3}-1/2}\\
\nonumber
0\leq & b_1(\ogr_{\Th})\leq 1+(9\cdot 2^s-1)2^{(6s-19.4)2^{s-3}+1}.
\end{align}
In particular, $E_{0,n}=2^{O(n\log n)}$ and 
$b_1(\ogr_n)=2^{O(n\log n)}$ as in Proposition \textup{\ref{upper}}.
\end{theorem}
\begin{proof}
Since $\chi\big(\PUT\big(\Z_{K_n}^{(2)}\big)\big)=-M_n/2$ by Theorem 
\ref{masst}\eqref{founded}, we have 
\[
E_{0,n}= 1-b_1(\ogr_n)+M_n/2=
v(\ogr_n)-e_r(\ogr_n)+M_n/2. 
\]
Using the bounds on $v(\gr_n)$ and $e_r(\gr_n)$
in Proposition \ref{none}\eqref{none2} as well as 
Theorem 
\ref{masst}\eqref{founded1} and $b_1(\ogr_n)\geq 0$ gives
\begin{align}
\label{chinese1}
\VM(\ogr_n)-2n\EM(\ogr_n)+M_n/2  \leq & E_{0,n}\leq 1+M_n/2\\
\nonumber
(3-6n)(M_n/2)=M_n-2n(3/2)M_n+M_n/2 \leq & E_{0,n}
\leq 1+M_n/2.
\end{align}
Using $b_1(\ogr_n)=1+M_n/2-E_{0,n}$ we get from \eqref{chinese1}
\begin{equation}
\label{lunchbox}
0\leq b_1(\ogr_n)\leq 1 + M_n/2-(3-6n)(M_n/2)=1+(6n-2)(M_n/2)=1+(3n-1)M_n.
\end{equation}
Now use the bounds for $M_n$ in Theorem \ref{soup}
to get the bounds in Theorem \ref{Resolve}.
\end{proof}

\begin{remark}
\label{dinner}
The upper bounds above
for $E_{i,n}$ are trivial (and useless for studying asymptotic behavior
as $n\rightarrow \infty$),
as are the lower bounds for $b_1$.  All the work from this point on through
the end of Section \ref{three} is to obtain strong upper bounds for $E$ and
strong lower bounds for $b_1$.  
As discussed in the Introduction, to get the asymptotic growth
rates for $b_1(\gr_n)$ and $b_1(\ogr_n)$ of Main Theorem
\ref{main}, we need an upper bound for $E_{i,n}$ which
has a slower growth rate than $M_n$. 
Compare the naive bounds in Theorems \ref{resolve}, \ref{Resolve}
with the bounds we ultimately prove in Theorems \ref{main2b},
\ref{main2a}.
\end{remark}

\subsection{Optimal embeddings}
\label{sec: optimal}

We remind the reader that the notation $v_1, v_{<1}$ is in Definition~\ref{defn: breakup} and that $e_r, e_{h}$ appear in Definition~\ref{gr}.
We are seeking lower bounds on $b_1(\gr) = e(\gr) - v(\gr) +1$
and $b_1(\overline{\gr}) = e_r(\overline{\gr}) - v(\overline{\gr})
+1$. These will follow from upper bounds on $v(\gr)$ and
$v(\overline{\gr})$ and lower bounds on $e(\gr)$ and
$e_r(\overline{\gr})$. From Proposition~\ref{prop:bounds} we get an upper bound on $v_1(\gr)$ and a
lower bound on $e(\gr)$. Then we obtain upper bounds on
$v_{<1}(\gr)$ and thus on $v(\gr) = v_1(\gr) + v_{<1}(\gr)$. We derive
an upper bound on $v(\overline{\gr})$ from that on $v(\gr)$ by
bounding from above the number of vertices in $\gr$ which are
ramified in the the double cover $\gr \rightarrow \overline{\gr}$.
We finally obtain an upper bound on the number of edges in $\gr$ which
are inverted by $\PUT\big(\Z_{K_n}^{(2)}\big)$ which bounds $e_h(\overline{\gr})$ from above,
thereby bounding $e_r(\overline{\gr})$ from below by
Proposition~\ref{prop:bounds}.

All of these bounds will be achieved by bounding maximal orders containing  specific types of elements: roots of unity, {\it ramifying elements}
(Definition~\ref{defn:ramify}), and {\it inverting elements}
(Definition~\ref{defn:invert}). If an $\Z_{F_n}$-order $\mN\subseteq\H_n$ contains an element
$\gamma$ generating a quadratic extension of $F_n$, then the 
quadratic $\Z_{F_n}$-order $\fo := \Z_{F_n}[\gamma]$ embeds in $\mN$.

\begin{definition}\label{optimal embedding}
  Let $\mN$ be an $\Z_{F_n}$-order in $\H_n$, $L$ a quadratic extension
of $F_n$,  and $\fo$ a (not necessarily maximal) $\Z_{F_n}$-order in
$L$. An {\sf optimal embedding} of $\fo$ into $\mN$ is an
embedding  $L \xrightarrow{\iota} \H_n$ with
$\iota(L) \cap \mN = \iota(\fo)$. 

We say that two optimal embeddings $\iota_1, \iota_2$ are {\sf equivalent} if
$\iota_1(x) = u^{-1}\iota_2(x)u$ for all $x \in \fo$, where $u$ is a unit in
$\mN$.  Define
$m(\fo, \mN)$ to be the number of optimal embeddings of $\fo$
into $\mN$ up to equivalence. Now fix a complete set $I_1, \ldots,
I_h$ of representatives of the left ideal classes of $\mN$ and let
$\mN_i$ be their right orders. Set $m_{\mN}(\fo) := \sum_{i=1}^h
m(\fo, \mN_i)$.
\end{definition}

We specialize \cite[Corollaire III.5.12]{V} to our setting.
\begin{theorem}\label{class_no}
Let $n = 2^s$ or $3\cdot 2^s$ and let $L/F_n$ be a CM-extension of
$F_n$.  Fix a $\Z_{F_n}$-order $\fo$ in $L$. Then
$ m_{\M_n}(\fo) = h(\fo)$.
\end{theorem}
\begin{remark} The only ramified primes of $\H$ are the archimedean ones, 
  and these all ramify in $L/F$.  It thus follows from
  \cite[Corollaire~III.3.4]{V}
  that $L$ can be embedded into $\H$.
\end{remark}

\begin{definition1}\label{defn:ramify}
Let $n = 2^s$ or $3\cdot2^s$ and set $F= F_n$, $\M \colonequals \M_n$.
A noncentral element $\gamma$ in $\widetilde{\M}^\times=\M[1/2]^\times$
is said to be a {\sf ramifying element}
if $\Norm_{F(\gamma)/F}(\gamma)\in \Z_{F_n,+}^{(2),\times}$ is a nonsquare.  We will only be
concerned with ramifying elements up to multiplication by
$\Z_{F_n}^{(2),\times}$.
\end{definition1}
\begin{theorem}\label{thm:ram}
Let $n = 2^s$ or $3\cdot 2^s$ and set $\Delta \colonequals \Delta_n$, 
$\gr\colonequals
\PSUT\big(\Z_{K_n}^{(2)}\big)\backslash \Delta$, and 
$\overline{\gr} \colonequals \PUT\big(\Z_{F_n}^{(2)}\big)\backslash\Delta$. 
Let $\M
\in \Ver(\Delta)$ be a maximal order and $\bv$ the vertex it covers in
$\gr$. Then $\bv$ is ramified in the cover $\gr \xrightarrow{\pi}
\overline{\gr}$ if and only if $\M$ contains a ramifying element.
\end{theorem}
\begin{proof}
Fix $\M \in \Ver(\Delta_n)$.  Let $\gamma \in \M \subset
\widetilde{\M}_n$ be a ramifying element. Thus,
$\Norm_{F_n(\gamma)/F_n}(\gamma)\in \Z_{F_n,+}^\times$ is a nonsquare
totally positive unit so that $\gamma$ is a unit in $\widetilde{\M}_n$
such that $\varphi_n^{-1}(\gamma)$ generates $\PUT(R_n)$ over its
index-$2$ subgroup $\PSUT\big(Z_{K_n}^{(2)}\big)$ (Theorem~\ref{rink}). Hence
$\pi^{-1}(\pi(\bv)) = \{ \bv, \bv^\gamma\} = \{\bv\}$ because
$\M^{\gamma} = \M$.

Now suppose $\bv$ ramifies. Then there exists $[A] 
\in \PUT\big(\Z_{K_n}^{(2)}\big)
\setminus \PSUT\big(\Z_{K_n}^{(2)}\big)$ with $\M^{\varphi_n([A])} = \M$. Let
$\gamma\in\widetilde{\M}_n^\times$ be a lift of
$\varphi_n([A])\in\Pp\!\widetilde{\M}_n^\times$, note that
$\Norm(\gamma)$ isn't a square.  Thus by (\ref{ver}), $\gamma \in
\widetilde{\M}_n^\times\cap F_{\fp}^\times\M_{\fp}^\times=\Z_{F_n}^{(2),\times}\M^\times$.  Write
$\gamma = v\gamma_1$ with $v\in \Z_{F_n}^{(2),\times}$ and
$\gamma_1\in\M^\times$.  Note that $\Norm(\gamma_1)$ isn't a square
since $\Norm(\gamma)$ isn't; thus, $\gamma_1$ is a ramifying element in
$\M$.
\end{proof}
\begin{lemma}\label{lem:ram}
If $n = 2^s$, then $\widetilde{\M}^\times=\M_n[1/2]^\times$ contains no ramifying  elements. 
\end{lemma}
\begin{proof}
The cover $\PSUT\big(\Z_{K_n}^{(2)}\big)\backslash
\Delta_n \rightarrow \PUT\big(\Z_{K_n}^{(2)}\big)
\backslash \Delta_n$ is \'{e}tale by
Theorem~\ref{index}\eqref{index2}.
\end{proof}
The group $\PU_2\big(\Z_{K_n}^{(2)}\big)$ preserves the partition
of $\Delta$ according to the parity of the distance from a fixed
vertex $v_0$ (this partition does not depend on $v_0$).  However,
it may interchange the two sets of this partition, in which case
it acts with inversions.
\begin{definition1}\label{defn:invert}
Let $n = 2^s$ , $s \ge 2$,  with $F\colonequals F_n$, $\M\colonequals
 \M_n$, $\fp:=\fp_n$, and
$p=p_n$ 
the totally positive generator of $\fp$ in Definition~\ref{gen}.
An element $\gamma$ in $\widetilde{\M}^\times=\M[1/2]^\times$ is said to be an 
{\sf inverting
  element} if
\begin{enumerate}[\upshape (a)]
\item \label{defn:invert1} $\Norm(\gamma) = p$, and
\item \label{defn:invert2} $\gamma^2 = up$, with $u \in \M^\times$.
\end{enumerate}
\end{definition1}
Recall that for $\gamma \in \H$, $ \M \in \Ver(\Delta)$, we denote
by $\M^\gamma$ the order which differs from $\M$ only at $\fp$ where we
conjugate by $\gamma$.

\begin{theorem}\label{thm:invert}
  \begin{enumerate}[\upshape (a)]
  \item
    \label{thm:invert1}
  If $n = 3\cdot 2^s$ then $\PUT\big(\Z_{K_n}^{(2)}\big)$
  acts on $\PSUT\big(\Z_{K_n}^{(2)}\big)\backslash \Delta_n$ without inversions.
\item
  \label{thm:invert2}
    If $n = 2^s$, let $\M\in \Ver(\Delta_n)$ be a maximal order.
    Then $\M$ contains an inverting
element if and only if the vertex it covers in $\gr_n :=
\PSUT\big(\Z_{K_n}^{(2)}\big)\backslash \Delta_n$ is connected to
an edge that
is inverted by $\PUT\big(\Z_{K_n}^{(2)}\big)$.
\item
  \label{thm:invert3}
If an inverting element $\gamma \in \M$
exists, then  $\M^{\gamma^2} = \M$ and
$\mathcal{E} \colonequals \M \cap \M^{\gamma}$ is an Eichler order of
level $\fp \colonequals \fp_n$ that contains $\gamma$. Thus, $\gamma$ inverts an
edge between the images of $\M$ and $\M^{\gamma}$ in $\gr_n$.
\end{enumerate}
\end{theorem}
\begin{proof}
Let $n = 2^s$ or $3\cdot 2^s$ with $s \ge 2$ and set $\fp \colonequals
\fp_n$, where $\fp$ is generated by $p = p_n$. Suppose 
$\PSUT\big(\Z_{K_n}^{(2)}\big)$ inverts the
edge $\be \in \gr:= \PSUT\big(\Z_{K_n}^{(2)}\big) \backslash \Delta_n$. Set $\bv =
o(\be)$ and fix a lift $\M \in \Delta:= \Delta_n$
of $\gr$. We can find an edge $\tilde{\be}\in \Delta$ covering $\be$ with origin
$\M$, and $\gamma_0 \in \PUT\big(\Z_{K_n}^{(2)}\big)$ 
inverting $\tilde{\be}$. Thus
$t(\tilde{\be}) = \M^{\gamma_0}$, $\M^{\gamma_0^2} = \M$, and
$\mathcal{E} := \M \cap \M^{\gamma_0}$ is an Eichler order of level $\fp$.
Since $\gamma_0$ moves $\M$ to an adjacent vertex, it must be that
$v_{\fp}\left(\Norm(\gamma_0)\right) =n$, with $n$ odd, say $n =
2i+1$. Set $\gamma = \gamma_0/p^i$, so
$v_{\fp}\left(\Norm(\gamma)\right) = 1$. As $p$ is central, we have
$\M^{\gamma} = \M^{\gamma_0}$ and $\M^{\gamma^2} = \M$.
As $\gamma^2$ fixes $\M$ we have $\gamma^2 \in F_{\fp}^\times
\M_{\fp}^{\times}$. Thus we may write $\gamma^2 = p^m u$ with $u \in
\M_{\fp}^\times$. Since $2m = v_{p}\left(\Norm(\gamma^2) \right) =
2v_p\left(\Norm(\gamma)\right) = 2$ we have $\gamma^2 = pu$. Finally,
$u = \gamma^2/p \in \M[1/2]^\times\cap \M_{\fp}^{\times} =
\M^{\times}$. But this says that $\Norm(\gamma^2)$ generates the ideal
$\fp^2$ in $\Z_F\colonequals \Z_{F_n}$, which forces the totally positive element
$\Norm(\gamma)$ to generate $\fp$.

\eqref{thm:invert1}: If $n = 3\cdot 2^s$, $\fp$ has no
totally positive generator so $\PUT\big(\Z_{K_n}^{(2)}\big)$ 
acts without inversions on
$\gr$.

\eqref{thm:invert2}:  If $n = 2^s$, then $\Norm(\gamma) = p$, 
the totally positive
generator of $\fp$. Hence $\gamma \in \M$ is an inverting element.

\eqref{thm:invert3}: Now 
assume that an inverting element $\gamma \in
\M$ exists. Then a local argument at $\fp$ shows that 
$\mathcal{E} \colonequals \M\cap \M^\gamma$ is an Eichler order of level
$\left(\Norm(\gamma)\right) =\fp$. But
$\gamma^2 = up$ with $p\in \uOO$  central and $u \in M^\times$,
$\M^{\gamma^2} = \M$, so $\gamma$
inverts an  edge in $\gr$ covered by an edge from $\M$ to $\M^\gamma$. Moreover, $\gamma
\in \M^{\gamma}$ because
$\gamma \in \gamma \M_{\fp} \gamma^{-1}$, so $\gamma \in \mathcal{E}$. 
\end{proof}

\subsection{\texorpdfstring{The \except{toc}{\boldmath{$n=2^s$}}\for
{toc}{$n=2^s$} family}{The n=2\unichar{"5E}s family}}
\label{calmly}

\subsubsection[Quadratic
  \texorpdfstring{$\Z_{F_{2^s}}$}
{Z[\unichar{"03B6}2\unichar{"5E}s]+}-orders
  containing roots of unity]
              {Quadratic
                $\Z_{F_{2^s}}$-orders containing roots of unity}
  \label{sec: norm}

The only quadratic extensions of 
$F\colonequals F_{2^s}$ generated by roots of $1$ are $F(i) = K_{2^s}
\equalscolon K$ and
$F(\sqrt{-3})$.
If $\M^\times/\Z_{F}^\times$ is nontrivial, then $\M$ contains
a $\Z_F$-order in one of these fields.
\begin{definition1}\label{def:Ok}
  Set $\zeta \colonequals \zeta_{2^s}$. For $0 \le k < 2^{s-2}$, let
\[
  \fo_k\colonequals
  \fo_{2^s,k} \colonequals \Z_F[i, (\zeta+\zeta^{-1})^k\zeta] \subset K .
  \]
\end{definition1}

(We stop at $k = 2^{s-2}-1$ because $\fo_{k} = \Z_f[i]$ for $k \ge 2^{s-2}-1$.)

\begin{prop}\label{indices}
  Fix $s$ with $K= K_{2^s}$. The $\fo_k$ are $\Z_F$-orders in $K$ with $\fo_0$ maximal and $\fo_{2^{s-2}-1} = \Z_F[i]$. The index of $\fo_k$ in $\fo_0$ is $2^k$
  and the conductor of $\fo_k$
over $\fo_0$
is $\fP^{2k}$. The $\fo_k$ are the only $\Z_F$-orders in $K$
containing an irrational root of unity. In particular, $\Z_F[\zeta_{2^w}] =
\fo_{2^{s-w}-1}$ for $w = 2, \ldots,
s$.
\end{prop}

\begin{proof} 
That $\fo_0 \colonequals \Z_F[i,\zeta] = \Z_F[\zeta_{2^s}]$ is maximal  
is \cite[Prop.~2.16]{W}.  It is also clear that $(\ATr\mO_K) = (\fP^2)$, since 
$(\zeta-\zeta^{-1}) = (\fP^2)$ and $\zeta^j-\zeta^{-j}$ is a multiple of
$\zeta-\zeta^{-1}$ for all $j$.  In addition $\ATr((\zeta+\zeta^{-1})^k\zeta) =
(\zeta+\zeta^{-1})^k(\zeta-\zeta^{-1})$ is a generator of $(\fP^{2(k+1)})$ and
the antitrace of $i^j ((\zeta+\zeta^{-1})^k\zeta)^\ell$ is equal to
$(\zeta+\zeta^{-1})^{k\ell}\ATr(i^j \zeta^\ell)$.  

Now, the antitrace of
$\pm \zeta^\ell$ is $\pm (\zeta^\ell-\zeta^{-\ell})$, while that of
$\pm i\zeta^\ell$ is $\pm (\zeta^{2^{s-2}+\ell}-\zeta^{-2^{s-2}-\ell})$, both of which
belong to $\fP^2$.  It follows that every element of $\fo_k$ has antitrace in
$\fP^{2(k+1)}$, since this order is spanned by the 
$i^j ((\zeta+\zeta^{-1})^k\zeta)^\ell$ for $j, \ell \in \N$.  Applying
Proposition~\ref{prop:atr-order} we conclude that the conductor
of $\fo_k$ is $\fP^{2k}$ as claimed.

The antitrace of every order containing $\Z_F[i] = \fo_{2^{s-2}-1}$ is an ideal
containing $\fP^{2(2^{s-1}-1)}$, and the only such ideals that descend to $F$
are the $\fP^{2k}$ for $0 \le k < 2^{s-2}$.  For every such $k$ we have
found an order with antitrace $\fP^{2k}$, so by Corollary~\ref{cor:distinct-atr}
we have found all the orders containing $\Z_F[i]$.
Since the roots of unity in $K$ are powers of $\zeta_{2^s}$, every order
containing a root of unity other than $\pm 1$ contains $i$, and so we have
found all orders containing an irrational root of unity.

\end{proof}

\begin{prop}\label{orders}
The order $\Z_F[\zeta_3]$ of $F_s(\sqrt{-3})$ is maximal and is therefore the
only $\Z_F$-order of $F_{2^s}(\sqrt{-3})$ containing an irrational root of $1$.
\end{prop}

\begin{proof} If
$L_1, L_2$ are number fields with disjoint sets of ramified primes, then
$\Z_{L_1L_2} = \Z_{L_1}\Z_{L_2}$
\cite[Prop.~I.2.11]{N2}.  Applying this to $F_{2^s}, \Q(\sqrt{-3})$ gives the 
first statement.
\end{proof}

Let $L$ be a CM number field with totally real subfield $L^+$.  Put
$h^+(L)=h(L^+)$ and $h^-(L)=h(L)/h(L^+)$.  This extends to orders:
for an order \mbox{$\fo\subseteq L$}, let $\fo^+=\fo\cap L^+$ and put
$h(\fo)=\#\Pic(\fo)$, 
$h^+(\fo)=\#\Pic(\fo^+)$, and $h^-(\fo)=h(\fo)/h^+(\fo)$.
\begin{prop}\label{mass-g-4th-root}
  Let $n=2^s$.  Then
$$v_{<1}(\gr) \le 2h^-(F_n(\sqrt{-3})) +
  2\sum_{k=0}^{2^{s-2}-1}h^-(\fo_k)  \, .$$
\end{prop}

\begin{proof}
Fix $s\ge2$ and set $F=F^{2^s}$, $K=K_{2^s}$, $\H=\H_{2^s}$ and  $\M = \M_{2^s}$.
 Let $\fo \subset \H$ be a quadratic $\Z_F$-order
containing a nontrivial root of unity. By Propositions~\ref{indices}
and \ref{orders} $\fo$ is $\Z_F[\zeta_3]$ or one of the $\fo_k$, $k = 0, \ldots,
2^{s-2}-1$ (Definition~\ref{def:Ok}). Fix an $\M$-class $\mathcal{C}$
of orders and let $\M_1, \ldots, \M_{h(\H)/h^+(\Z_K)}$ be a complete set of
elements of $\mathcal{C}$ representing the $\Cl(\Z_F)$-orbits in
$\mathcal{C}$.

Combining theorem~\ref{thm:Mn} with the fact that $\gr$ covers
$\PUT\big(\Z_K^{(2)}\big)\backslash\Delta_s$ with degree $2$, we see that an optimal 
embedding of
$\fo$ into one of the $\M_i$ gives rise to embeddings of the roots of
unity  in $\fo$ into
the
$\SUT\big(\Z_K^{(2)}\big)$-stabilizer group of at most 2 vertices of $\gr$. 
\end{proof}
\noindent
Every nontrivial vertex stabilizer arises from an optimal embedding of at least
one of our orders $\fo$. A given maximal order may receive optimal
embeddings of more than one of the orders, or indeed, more than one
inequivalent embedding of the same order.

For a commutative ring $R$ let $\mu(R)$ be the group of roots of unity
in $R$ and let $w_k$ be such that $\# \mu(\fo_k) = 2^{w_k}$.
\begin{prop}
  \label{granted}
The unit
index $[\fo_0^\times:\fo_k^\times]$ is $2^{s-w_k} =
[\mu(\fo_0):\mu(\fo_k)]$.
\end{prop}
\begin{proof} The cyclotomic units of $K^+ =F$ are generated by 
  the \mbox{$\zeta^{(1-i)/2} \frac{1-\zeta^i}{1-\zeta}$} for odd $i$ from
  $1$ to $n-1$, while the cyclotomic units of $K$ are
  generated by these and $\zeta$ \cite[Lemma~8.1]{W}.  Observe that
  every generator $\frac{1-\zeta^i}{1-\zeta}$ is equal to
  $\zeta^{(i-1)/2} \sum_{j=-(i-1)/2)}^{(i-1)/2} \zeta$ and is therefore the product
  of a real unit and a root of unity.
  
  Let $U$ be the group of cyclotomic units of $\fo_0$ and
  let $U_k = U \cap \fo_k$.  
  By \cite[Thm.~8.2]{W}, the index of $U$ in the
  full unit group is $h^+(K)=h(F)$,
  which is odd by \cite[Thm.~10.4]{W}.
  We show that $[U:U_k] = 2^{s-w_k}$.  

Indeed, the class of $\zeta$ in the
  quotient has order $2^{s-w_k}$, while the remark above shows that
  $\zeta$ generates the quotient.
  On the other hand, we know that $[\fo_0^\times:\fo_k^\times]$ is a power of $2$,
  because $(\fo_0/2)^\times$ is a $2$-group and every unit that is $1$ mod $2$
  is contained in every $\fo_k$.  Since
\[
  [U:U_k] | [\fo_0^\times:\fo_k^\times] | [\fo_0^\times:U_k] = h^+(K) [U:U_k],
\]
  it follows that $[\fo_0^\times:\fo_k^\times] = [U:U_k] = 2^{s-w_k}$.
\end{proof}

\begin{prop}
  \label{prop: class_no_k_1}
The class number of $\fo_k$ is equal to 
$h^+(K)h^-(K) \cdot 2^{k-s+w_k}$.
\end{prop}
\begin{proof} By \cite[Thm.~I.12.12]{N2}, the class number of $\fo_k$ is
$$(h^+(K)h^-(K)/[\fo_0^\times:\fo_k^\times]) (\#(\fo_0/\f)^\times/\#(\fo_k/\f)^\times),$$
where $\f$ is the
conductor.  Both quotients by $\f$ are local rings with residue field of order
$2$, so half of their elements are units; since the conductor of $\fo_k$ has
index $2^{2k}$ in $\fo_0$, its index in $\fo_k$ is $2^k$, and so the second 
factor is $2^{2k-1}/2^{k-1} = 2^k$.  The result follows by using Proposition~\ref{granted}
to evaluate $[\fo_0^\times:\fo_k^\times]$.
\end{proof}

\begin{prop}\label{prop: roots bound}
  Let $n=2^s$.  Then
  $$v_{<1}(\gr) \le 2h^-(F_n(\sqrt{-3}))+ 2^{2^{s-2}+1} h^-(K_n) \, .$$
\end{prop}

\begin{proof}
  By Proposition~\ref{prop: class_no_k_1} we have
$$\frac{\sum_{k=0}^{2^{s-2}-1}h\left(\fo_k\right) }{h(F_n)} =
  h^-(K_n)\sum_{k = 0}^{2^{s-2}-1}2^{k-s+w_k} \le h^-(K_n)\sum_{k =
    0}^{2^{s-2}-1}2^{k-s+s}< h^-(K_n)2^{2^{s-2}}\, ,$$
where $\fo_k$ has exactly $2^{w_k}$ roots of unity.
Combining this bound with Proposition~\ref{mass-g-4th-root} completes
the proof.
\end{proof}

\subsubsection[Quadratic
  \texorpdfstring{$\Z_{F_{2^s}}$}{Z[\unichar{"03B6}2\unichar{"5E}s]+}-orders
  containing inverting elements]
              {Quadratic
                $\Z_{F_{2^s}}$-orders
                containing inverting elements}
\label{Sec:invert}
By Theorem~\ref{thm:invert}, we can bound inversions in 
$\gr:= \SUT\big(\Z_{K_n}^{(2)}\big) \backslash \Delta_n$, and
hence half-edges in $\overline{\gr} \colonequals
 \PUT\big(\Z_{K_n}^{(2)}\big)\backslash \Delta_n$,
by counting optimal embeddings of quadratic $\Z_{F_n}$-orders
containing inverting elements into an $\M_n$-class of maximal
quaternionic orders.

\begin{theorem}\label{which_fields}
Let $L$ be a totally complex 
quadratic extension of $F\colonequals F_{2^s}$.  
Suppose that $\p \Z_L$ is the square of a principal ideal.  
Then $L$ is either $K= K_{2^s}$ or
$F(\sqrt{-p_{2^s}})$.  In either case a generator of the
prime ideal of $L$ above
$\p := \p_{2^s}\subset \Z_F$ generates the maximal order $\Z_L$ over $\Z_F$.
\end{theorem}
\begin{proof}
Let $q$ be a generator of the prime ideal, $\q$, above $\p$ and
set $p=p_{2^s}$. Consider the exact
sequence
$$0\to\Z_L^\times\to\Z_L^\q=\Z_L[1/q]\xrightarrow{v_\q}\Z\to0.$$ Now
applying $L/F$-Galois cohomology
gives
\begin{align*}
  0\to\Z/v_\q(\Z_F[1/p]^\times)&(\cong\Z/2\Z)\to H^1_{L/F}(\Z_L^\times)\\
& \to
  H^1_{L/F}(\Z_L[1/q]^\times) \to H^1_{L/F}(\Z) = 0.
\end{align*}
\noindent It follows from
\cite[Thm.~4.12]{W} that $\#H^1_{L/F}(\Z_L^\times)$ is either $1$ or
$2$.  Hence we must have $H^1_{L/F}(\Z_L^\times)\cong\Z/2\Z$ and
$H^1_{L/F}(\Z_L[1/q]^\times)=0$.

We next apply Galois cohomology to the exact
sequence $$0\to\Z_L[1/q]^\times\to L^\times\to\Prin(\Z_L[1/q])\to0$$ to
obtain $$0\to\Z_F[1/p]^\times\to F^\times\to\Prin(\Z_L[1/q])^{\text{Gal}(L/F)}
\to
H^1_{L/F}(\Z_L[1/q]^\times)=0.$$
Therefore,
$\Prin(\Z_L[1/p])^{\text{Gal}(L/F)}=\Prin(\Z_F[1/p])$ and in particular the
discriminant of $L/F$ must be a unit in $\Z_F[1/p]$.  This unit must
also be negative definite since $L$ is totally complex.  By Theorem~\ref{fields} and its proof the only such
units up to squares are $-1$ and $-p$.  Hence $L$ is either $K$
or $F(\sqrt{-p})$.

In the latter case, $q$ must be equal to $\sqrt{-p}$ up to units of $\Z_F$
(again by Theorem~\ref{fields}), hence it generates the maximal order
$\Z_F[\sqrt{-p}]$.  In the
former case it is easy to see that $q = \zeta^k(1+\zeta)$
up to units of $\Z_F$.  Hence multiplying by
$\zeta^{-k}(1+\zeta^{2k+1})(1+\zeta)^{-1}\in\Z_F^\times$ gives us
$1+\zeta^{2k+1}$ which generates the maximal order.
\end{proof}
\begin{lemma}\label{lem:invert}
Let $n = 2^s$ and set $F = F_{2^s}, K = K_{2^s}, p = p_{2^s},
\M=\M_{2^s}$. 
If $\gamma \in \M[1/2]^\times$ is an inverting
element, then $F(\gamma)$ is either $K$ or $F(\sqrt{-p})$. 
\end{lemma}
\begin{proof}
  By Proposition~\ref{which_fields} it suffices to show that $\p$ ramifies
as the square of a principal ideal in $L=F(\gamma)$. By definition
$\Norm(\gamma) = p$ so $\gamma^2$ generates $\Z_L \cdot p = \Z_L \fp$.
\end{proof}

\begin{proposition}\label{char_poly}
  Let $n = 2^s$. Set $F = F_n$, $K= K_n$,
  $p = p_{n}$, and $\Delta =
  \Delta_{n}$. Fix a maximal order
  $\M \in \Ver(\Delta)$ and let $\bv$, $\overline{\bv}$ be
  its image in $\gr$ and $\overline{\gr}$ respectively. Every embedding of the ring of integers in $F(\sqrt{-p})$
  or $K$ into $\M$ is associated to exactly one inverted pair of an edge and its opposite in $\gr\colonequals
  \PSUT\big(\Z_K^{(2)}\big)\backslash \Delta$ incident 
upon $\bv$. This pair thus covers a unique
  half-edge  in $\overline{\gr} \colonequals
  \PUT\big(\Z_K^{(2)}\big)\backslash \Delta$ incident upon $\overline{\bv}$. All
  half-edges in $\overline{\gr}$  incident upon $\overline{\bv}$ arise
  from such embeddings.
  Moreover, an equivalent embedding gives rise to the same half-edge in 
$\overline{\gr}$.
\end{proposition}
\begin{proof}

Set $\zeta =\zeta_{2^s}$. From Theorem~\ref{thm:invert} and
Proposition~\ref{lem:invert} it suffices to consider optimal
embeddings of the orders $\fo_{\gamma} \colonequals\Z_F[\gamma]$ 
into $\M$ for $\gamma$ an inverting
element in  $L = F(\sqrt{-p})$ or $L=K$. 
Fix such an embedding so that $\fo_\gamma\subset\Z_L$. 
We have $\Norm(\gamma) = p$
and $\gamma^2 = up$ with $u = \gamma^2/p \in  \M
^\times \cap L = \fo_{\gamma}^\times$.

Now $\gamma$ inverts the
unique pair of an edge and its opposite
between $\M$ and $\M^{\gamma}$ in $\Delta$ which maps to an inverted
pair of an edge and its opposite between $\bv$ and $\bv^{\gamma}$ in
$\gr$. This pair covers a
half-edge $\be$ (which is its own opposite) in $\overline{\gr}$ with 
$o(\be) = \overline{\bv} = t(\be)$.

Any other element in $L$ of norm $p$
must differ from $\gamma$ by a unit $\mu$ whose norm
down to $F$ is 1.  Such a unit is a root of
unity in $L$. 

If $L= F(\sqrt{-p})$, then $\sqrt{-p}$ is clearly an inverting element
so $\gamma = \pm \sqrt{-p}$ and thus $\fo_{\gamma} = \Z_L$. Since
$\M^{\gamma} = \M^{-\gamma}$ both choices for $\gamma$ invert the same pair of
edges in $\gr$ which covers the same half-edge $\be$ in $\overline{\gr}$.

Now suppose $L= K$. It is obvious that
$\gamma_i = \zeta^i(1 + \zeta)$ has norm $p = 2+\zeta+\zeta^{-1}$
(this is how we defined $p=p_{2^s}$ at the beginning of this section)
for $i = 0,\ldots, 2^s$. Since
\[
\gamma_i^2 = \zeta^{2i}(1+\zeta)^2
= \zeta^{2i}(1 + 2\zeta + \zeta^2) =
\zeta^{2i+1}(2 + \zeta+ \zeta^{-1}) = \zeta^{2i+1} p\, ,
\]
we have that $\gamma_i$ is an inverting element. There cannot be any others in
$L$. We also have $\fo_{\gamma_i} = \Z_L$ for all $i$. In particular
$\zeta$ lies in $\M^\times$. Thus
$\M^{\zeta^i\gamma} = \M^{\gamma}$ so all choices for $\gamma$ invert the same pair of
edges in $\gr$ which covers the same half-edge $\be$ in
$\overline{\gr}$.

Inequivalent embeddings $\Z_L \hookrightarrow \M$ differ by conjugation
by a unit in $\M$. So take $\mu\in \M^{\times}$ and  consider
$\M^{\mu^{-1}\gamma\mu} = (\M^{\gamma})^{\mu}$. But $\mu \in
\M^{\times} \subset \M[1/2]^\times$. By Theorem~\ref{U} the action of
$\PUT\big(\Z_K^{(2)}\big)$ on $\Delta$ is via its identification with
$\M[1/2]^\times$ modulo scalars. Thus the pair of edges inverted by
$\gamma$ and the pair inverted by $\mu^{-1}\gamma\mu$ all cover the same half-edge $\be$ in $\overline{\gr}$.
\end{proof}
Note that inequivalent embeddings may give rise to the same half-edge,
so counting inequivalent embeddings gives only an upper bound on half-edges,
not an exact count.
\begin{theorem}
  \label{ponds}
  Let $n = 2^s$ with $s \ge 2$. The number of half-edges $e_h(\overline{\gr})$
  of $\overline{\gr}\colonequals \PUT\big(\Z_{K_n}^{(2)}\big)
\backslash \Delta_n$ satisfies
  \begin{equation*}
    e_h(\overline{\gr}_s) \leq   h^-(K_n)+h^-(F_n(\sqrt{-p}))\, .
     \end{equation*}
\end{theorem}
\begin{proof}
  Combine  Theorem~\ref{thm:Mn}, Theorem~\ref{class_no}, and  Proposition~\ref{char_poly}.
\end{proof}

\subsection
{\texorpdfstring{The \except{toc}{\boldmath{$n=3\cdot 2^s$}}
\for{toc}{$n=3\cdot 2^s$} family}{The n=3\unichar{"B7}2\unichar{"5E}s family}}
\label{harried}

\begin{notation}
\label{dark}
{\rm
In this Section \ref{harried} we have $n=3\cdot 2^s$ and $\zeta=\zeta_n$.
We have $K=K_n$ and  $F=F_n=K_n^+$.
Let $\fP'\colonequals \fP_n$ be the unique prime above $2$ in $K$ and
$\fp'\colonequals \fp_n$ the unique prime above $2$ in $F$. Set 
$p'_n = 1 + \zeta_{3\cdot 2^s} + \zeta_{3\cdot 2^s}^{-1}$ and  put
$\gr'=\gr_n=\gr_{\Th}$.
}
\end{notation}

\subsubsection[Quadratic
  \texorpdfstring{$\Z_{F_\tts}$}{Z[\unichar{"03B6}3\unichar{"B7}2\unichar{"5E}s]+}-orders
  containing roots of unity] 
              {Quadratic $\Z_{F_\tts}$-orders
  containing roots of unity}\label{sec: norm3}

The only quadratic extensions of $F\colonequals F_n$  that contain roots of $1$
other than $\pm 1$
are $K=K_n$ and $F(i)$.  These two fields are isomorphic for $s \ge 2$.

%
\begin{definition}
  Let $n = 3\cdot 2^s$ with $s \ge 2$.
  For $0 \le k \le 2^{s-1}$,
  let 
$$\fo'_k := \fo'_{n,k} := \Z_F[i,(\zeta+\zeta^{-1}+1)^k \zeta]
  \subset K . $$
\end{definition}

\begin{prop}\label{indices-three}
  Fix $n = 3\cdot 2^s$ with $s \ge 2$ and set $K= K_n$. The $\fo'_k$
  are $\Z_F$-orders in $K$ with $\fo'_0 = \Z_F[\zeta_{3\cdot 2^s}]$
  maximal, $\zeta_3 \notin \fo'_1$, and $\fo'_{2^{s-1}} = \Z_F[i]$. The index
 of $\fo'_k$ in $\fo'_0$ is $2^k$
  and the conductor of $\fo'_k$
in $\fo'_0$
is ${\fP'}^k$. Moreover, $\Z_{F}[\zeta_{2^w}] =
\fo'_{2^{s-w+1}}$ for $w = 2, \ldots, s$.
\end{prop}

\begin{proof}
The proof is essentially identical to that of Proposition~\ref{indices}.
As before, $\zeta+\zeta^{-1}+1$ generates $\fp'$, but now 
$\fP' = \fp' \Z_K$ (previously we had $\fP = \fp^2 \Z_K$).  So the antitrace
of \mbox{$(\zeta+\zeta^{-1}+1)^k i^j \zeta^n$} is a multiple of $(\zeta+\zeta^{-1}+1)^k$
and hence belongs to $\fP'^k$.  Conversely we have 
$\ATr((\zeta+\zeta^{-1}+1)^k\zeta) = (\zeta+\zeta^{-1}+1)^k(\zeta-\zeta^{-1})$,
in which the second factor is a unit.  This proves that the given orders are
distinct.

To complete the proof, we need to show that the conductor of 
$\fo'_{2^{s-1}} = \Z_F[i]$ is ${\fP'}^{2^{s-1}}$.  If this ideal is contained in 
$\Z_F[i]$, then the conductor cannot be smaller, because the number of
distinct orders containing $\Z_F[i]$ is equal to the number of ideals 
containing the conductor, and we have already found $2^{s-1}+1$ such orders.

Thus we must prove that all elements of ${\fP'}^{2^{s-1}} = (2)$ belong to
$\fo'_{2^{s-1}}$.  It suffices to do this for elements of the form $2\zeta^j$;
an easy induction on $s$ allows us to assume that $j$ is odd, while the case
$3|j$ reduces to the result of Proposition~\ref{indices}.  Since 
$\fo'_{2^{s-1}}$ is Galois-invariant, it is enough to consider $j = 1$,
for which we note that $(2\zeta-(\zeta+1/\zeta))/i$ is a real algebraic 
integer and hence belongs to $\Z_F$.
\end{proof}



\begin{definition1}\label{def:T}
  Let $n = 3\cdot 2^s$ with $s\ge 3$ and set $T:= \Z_F[\zeta_3]$ to be the
  $\Z_F$-order
of $K$ generated by $\zeta_3$.  Let $\fp'_3$ be the unique prime of $F=K^+$ 
above
$3$ and let $\fP'_3, \fP''_3$ be the two primes of $K$ above $\fp'_3$.
\end{definition1}

\begin{prop}\label{cond-zeta-3}
The conductor of $T$ is $\fP'_3 \fP''_3 = \fp'_3$.
\end{prop}

\begin{proof}
First we note that $T$ is not the maximal order.  Indeed, consider the 
reduction map $\mO \to \mO/\fP'_3 \oplus \mO/\fP'_3$.  By a standard result it
is surjective; however, the image of $T$ is generated by elements of the form
$(a,\beta(a))$, where $\beta$ is the isomorphism that identifies the images of
$\zeta$ in the two quotients, and so it is isomorphic to $\mO/\fP'_3$.

Now we show that $\fP'_3 \fP''_3 \subset T$.  Note that 
$\fP'_3 \fP''_3 = (\zeta_6+1) = (\zeta^{2^{s-1}}+1)$, so it suffices to show
that elements of the form $(\zeta^{2^{s-1}}+1)\zeta^i$ belong to $T$.
As before, by induction we may assume that $i$ is odd, and
by Galois invariance we reduce to the cases $i = 1, 3$.  In each case,
observe that $[(\zeta^{2^{s-1}}+1)\zeta^i-(\zeta^i+\zeta^{-i})]\zeta_3^{-1}$ is a 
real algebraic integer and hence an element of $\uOO$; the claim follows.

Thus the conductor of $T$ is an ideal containing $\fp'_3$ but not equal to
$\Z_F$.  Since $\fp'_3$ is prime, the result follows.
\end{proof}

\begin{corollary}\label{orders-three-zeta3} There are two orders of $K$ containing $T$, namely $T$ and
the maximal order of $K$.
\end{corollary}

\begin{proof}
Since the conductor of $T$ is prime, this follows from Theorem~\ref{thm:cond}.
\end{proof}

\begin{prop}\label{mass-g three}
  Let $n=3\cdot2^s$ with $s \ge 3$.  Then
$$v_{<1}(\gr') \le 2h^-\left(T\right) +
  2\sum_{k=0}^{2^{s-1}}h^-\left(\fo'_k\right)  \, .$$
\end{prop}

\begin{proof}
  The proof is similar to that of Proposition~\ref{mass-g-4th-root},
this time using Proposition~\ref{indices-three} and Corollary~\ref{orders-three-zeta3}. 
\end{proof}

For a commutative ring $R$, let $\mu(R)$ be the group of
torsion elements of $R^\times$ and let $\mu_p(R)$ be the subgroup of $\mu(R)$
consisting of elements of $p$-power order.
Define $w'_k$ to be the $2$-adic valuation of $|\mu_2(\fo'_k)|$.
\begin{prop}\label{unit-index-3}
  The unit index $[{\fo'_0}^\times:{\fo'_1}^\times]$ is equal to $3 =
  [\mu(\fo'_0): \mu(\fo'_1)]$.
  For $k > 1$, the unit index $[{\fo'_0}^\times:{\fo'_k}^\times]$ is equal to
  $3\cdot 2^{s-w'_k +1} = 2[\mu(\fo'_0): \mu(\fo'_k)]$. 
\end{prop}

\begin{proof} 
  We first observe that $[{\fo'_0}^\times:{\fo'_k}^\times]$ is of the form
  $2^r 3^j$ with $j \le 1$, because the order of $(\fo'_0/2)^\times$ is of that
  form (since $(2)$ is a power of $\fP'$, whose residue field has order $4$).
  Every element of $\fo'_k$ is congruent to $0$ or $1 \bmod \fP'$, because
  $\fo'_k$ is generated by 
\[
\zeta+\zeta^{-1}+1,\,\,\, (\zeta+\zeta^{-1}+1)^k\zeta, 
\text{  which belong to $\fP'$, and $i$,}
\]
 which is $1 \bmod \fP'$.
  On the other hand, the cyclotomic unit $\zeta-1$ is not $0$ or $1 \bmod \fP'$.
  Thus the factor $3$ always appears.  It remains to determine
  the power of $2$.

  Consider the subgroup of $\Z_K^\times$ generated by
  $\Z_F^\times$ and $\zeta$.  By \cite[Thm.~4.12,
    Cor.~4.13]{W} we know that it is of index $2$.  However, it is a
  subgroup of the group generated by ${\fo'_k}^\times$ and $\zeta$, because
  $\Z_F\subset \fo'_k$.  So it suffices to determine whether
  $\langle \Z_F^\times, \zeta \rangle =
    \langle {\fo'_k}^\times, \zeta \rangle$.
  
  We will show that this holds for $k > 1$ but not for $k = 1$.  Consider
  the unit $\zeta+1$, which is not in $\langle \Z_F^\times, \zeta \rangle$.
  However, it is in $\fo'_1$, because $\fo'_1$ has index $2$ in the maximal
  order and does not contain $\zeta^r$ if $r$ is not a multiple of $3$.
  On the other hand, the index of $\fo'_2$ is $4$, with the quotient generated
  by $\zeta, \zeta^2$, each of order $2$.
  Since $\zeta^3 \in \fo'_2$, the class of
  $\zeta^{3r+j} \bmod \fo'_2$ is equal to that of $\zeta^j$, and so the classes
  of $\zeta^{j}$ and $\zeta^{j+1}$ are always distinct.  This establishes that
  $(\zeta+1)$ multiplied by any power of $\zeta$ does not belong to $\fo'_2$;
  {\it a fortiori} no such product belongs to $\fo'_k$ for $k \ge 2$.
\end{proof}

\begin{prop}\label{unit-index-t} The unit index of
  $T = \Z[\zeta + \zeta^{-1}, \zeta_3]$ is $2^s$.
\end{prop}

\begin{proof} Let $T^\times$ be the group of units in $T$ and $\mu_{3 \cdot 2^s}$
  the group generated by $\zeta =\zeta_{3 \cdot 2^s}$.  Since $T$ contains exactly $6$
  roots of unity, it suffices to show that
  $\langle T^\times, \zeta_{3 \cdot 2^s} \rangle$
  has index $2$ in $\Z_K^\times$.

  To see that the groups are not equal, we show that
  $(1-\zeta)\zeta^r$ does not belong to $T$ for any integer $r$.
  As seen above (and as can easily be verified directly), every element
  $t \in T$ satisfies $\ATr(t) \in \fP'_3 \fP''_3$.  On the other hand, we
  have 
  \begin{align*}
    \ATr((1-\zeta)\zeta^r) &= \zeta^r - \zeta^{-r} - \zeta^{r+1} + \zeta^{-r-1}\\
    &= (1-\zeta)(\zeta^r+\zeta^{-1-r})\\
    &= (1-\zeta)(\zeta^{-1-r})(1+\zeta^{2r+1}).
  \end{align*}
  The first two factors are units, while the third has norm $4$ if $3|(2r+1)$ 
  and $1$ otherwise (this follows from Lemma~\ref{norms} together with the
  fact that $N_{L/K}(x) = x^{[L:K]}$ for $x \in K$).  Hence 
  $\ATr((1-\zeta)\zeta^r) \notin \fP'_3 \fP''_3$ and 
  $(1-\zeta)\zeta^r \notin T$ as claimed.

\end{proof}

\begin{prop}\label{prop: class number}
Set $h'=h(\Z_K)$ for $n=3\cdot 2^s$ with $s \ge 3$.
  The class number of $\fo'_1$ is $h'$.
  For $k > 1$ the class number of $\fo'_k$ is
  $h'\cdot 2^{k -s + w'_k-2}$.
\end{prop}

\begin{proof} We evaluate the factors in \cite[Thm.~I.12.12]{N2}.
  We computed the unit index in Proposition~\ref{unit-index-3}
  to be $3$ for $k = 1$ and $6\cdot 2^{s-w'_k}$ for $k > 1$.
  Since $\Z_K/\f_{\fo'_k}$ is isomorphic to
  $\F_4[t]/(t^k)$, its group of units has order $3 \cdot 4^{k-1}$.
  Finally, $\fo'_k/\f_{\fo'_k} \cong \F_2[t]/(t^k)$; hence the group of 
units there has order
$2^{k-1}$. So the class number is
  \[
  \frac{h(\Z_K) \cdot 3 \cdot 4^{k-1} \cdot 2^{w'_k}}{6 \cdot 2^s \cdot 2^{k-1}}
  = h'\cdot 2^{k-2-s+ w'_k} 
  \]
as claimed.
\end{proof}

Let $n=\tts$, $n\geq 12$.
Let $T_n\colonequals\Z_{F_n}[\zeta_{3}]$ and $\f=\f_{n}$ be the conductor
of $T_n$ in $\Z_{K_n}$.
\begin{prop} \label{prop:class T}The class number of $T_n$
   is $h(K_n) \cdot (3^{2^{s-2}}-1)/2^s$.
\end{prop}

\begin{proof} This follows from \cite[Thm.~I.12.12]{N2} as before.
  The unit index is $2^s$, as we saw in Proposition~\ref{unit-index-t}, while
  the conductor is $\fp_3$ (Corollary~\ref{orders-three-zeta3}).
  So we have $\#(\Z_{K_n}/\f)^\times = (3^{2^{s-2}}-1)^2$.  Since every element of
  $T_n$ has antitrace in $\fp_3$ by Proposition~\ref{prop:atr-order},
  the image of $(T_n/\f)^\times$ in $(\Z_{K_n}/\f)^\times$ is the diagonal subgroup
  and hence has order $3^{2^{s-2}}-1$.
\end{proof}

\begin{prop}\label{prop: three class}
  Let $n=3\cdot2^s$ with $s \ge3$.  Then
\[
v_{<1}(\gr'_n) \le \big(2^{2^{s-1}-s+1}+2^{2^{s-1}}\big)h^-(\Q(\zeta_n)).
  \]
\end{prop}

\begin{proof}
  Set $h^{-}=h^{-}(\Q(\zeta_n))$.
Using Propositions~\ref{indices-three} and~\ref{prop:class T} we get
$$h^-(T_n) + h^-(\fo'_0)= h^-\cdot\bigg( \frac{3^{2^{s-2}} -1}{2^s} + 1\bigg)\, .$$
Now
$$\frac{3^{2^{s-2}} -1}{2^s} + 1 \le 2^{2^{s-1}}/2^s = 2^{2^{s-1}-s} \, .$$  
  By Proposition~\ref{prop: class number} we have
$$\sum_{k=1}^{2^{s-1}}2h^-\left(\fo'_k\right)  = h^-\sum_{k =
    1}^{2^{s-1}}2^{k-s+w'_k-1} \le h^-\sum_{k = 1}^{2^{s-1}}2^{k-s+s-1}
  < \, h^-2^{2^{s-1}}.$$
Combining these bounds with Proposition~\ref{mass-g three} completes
the proof.
\end{proof}

\subsubsection[Quadratic
  \texorpdfstring{$\Z_{F_\tts}$}{Z[\unichar{"03B6}3\unichar{"B7}2\unichar{"5E}s]+}-orders
    containing ramifying elements] 
    {Quadratic $\Z_{F_\tts}$-orders containing ramifying elements}
\label{Sec:ramify}

When $n = 3 \cdot 2^s$, by Theorem~\ref{thm:invert}, 
$\PUT\big(\Z_{K_\tts}^{(2)}\big)$ 
acts on $\Delta_n$ without inversions, so the graph
$\overline{\gr}_n$ does not contain half-edges.  However, the covering
$\gr_n\to\overline{\gr}_n$ does ramify at some vertices, so
$$v(\overline{\gr}_n)= v(\gr_n)/2 + v_r(\gr_n/\overline{\gr}_n)/2$$
where $v_r(\gr_n/\overline{\gr}_n)$ is the number of ramified
vertices.  This is in contrast to the case of $n = 2^s$.

By Theorem~\ref{thm:ram} a vertex in $\gr_n$ is ramified if and only if a
maximal order $\M \in \Ver(\Delta_n)$ covering it contains a ramifying
element $\gamma$. The norm of a ramifying element is a 
totally positive nonsquare unit.  By Theorem~\ref{fields},
there is only one such up to squares, namely:
\begin{equation}
  \label{billy}
  u_+:=\zeta+\zeta^{-1}+2=p_n'+1 .
  \end{equation}

We will now characterize possible extensions $F_n(\gamma)$.

\begin{theorem}\label{thm:ramifying}
Let $L$ be a totally complex quadratic extension of $F= F_n$ such that
there exists $\gamma\in \Z_L$ of norm $u_+$.  Then $L$ is isomorphic
to $K_n$ or to $F(\sqrt{-u_+}\,)$, and these fields are not isomorphic
to each other.
  \begin{enumerate}[\upshape (a)]
  \item
\label{thm:ramifying1}
    If $L \cong F(\sqrt{-u_+}\,)$, then
    $\gamma=\sqrt{-u_+}$ and $\Z_F[\gamma]=\Z_L$.
  \item
\label{thm:ramifying2}
    If $L \cong K_n$, then $\gamma=\zeta^{t+1}+\zeta^t$ for some $t$, and
    $\Z_F[\gamma]$ is either the maximal order $\Z_L = \fo'_0$ or 
    the unique order $\fo'_1$ with conductor $\fP$.
  \end{enumerate}
\end{theorem}
\begin{proof}
  Let $\eta=\gamma^2/u_+$.  Then $\eta$ is a unit of relative
  norm $1$ in the CM extension $L/F$, so $\eta$ is a root of unity.
  The two cases in the theorem are different, because $K_n = F(\sqrt{-1})$
  and $u_+$ is not a square in $F$.
  We now consider them separately:

  \begin{description}
    \item[{\bf Case 1. $\eta = \pm 1$}]  This means that $\gamma^2 = \pm u$.
      Since $L$ is a totally complex quadratic extension and $u_+$ is
      totally positive, it must be that $\gamma^2 = -u_+$: in other words,
      $L=F(\sqrt{-u_+})$.

      It remains to show that $\Z_F[\gamma]$ is maximal.  Consider
      $a + b\gamma \in \Z_L$: we must show that $a, b$ are integral.
      First, $2a = T_{L/F}(a+b\gamma) \in \Z_F$, so $a$ is integral away
      from $\p$; likewise $2(a+b\gamma)-2a$ is integral, so $b\gamma$ is
      integral away from $\p$.  Thus $b$ is as well, because $\gamma$ is
      a unit.

      Now suppose that $c, d \in \Z_F$, but not both are in $\p$, and
      consider $N_{L/F} (c+d\gamma) = c^2 + u_+d^2$.  If one of $c, d$ is in
      $\p$, then this expression is not in $\p$.  If neither $c$ nor $d$
      belongs to $\p$, then $c^2, d^2 \equiv 1 \bmod \p^2$, because $\p$
      has residue field of order $2$.  But $u_+ \not \equiv 1 \bmod \p^2$,
      so $c^2 + u_+d^2 \notin \p^2$.

      If $a + b\gamma \in \Z_L$, then write
      $a = c{p_n'}^k, b = d{p_n'}^k$,
      where $k \in \Z$ and not both $c$ and $d$ are in $\p$.  From the
      above, the $\p$-adic value of $N_{L/F}(a+b\gamma)$ is either
      $2k$ or $2k+1$, so $k \ge 0$ as desired.
  \item[{\bf Case 2. $\eta \ne \pm1$}] Since $\eta\in L$, we must
    have $L=K_n$ (recall that $K_n \equiv F_n(\zeta_r)$ for all
    $r>2$ dividing $s$).  Let $\eta=\zeta^k$.  Then $u_+\zeta^k=\gamma^2$;
    here $k$ must be odd, since $u_+$ is not a square.  Let
    $k=2t+1$ and $\gamma=\zeta^{t+1}+\zeta^t$, so that
    $N_{L/F} \gamma = N(\zeta^t) N(\zeta+1) = u_+$ as desired.
    Now again consider two cases:
      \begin{description}
      \item[{\bf Case 2(i). $3\nmid k$}]  In this case $k$ is a
        unit mod $n$: let $k'$ be its inverse.  Thus
        $\zeta=\eta^{k'}\in\Z_F[\gamma]$, so $\Z_F[\gamma]$ is maximal.
      \item[{\bf Case 2(ii). $3\mid k$}]  Then
        $t \equiv 1 \pmod{3}$ and there exists a $k'$ such that
        $kk' \equiv 3 \pmod{n}$.  Hence $\zeta^3=\eta^{k'}\in\Z_F[\gamma]$.
        It follows that
        $\zeta^2+\zeta = \gamma(\zeta^3)^{\frac{s+1-t}{3}} \in \Z_F[\gamma]$,
        and so we may take $t = 1$.  The antitrace of $\zeta$ is a unit, while
        the antitrace of $\zeta^2+\zeta$ generates $\fP \Z_L$.  It follows
        that the antitrace of any power of $\zeta^2+\zeta$ is in $\fP \Z_L$,
        so $\fP$ is the conductor of $\Z_F[\gamma]$.
      \end{description}
  \end{description}
\end{proof}

Set $h^-(\Z_F[\sqrt{-u_+}\,]) = h(\Z_F[\sqrt{-u_+}\,])/h(\Z_F)$.
\begin{corollary}\label{cor:ramifying}
  Let $n = 3\cdot 2^s$ with $s \ge 2$.  Then
  $$v_r(\gr_n/\overline{\gr}_n) \le 2h^-(\Z_{K_n}) +
 h^-(\Z_{F_n}[\sqrt{-u_+}\,]) \,.$$
\end{corollary}
\begin{proof}
Combine Theorems~\ref{thm:Mn} and~\ref{thm:ramifying} with   Proposition~\ref{prop: class number}.
\end{proof}

\section{The main theorem for
  \texorpdfstring{$n=2^s$}{n=2\unichar{"5E}s}}
\label{two}

\subsection{Class number bounds}
\label{subsec:bounds-first}

Throughout Section \ref{subsec:bounds-first} $n=2^s$ for $s\geq 3$.  
We have $h^+(K_n)=h(F_n)$ and $h(K_n)=h^+(K_n)h^-(K_n)$.
Recall that $p_n=2+\zeta_{n}+\zeta^{-1}_{n} = N_{K_n/F_n}(\zeta_{n}+1)$
is a totally positive element of norm $2$ generating the unique
ideal $\fp=\fp_n$ of $\Z_{F_n}$ above $2$.
Normally subscripts will not change in a single calculation, so
they will frequently be omitted.

In both Section \ref{subsec:bounds-first}   and Section \ref{kitty},
we bound certain class numbers from above.
This
is the key to bounding the first Betti numbers of the graphs $\gr_{2^s}$, 
$\overline{\gr}_{2^s}$, $\gr_{\tts}$, $\overline{\gr}_{\tts}$ from below.
Specifically, let $L$ be a CM-field which is abelian over $\Q$ with
associated Dirichlet characters $\{\chi\}$.
Let $E$ be the unit group of $L$, $E^+$ the unit group of its totally
real subfield $L^+$, $W$ the group of roots of unity in $L$, and $w=w(L)=\#W$.
Set $Q=Q(L)=[E:WE^+]$; it is always $1$ or $2$.
We have to bound $h^-(L)$ for certain such fields $L$.
In every case, we do so by applying the following technique, which we
learned from \cite{N}:
\begin{enumerate}[\upshape (a)]
\item
  \label{technique1}
Express $h^{-}$ in terms of a product of generalized Bernoulli numbers
$B_{1,\chi}$ \cite[Thm.~4.17]{W}:
\begin{equation}
  \label{classn}
  h^{-}(L)=Qw\prod_{\chi\mbox{ \footnotesize{odd}}}\Big(-\frac{1}{2}B_{1,\chi}\Big)
\end{equation}
with $B_{1,\chi} =
\sum_{i=1}^{f_\chi} i\chi(i)/f_\chi$.
\item
  \label{technique2}
Apply the arithmetic mean/geometric mean inequality to \eqref{classn}
to bound from above,
cf. \cite[Cor.~2
to Prop.~8.12]{N}. 

\end{enumerate}
A result such as \cite[Thm.~4.20]{W} is not sufficient for our purposes
because it is not effective.

These class numbers grow very rapidly.  They can be estimated from
below using \cite[Prop.~11.16]{W}.  For example, this
proposition shows that
\[
\log h^{-}(\Q(\zeta_{2^s})) \ge (s-1)2^{s-3}\log 2 - (1.08)\cdot 2^{s-1}
\]
for $s > 8$.
With $s = 9$ the two sides are approximately $126.2$ and $78.4$, so the bound
is not terrible.

Schrutka von Rechtenstamm \cite{Sch}
computed $h^{-}(\Q(\zeta_{n}))$ for $\phi(n)\leq 256$; the table of values
is reproduced in \cite[p. 412]{W}. Computed values for the 
$n=2^s$ family are given in Figure~\ref{minus2} below and for the $n=3\cdot 2^s$
family in Figure~\ref{minus3}.

\begin{figure}[ht]
\begin{center}
\label{minus2}
\begin{tabular}{l|l}
$s$ & $h^-(\Q(\zeta_{2^s}))$\\
      \hline
     $ 4$ & $1$\\
      $5$ & $1$\\
     $ 6$ & $17$\\
     $ 7$ & $359057=17\cdot 21121$\\
     $ 8$ & $10\,449592\,865393\,414737=17\cdot 21121\cdot
      29\,102880\,226241$\\
      $ 9$ & $ 6\,262503\,984490\,932358\,745721\,482528\,922841\,978219\,389975\,605329$\\
      &$=17\cdot21121\cdot76\,532353\cdot29\,102880\,226241$\\
   &   $\cdot
      7830\,753969\,553468\,937988\,617089$
\end{tabular}
\end{center}
\caption{The minus part of the class number for the cyclotomic $n=2^s$ family}
\end{figure}

\begin{theorem}
  \label{class1}
  Let $n=2^s$ with $s\geq 3$.
  \begin{enumerate}[\upshape (a)]
  \item
    \label{class11}
   $h^-(K_n)\leq 2^{s+(s-4)2^{s-3}}/3^{2^{s-3}}=
    2^{s + (s - 4 - \log_2 3)2^{s-3}}.$
  \item
    \label{class12}
$h^-(F_n(\sqrt{-3}))\leq
    3^{1-2^{s-2}} 2^{1+s\cdot2^{s-3}}=2^{1+\log_{2}3+(s-2\log_23)2^{s-3}}$. 
  \item
    \label{class13}
    $h^-(F_n(\sqrt{-p_n}))\leq 2^{1-2^{s-1}+s\cdot2^{s-3}}=2^{1+(s-4)2^{s-3}}$.
    \end{enumerate}
\end{theorem}
\begin{proof}
 \eqref{class11}: For $K_n$  we have $Q = 1$ by \cite[Cor.~4.13]{W} 
and $w = 2^s$.
The product \eqref{classn}  runs over odd characters $\chi$ of conductor
  $2^s$, and $B_{1,\chi} = \sum_{i=1}^{2^s} i\chi(i)/2^s$.
 Gathering all the powers of $2$, we have
 \begin{equation}
   \label{h_}
h^-(K_n) = 2^{s - (s+1)\cdot 2^{s-2}} 
\bigg(\prod_\chi \bigg| \sum_{i=1}^{2^s} i\chi(i) \bigg|^2 \bigg)^{1/2}.
\end{equation}
Now apply the arithmetic mean/geometric mean inequality to \eqref{h_}:
\begin{align*}
    h^-(K_n)&\le 2^{s - (s+1)\cdot 2^{s-2}}
    \Bigg(\frac{\sum_\chi \Big|\sum_{i=1}^{2^s} i\chi(i)
\Big|^2}{2^{s-2}} \Bigg)^{2^{s-3}} \\
\nonumber &= 2^{s - (s+1)\cdot 2^{s-2}} \left( 2^{3s-2}/3 - 2^{s+1}/6\right)^{2^{s-3}}\\
\nonumber &< 2^{s - (s+1)\cdot 2^{s-2}}\left( 2^{3s-2}/3\right)^{2^{s-3}}
= 2^{s + (s - 4 - \log_2 3)2^{s-3}}.
\end{align*}

The only nontrivial step here is the one that evaluates the sum of the squares
of absolute values of the Bernoulli numbers.  
We have
\begin{equation}
  \label{dated}
\sum_\chi \Big|\sum i\chi(i)\Big|^2 = \sum_\chi \sum_{i,j} ij\chi(i) \bar\chi(j)
= \sum_{i,j} \sum_\chi ij \chi(i) \bar\chi(j).
\end{equation}

Now, for $i = j$ we have $\chi(i) \bar\chi(j) = 1$ for all $\chi$, and for $i = 
2^s-j$ we have $\chi(i) \bar\chi(j) = -1$ since our characters are odd.
For other values of $i,j$, choose an even character $\epsilon$ such that
$\epsilon(i/j) \ne 1$; then 
$$\sum_\chi \chi(i) \bar\chi(j) = 
\sum_\chi \chi\epsilon(i) \bar{\chi\epsilon}(j) =
\epsilon(i/j)\sum_\chi \chi(i) \bar\chi(j),$$
so the sum is $0$.  Since there are $2^{s-2}$ characters,
the sum \eqref{dated} reduces to
\begin{equation}
  \label{cyc 2 chi sum}
  2^{s-2}
\sum_{i=1\atop
 i\mbox{\tiny{ odd}}}^{2^s} i^2 - i(2^s-i)=2^{s-2}\left(
2^{3s-2}/3-2^s/3\right)
\end{equation}
by induction, proving Theorem~\ref{class1}\eqref{class11}.\\
\eqref{class12}: Let $L_n=F_n(\sqrt{-3})$.  Then $Q=1$ and $w=6$.
Note that we now use only the characters belonging to $L_n$, in other words
those of conductor divisible by $3$, rather than all odd characters.
From \eqref{classn}
we have
\begin{align*}
  h^-(L_n) &= 6\prod_{\chi \text{ odd}}\Big( \frac{-1}{2}B_{1,\chi}\Big)
 =6\cdot 2^{-2^{s-2}}
  \prod_{\chi}| B_{1,\chi}|\\
 & =6\cdot 2^{-2^{s-2}}\bigg(\prod_{\chi}\bigg| \sum_{i=1}^{f_{\chi}}i\chi(i)/f_{\chi}
  \bigg|^2\bigg)^{1/2}\\
  &=6\cdot 2^{-2^{s-2}}\bigg(\prod_{\chi}\bigg|
  \sum_{i=1}^{3\cdot2^s}i\chi(i)/(3\cdot2^s)
  \bigg|^2\bigg)^{1/2}\\
  & =3^{1-2^{s-2}} 2^{1-(s+1)2^{s-2}}\bigg(\prod_{\chi}\bigg| \sum_{i=1}^{3\cdot2^s}i\chi(i)
  \bigg|^2\bigg)^{1/2}\\
  &\le3^{1-2^{s-2}}
  2^{1-(s+1)2^{s-2}}\Bigg(\frac{\sum_{\chi}\big|\sum_{i=1}^{3\cdot2^s}
    i\chi(i)\big|^2}{2^{s-2}}\Bigg)^{2^{s-3}}\\
  &= 3^{1-2^{s-2}} 2^{1-(s+1)2^{s-2}} \bigg(\frac{2^{s-2}(2\cdot 2^{3s} + 
(-1)^{s+1}\cdot 4 \cdot 2^{2s})-3\cdot 2^{2s}}{2^{s-2}}\bigg)^{2^{s-3}}\\
  &= 3^{1-2^{s-2}} 2^{1-(s+1)2^{s-2}}\big(2 \cdot 2^{3s}+
(-1)^{s+1}\cdot 4 \cdot 2^{2s}-3\cdot 2^{s+2}\big)^{2^{s-3}}\\
  &\le 3^{1-2^{s-2}} 2^{1-(s+1)2^{s-2}}(2 \cdot 2^{3s}+ 2^{2s+2})^{2^{s-3}}\\
  &\le 3^{1-2^{s-2}} 2^{1-(s+1)2^{s-2}+(3s+2)2^{s-3}}
  =3^{1-2^{s-2}} 2^{1+s\cdot2^{s-3}},\\
  \end{align*}
\noindent where for simplicity
we have used a very poor approximation $2^{2s+2} < 2^{3s+1}$ for $s>1$
 in the last line.

To perform this calculation,
we  evaluate the sum of the squares of absolute values
of the Bernoulli numbers as in \eqref{dated}.
Note that the conductors of the characters are all of the form
$3\cdot 2^k$ for some $k$.  Thus all the characters vanish on $i$ and
$j$'s not relatively prime to $6$ except for $\chi_3$.

Now, as before for $i = j$ we have $\chi(i) \bar\chi(j) = 1$ for all
$\chi$, and for $i = 3\cdot 2^s-j$ we have $\chi(i) \bar\chi(j) = -1$ since
our characters are odd.   However, there are other nonzero values.
Namely, for $i\equiv j\pmod{2^s}$ with $i\ne j$ we have $\chi(i)
\bar\chi(j) = -1$, and for $i\equiv -j\pmod{2^s}$ with $i\ne
3\cdot2^s-j$ we have $\chi(i)
\bar\chi(j) = 1$.  For other values of $i,j$  relatively prime to $6$
we still have
$\sum_\chi \chi(i) \bar\chi(j) = 0$.
Since there are $2^{s-2}$ characters,
the sum \eqref{dated} expands to
\begin{equation}
\label{alumni}
\begin{split} 
  2^{s-2} \Bigg(\sum_{i=1\atop (i,6)=1}^{3\cdot 2^s} i^2\quad \!\!
+&\quad\!\! (-1)^{s+1} \!\!\!\!\!\!\!\!\!\!\!\!\!\!\!\!\!\sum_{{i,j=1\atop (i,6)=1}\atop j \equiv (2^s+(-1)^s)i \bmod 3\cdot 2^s}^{3\cdot 2^s}\!\!\!\!\!\!\!\!\!\!\!\!\!\!\!ij 
\quad\!\! - \quad\!\!
\sum_{i=1\atop (i,6)=1}^{3 \cdot 2^s} i(3 \cdot 2^s-i) 
\quad \!\! +\quad \!\!
(-1)^s \!\!\!\!\!\!\!\!\!\!\!\!\!\!\!\!\!\!\!\!\!\sum_{{i,j=1\atop (i,6)=1}\atop j \equiv (2^{s+1}+(-1)^{s+1})i \bmod 3\cdot 2^s}^{3\cdot 2^s}\!\!\!\!\!\!\!\!\!\!\!\!\!\!\!\!\!\!\!\!ij \Bigg) + c_3,
\end{split}
\end{equation}
where 
\[
c_3 = \sum_{i,j=1\atop{i,j \mbox{\tiny{ even}}}}^{3\cdot 2^s} ij \chi_3(ij)
+2\sum_{i,j=1\atop{i \mbox{\tiny{ even}}, j \mbox{\tiny{ odd}}}}^{3\cdot 2^s} ij \chi_3(ij)
\]
is a correction term that arises from evaluating $\chi_3$ at even arguments.
(All the other characters have conductor a multiple of $6$, so their values
at even integers are $0$.)
Let $G_n$ be the first sum in the parentheses in \eqref{alumni}.  
One readily checks that
$G_n = 3\cdot 2^{3s} + 2^s$ and that the second term is
$G_n/2-\frac{3}{2}\cdot 2^s$.
The third and fourth sums are evaluated by breaking up the range for $i$
according to the congruence class of $i \bmod 3$ and $\lfloor i/2^s \rfloor$.
For example, in the third sum, with $s$ even and $i \equiv 1 \bmod 3$, we have
$j = 2^s+i, 2^s+i, -2(2^s)+i$ for $\lfloor i/2^s \rfloor = 0, 1, 2$
respectively.  This reduces the evaluation of these sums to sums of squares
and arithmetic progressions.  When $s$ is even, we obtain
$G_n + 2^{2s}(2-2^s) = 2 \cdot 2^{3s} + 2 \cdot 2^{2s} + 2^s$ and
$\frac{5}{2} 2^{3s} - 2 \cdot 2^{2s} - 2^s$;
when $s$ is odd, the results are
$\frac{5}{2}\cdot 2^{3s}+2\cdot 2^{2s}-2^s$ and $2\cdot 2^{3s}-2\cdot2^{2s}+2^s$.
Whether $s$ is even or odd, the sum of the four terms in parentheses in
\eqref{alumni}
works out to $2 \cdot 2^{3s} + (-1)^{s+1} \cdot 4 \cdot 2^{2s}$.
It is easily checked that $c_3 = -3 \cdot 2^{2s}$, which proves the statement.

\eqref{class13}:  Put $M_n=F_{n}(\sqrt{-p_n})$.
Note that $M_n\subseteq K_{2^{s+1}}$ since
$(\zeta_{2^{s+1}}+\overline{\zeta}_{2^{s+1}})^2=p_n$. The characters $\chi$
associated to $M_n$ are the characters $\chi:(\Z/2^{s+1}\Z)^\times\rightarrow
\BC^\times$ satisfying $\chi(2^s-1)=1$ and $\chi(-1) = -1$.
For $M_n$ we have $w=2$ and
$Q=1$.  Hence from \eqref{classn}
\begin{align*}
  h^-(M_n) &= 2\prod_{\chi \text{ odd}}\Big( \frac{-1}{2}B_{1,\chi}\Big)=
  2\cdot 2^{-2^{s-2}}
  \prod_{\chi}\left| B_{1,\chi}\right|\\
  &=2\cdot 2^{-2^{s-2}}\bigg(\prod_{\chi}\Big| \sum_{i=1}^{f_{\chi}}i\chi(i)/f_{\chi}
  \Big|^2\bigg)^{1/2}\\
 & =2^{1-2^{s-2}}\bigg(\prod_{\chi}\Big| \sum_{i=1}^{2^{s+1}}i\chi(i)/2^{s+1})
  \Big|^2\bigg)^{1/2}\\
  &=2^{1-2^{s-2}} 2^{-(s+1)2^{s-2}}\bigg(\prod_{\chi}\Big| \sum_{i=1}^{2^{s+1}}i\chi(i)
  \Big|^2\bigg)^{1/2}\\
  &\le 2^{1-(2+s)2^{s-2}}
  \Bigg(\frac{\sum_{\chi}\big|\sum_{i=1}^{2^{s+1}}i\chi(i)\big|^2}{2^{s-2}}
\Bigg)^{2^{s-3}}\\
  &=2^{1-2^{s-1}-s\cdot2^{s-2}-s\cdot2^{s-3}+2^{s-2}}(2^{4s-2})^{2^{s-3}}\\
  &=2^{1-2^{s-2}-3s\cdot 2^{s-3}}2^{4s\cdot 2^{s-3}-2^{s-2}}
  =2^{1-2^{s-1}+s\cdot2^{s-3}}.
  \end{align*}

Again we use \eqref{dated}.
For $i = j$ we have $\chi(i) \bar\chi(j) = 1$ for all $\chi$, and for
$i = 2^{s+1}-j$ we have $\chi(i) \bar\chi(j) = -1$ since our
characters are odd.  For other odd values of $i,j$ the sum is $0$, except that
for $i\equiv j+2^{s}\pmod{2^{s+1}}$ we have $\chi(i) \bar\chi(j) = -1$,
and for $i\equiv 2^s-j\pmod{2^{s+1}}$ we have $\chi(i) \bar\chi(j) = 1$.
Since there are $2^{s-2}$ characters,
\eqref{dated} in this case expands to
\begin{align*}
&2^{s-2}\bigg(
\sum_{i=1\atop
 i\mbox{\tiny{ odd}}}^{2^{s+1}}(i^2 - i(2^{s+1}-i)) + \sum_{i=1\atop
  i\mbox{\tiny{ odd}}}^{2^{s}}i(2^s-i) + \\
&\qquad\qquad\qquad\qquad\sum_{i=2^s+1\atop
 i\mbox{\tiny{ odd}}}^{2^{s+1}}i(3\cdot2^s-i)-2\sum_{i=1\atop
  i\mbox{\tiny{ odd}}}^{2^{s}}i(2^s+i) \bigg)\\
&\mbox{}=2^{s-2}\big( 
(2^{3s+1}-2^{s+1})/3 + (2^{3s-2}+2^{s-1})/3 +\\
&\qquad\qquad\qquad
(13\cdot 2^{3s-2}+2^{s-1})/3 - 2(5\cdot 2^{3s-2}-2^{s-1})/3 \big) \\
&=2^{s-2}2^{3s}=2^{4s-2}.\qedhere
\end{align*}
\end{proof}

\subsection
{\texorpdfstring{The main theorem for \except{toc}
{\boldmath{$\PSUT\big(\Z_{K_{2^s}}^{(2)}=\Z[\zeta_{2^s},1/2]\big)$}}
\for{toc}{$\PSUT\big(\Z_{K_{2^s}}^{(2)}=\Z[\zeta_{2^s},1/2]\big)$}}{The main theorem for
                           PSU2(Z[\unichar{"03B6}2\unichar{"5E}s,1/2])}}
\label{rubbish1}

We first bound 
\[
E_{1,2^s}=E\big(\PSUT\big(Z_{K_{2^s}}^{(2)}\big),\Delta\big)
\]
as in
Section \ref{fern} from above.
\begin{align}
\nonumber
E_{1,2^s}
 &\leq v_{<1}(\gr_s)
\text{ by \eqref{shade1}}\\
\nonumber
&\leq 2h^{-}(F_{2^s}(\sqrt{-3}))+2^{2^{s-2}+1}h^{-}(K_{2^s})
\text{ by Prop.~\ref{prop: roots bound}}\\
\nonumber
&\leq 2^{2+s2^{s-3}}3^{1-2^{s-2}}+2^{1+s+(s-2)2^{s-3}}3^{-2^{s-3}}
\text{ by Thm.~\ref{class1}}\\
\nonumber
&=2^{(s-2\log_23)2^{s-3}}\big(12+2^{1+s-(2-\log_23)2^{s-3}}\big)\\
\nonumber
&\leq 2^{(s-3.2)2^{s-3}}\big(12+2^{1+s-(.4)2^{s-3}}\big)\\
\nonumber
&\leq 2^{(s-3.2)2^{s-3}}(12+2^5)
=44\cdot 2^{(s-3.2)2^{s-3}}\\
& <2^{5.5+(s-3.2)2^{s-3}}\label{green}
\end{align}
since $1+s-(0.4)2^{s-3}\leq 5$ .

We now bound $b_1(\gr_{2^s})$ from below using \eqref{green} and the bounds
on $M_{2^s}$ in Theorem \ref{soup}\eqref{soup1}.
\begin{align}
\label{eq:bound-psu2-2s}
\nonumber
  b_1(\gr_{2^s}) &= 1+ M_{2^s}-E_{1,2^s}
\text{ from Thm.~\ref{masst}\eqref{founded}}\\
\nonumber &>1+2^{(3s-13.7)2^{s-3}-1/2}-2^{5.5+(s-3.2)2^{s-3}}\\   
 &= 1+2^{(s-3.2)2^{s-3}}\big( 2^{(2s-10.5)2^{s-3}-1/2}-2^{5.5}\big).
\end{align}
From \eqref{eq:bound-psu2-2s} we see that $b_1(\gr_{2^s})$ grows as 
$2^{\Omega(s2^s)}$. We know that $b_1(\gr_{2^s})=2^{O(s2^s)}$ from
Proposition~\ref{upper}.  Hence $b_1(\gr_{2^s})=2^{\Theta(s2^s)}$, as claimed.
In addition, if 
$s\geq 6$, then 
\[
(2s-10.5)2^{s-3}-1/2>5.5
\]
 and so $b_1(\gr_{2^s})>1$.
For $s=5$ the bound that follows from this argument is not strong enough to
show that $b_1(\gr_{2^s}) > 1$, but a direct computation 
using \cite[Sect.~7.3]{IJKLZ2} shows that
$b_1(\gr_{32})=40$.
This completes the
proof of Main Theorem~\ref{main} for $\PSU_2\big(\Z_{K_{2^s}}^{(2)}\big)$. 

\subsection
{\texorpdfstring{The main theorem for 
\except{toc}{\boldmath{$\PUT\big(\Z_{K_{2^s}}^{(2)}=\Z[\zeta_{2^s},1/2]\big)$}}
\for{toc}{$\PUT\big(\Z_{K_{2^s}}^{(2)}
=\Z[\zeta_{2^s},1/2]\big)$}}{The main theorem for 
                 PU2(Z[\unichar{"03B6}2\unichar{"5E}s,1/2])}}
\label{rubbish2}

We will need a bound on the number $e_h(\overline{\gr}_{2^s})$
of half-edges in $\overline{\gr}_{2^s}$:
\begin{proposition}
  \label{cor: upshot}
We have 
\begin{align*}
e_h(\overline{\gr}_{2^s})  &\leq
2^{s + (s - 4 - \log_2 3)2^{s-3}} +  2^{1+(s-4)s2^{s-3}} = 2^{(s-4)2^{s-3}}\big(
2^s/3^{2^{s-3}}+2\big)\leq 5\cdot 2^{(s-4)2^{s-3}}.
\end{align*}
\end{proposition}
\begin{proof}
  For the first inequality,
  combine Theorem \ref{class1} with the class number bounds
  in Theorem \ref{ponds}.  For the second, use the easily checked fact that
  $2^s/3^{2^{s-3}}\leq 3$.
\end{proof}

Using Proposition \ref{cor: upshot}, we can
bound $E_{0,2^s}=E\big(\PUT\big(\Z_{K_{2^s}}^{(2)}\big),\Delta\big)$ from above for $s\geq 3$:
\begin{align}
\label{lunch}
\nonumber
E_{0,2^s}&=(E_{1,2^s}+
e_h(\ogr_{2^s}))/2\text{ by \eqref{dealt2}}\\
\nonumber
&< \big( 2^{5.5+(s-3.2)2^{s-3}}+5\cdot 2^{(s-4)2^{s-3}}\big)/2\text{ by \eqref{green}
and Prop.~\ref{cor: upshot}}\\
\nonumber
&=2^{(s-4)2^{s-3}}\big( 2^{5.5-0.8\cdot 2^{s-3}}+5\big)/2\\
&<2^{4+(s-4)2^{s-3}}
\end{align}
since $2^{5.5-0.8\cdot 2^{s-3}}+5<32$ for $s\geq 3$.

We now use this to bound $b_1(\ogr_{2^s})$ from below:
\begin{align}
\label{punted}
\nonumber
 b_1(\ogr_{2^s})&=1+M_{2^s}/2-E_{0,2^s}
\text{ by Thm.~\ref{masst}\eqref{founded}}\\
\nonumber
 &> 1+2^{(3s-13.7)2^{s-3}-3/2}-2^{4+(s-4)2^{s-3}}\text{ by Thm.~\ref{soup}\eqref{soup1}
and \eqref{lunch}}\\
&=1+2^{(s-4)2^{s-3}}\big(2^{(2s-9.7)2^{s-3}-3/2}-2^4\big)\quad\text{for $s\geq 3$}.
\end{align}

We see that $b_1(\overline{\gr}_{2^s})$ grows as $2^{\Omega(s2^s)}$ as $s \to \infty$.
Combining this with Proposition~\ref{upper} then gives 
$b_1(\overline{\gr}_{2^s})=2^{\Theta(s2^s)}$.
Further, $(2s-9.7)2^{s-3}-3/2> 4$ if $s\geq 6$,
so $b_1(\overline{\gr}_{2^s})>1$ if $s\geq 6$.

For $s=5$ the bound coming from this argument is not sufficiently
strong, but the direct computation 
of \cite{IJKLZ2}*{Sect.~7.3} shows that
$b_1(\overline{\gr}_{32}) = 16$.  This completes the proof of Main Theorem~\ref{main}
for $\PUT\big(\Z_{K_{2^s}}^{(2)}\big)$.

\subsection{Summary of the cyclotomy bounds at $\fp$ for $\mathbf{n=2^s}$}
\label{deaf}

Below we summarize the cyclotomy bounds for the $\tts$ family
in this section.
\begin{theorem}
\label{main2b}
Suppose $n=2^s$ with $n\geq 8$.
\begin{align*}
\hspace*{-1.2in}\textup{(a)}\hspace*{2in} E_{1,2^s}&< 2^{5.5+(s-3.2)2^{s-3}}\\
E_{0,2^s}&< 2^{4+(s-4)2^{s-3}}\quad\text{if $s\geq 3$}\\[.1in]
\hspace*{-1in}\textup{(b)}\hspace*{1.77in} b_1(\gr_{2^s})&>1+2^{(s-3.2)2^{s-3}}\big( 2^{(2s-10.5)2^{s-3}-1/2}-2^{5.5}\big)\\
b_1(\ogr_{2^s})&>1+2^{(s-4)2^{s-3}}\big(2^{(2s-9.7)2^{s-3}-3/2}-2^4\big)\quad
\text{if $s\geq 3$}
\end{align*}
\end{theorem}

\section{The main theorem for
  \texorpdfstring{$n=3\cdot 2^s$}{n=3\unichar{"B7}2\unichar{"5E}s}}
\label{three}
\subsection{Class number bounds}
\label{kitty}

          Throughout Section~\ref{kitty}
          we let $n=3\cdot 2^s$ for $s\geq 3$.
        The unit $u_{+}:=2+\zeta_{3\cdot 2^s}+\zeta_{3\cdot 2^s}^{-1}$
          was defined in \eqref{billy}. We have
$h^+(K_n)=h(F_n)$ and $h(K_n)=h^+(K_n)h^-(K_n)$.
Recall that $p_n'=1+\zeta_{3\cdot 2^s}+\zeta^{-1}_{3\cdot 2^s}$
is an element of norm $-2$ generating the unique
ideal $\fp=\fp_n$ of $\Z_{F_n}$ above $2$, and that
$u_{+}\in\Z_{F_n}$  is a totally positive nonsquare unit.

We require two class number bounds to prove the Main Theorem~\ref{main}
in the case $n=3\cdot 2^s$, $s\geq 3$. These class numbers are known
to grow quickly, as we illustrate with known values of $h^{-}(K_n)$.
Again we can prove a lower bound using \cite[Prop.~11.16]{W}:
this time it is 
$$\log h^-(K_\tts) \ge 2^{s-1}(\log 3 + (s-1) \log 2)/4 
- 1.08 \cdot 2^s - 5 \cdot 2^{s-1}/24.$$
  Though this bound is not as
precise as the one in Section~\ref{subsec:bounds-first}, it is certainly
doubly exponential.
\begin{figure}[ht]
\begin{center}
\label{minus3}
\begin{tabular}{l|l}
$s$ & $h^-(\Q(\zeta_{3\cdot 2^s}))$\\
      \hline
     $ 3$ & $1$\\
      $4$ & $1$\\
     $ 5$ & $9=3^2$\\
     $ 6$ & $61353=3^2\cdot 17\cdot 401$\\
     $ 7$ & $107878\,055185\,500777=3^2\cdot 17\cdot 401\cdot 1697\cdot 21121\cdot
      49057$\\
     $ 8$ & $1067\,969144\,915565\,716868\,049522\,568978\,331378\,093561\,484521$\\
       & $ =3^2\cdot 17\cdot 401\cdot 1697\cdot 13313 \cdot 21121\cdot
      49057\cdot 175361\cdot 198593$\\
      & $\cdot 733697\cdot29\,102880\,226241$
 \end{tabular}
\end{center}
\caption{The minus part of the class number for the cyclotomic $n=\Th$ family}
\end{figure}

\begin{theorem}
  \label{class2}
  Let $n=3\cdot 2^s$ with $s\geq 3$.
\begin{enumerate}[\upshape (a)]
  \item
    \label{class21}
  $h^-(K_n)\leq 3^{1-2^{s-1}}2^{s+1+(s-1)2^{s-2}}=
    2^{s+1+\log_23 +(s-1-2\log_23)2^{s-2}}.$
  \item
    \label{class22}
    $h^-\big(F_n(\sqrt{-u_+})\big)\leq 2^{(s-7/2)2^{s-2}+2}$. 
    \end{enumerate}
\end{theorem}
\begin{proof}
The method of proof is the same as that used to prove Theorem~\ref{class1}.\\
 \eqref{class21}:  For $K_n$ we have $Q=2$ and $w=3\cdot2^s$.  From \eqref{classn}
we have
\begin{align*}
 h^-(K_n)& = 3\cdot2^{s+1}\prod_{\chi \text{ odd}}\Big( \frac{-1}{2}B_{1,\chi}\Big)=
  3\cdot2^{s+1}2^{-2^{s-1}}
  \prod_{\chi}\left| B_{1,\chi}\right|\\
  &=3\cdot2^{s+1-2^{s-1}}\bigg(\prod_{\chi}\Big| \sum_{i=1}^{f_{\chi}}i\chi(i)/f_{\chi}
  \Big|^2\bigg)^{1/2}\\
  &=3\cdot 2^{s+1-2^{s-1}}\bigg(\prod_{\chi}\Big| \sum_{i=1}^{3\cdot2^s}i\chi(i)/(3\cdot2^s)
  \Big|^2\bigg)^{1/2}\\
 &=3^{1-2^{s-1}} 2^{s+1-(s+1)2^{s-1}}\bigg(\prod_{\chi}\Big| \sum_{i=1}^{3\cdot2^s}i\chi(i)
  \Big|^2\bigg)^{1/2}\\
&\le3^{1-2^{s-1}}
  2^{s+1-(s+1)2^{s-1}}\Bigg(\frac{\sum_{\chi}\Big|\sum_{i=1}^{3\cdot2^s}i\chi(i)\Big|^2}{2^{s-1}}\Bigg)^{2^{s-2}}\\
\,\,\,\le 3^{1-2^{s-1}}&2^{s+1-(s+1)2^{s-1}}\Big(\frac{3\cdot 2^{s-2}\cdot(2^{3s-2}-2^s)+2^{2s-1}(2^{2s}+2^{s+1}-6)}{2^{s-1}}\Big)^{2^{s-2}}\\
 &=3^{1-2^{s-1}} 2^{s+1-(s+1)2^{s-1}}(11 \cdot 2^{3s-3}+2^{2s+1}-15\cdot 2^{s-1})^{2^{s-2}}\\
 &\le3^{1-2^{s-1}} 2^{s+1-(s+1)2^{s-1}} \left(2^{3s+1}\right)^{2^{s-2}} = 3^{1-2^{s-1}}2^{s+1+(s-1)2^{s-2}}.\\
\end{align*}

However, this time we split the characters
into two sets based on whether $3$ divides the conductor of the character:
\begin{align*}
\sum_\chi \big|\sum i\chi(i)\big|^2 &= \sum_\chi \sum_{i,j} ij\chi(i) \bar\chi(j)
= \sum_{3\nmid f_\chi} \sum_{i,j=1}^{3\cdot2^s} ij\chi(i) \bar\chi(j) + \sum_{3\mid
  f_\chi} \sum_{i,j=1}^{3\cdot2^s} ij\chi(i) \bar\chi(j)\\
&=9\sum_{3\nmid f_\chi} \sum_{i,j=1}^{2^s} ij\chi(i) \bar\chi(j) + \sum_{3\mid
  f_\chi} \sum_{i,j=1}^{3\cdot2^s} ij\chi(i) \bar\chi(j).
\end{align*}
These two character sums were already evaluated above: the first in
\eqref{cyc 2 chi sum} and the second in the course of proving part 2 of
Theorem~\ref{class1}.  
In particular, they are equal to
$$2^{s-2} \cdot \Big( \frac{2^{3s-2}-2^s}{3}\Big), \quad 2^{2s-1}\big(2^{2s}+(-2)^{s+1}-6\big).$$

 \eqref{class22} Let $L_n=F_n(\sqrt{-u_+})$ and note that
  $L_n\subseteq K_{3\cdot 2^{s+1}}$, fixed by $\sigma\in\Gal(K_{3\cdot 2^{s+1}}/\Q)$
  with
  $\sigma: \zeta_{3\cdot 2^{s+1}}\mapsto \zeta_{3\cdot 2^{s+1}}^{3\cdot 2^s-1}$.
  We claim that $w(L_n)=2$; if either $\zeta_3$ or $\zeta_4$ belonged
  to $L_n$ they would be fixed by $\sigma$, but they are not since
  $3\cdot 2^s-1$ is not congruent to $1$ mod $3$ or $4$.
  To see that $Q=Q(L_n)=2$, note that $i(\zeta_{3\cdot 2^s}+\zeta_{3\cdot 2^s}^{-1})$ is a nonreal unit of $L_n$, and, since $W(L_n)=\langle\pm 1\rangle$, it cannot
  be the product of a real unit and a root of unity.

  The $2^{s-1}$ odd characters corresponding to $L_n$ all satisfy
  $\chi(-1) = -1 \ne \chi(3 \cdot 2^s-1)$, so their conductors cannot
  divide $3 \cdot 2^s$.  Thus there are $2^{s-2}$ each of conductor $2^{s+1}$
  and $3 \cdot 2^{s+1}$.  The product of these conductors is
  $3^{2^{s-2}} \cdot 2^{(s+1)\cdot 2^{s-1}}$.
  We note that with $f=f_{\chi}$,  $\sum_{i=1}^f
  \left|i\chi(i)\right|^2$ is equal to $f^3/9+f/3$ if $f=3\cdot 2^{k}$
  and to $(f^3-f)/6$ if $f=2^k$.

  We use this to estimate $\sum_{\chi}\big|\sum_{i=1}^{f_\chi}i\chi(i)\big|^2
  \leq \sum_{\chi}\sum_{i=1}^{f_\chi}|i\chi(i)|^2$.
  The characters whose conductors are not multiples of $3$ give
  \begin{equation*}
2^{4s-2},
\end{equation*}
while those with $3|f$ give
\begin{equation*}
2^{s-2}\big(3\cdot 2^{3(s+1)}+(-1)^s\cdot 3\cdot 2^{2s+4}\big).
\end{equation*}
The total is bounded above by 
\begin{equation*}
  (1+7)\cdot 2^{s-3} \cdot 2^{3(s+1)} = 2^{4s+3}.
  \end{equation*}

We now put everything together.  We need an upper bound
for
\[
h^-(L_n)=Qw\prod_{\chi}\Big(-\frac{B_{1,\chi}}{2}\Big)=
\frac{\prod_{\chi}B_{1,\chi}}{2^{2^s-2}}.
\]
Recalling $B_{1,\chi}=\sum_{i=1}^{f_{\chi}}i\chi(i)/f_{\chi}$ and using our
determination of the product of the conductors above, we rewrite this as
\begin{equation}
  \label{eagle}
  \frac{\prod_{\chi}\sum_{i=1}^{f_\chi}i\chi(i)}
  {3^{2^{s-2}} \cdot 2^{(s+1)\cdot 2^{s-1}} \cdot 2^{2^s-2}}.
\end{equation}
To estimate the numerator of \eqref{eagle}, note that 
\begin{align*}
  \prod_{\chi}\sum_{i=1}^{f_\chi}i\chi(i) &\leq \Big(\prod_{\chi}\sum
  _{i=1}^{f_{\chi}}\left| i\chi(i)\right|^2\Big)^{1/2}
  \leq \bigg(\frac{\sum_{\chi}\sum_{i=1}^{f_\chi}
    | i\chi(i)|^2}{2^{s-1}}\bigg)^{2^{s-2}}\\
&\leq\Big(\frac{2^{4s+3}}{2^{s-1}}\Big)^{2^{s-2}}=2^{(3s+4)\cdot 2^{s-2}}.
\end{align*}
On the other hand, the denominator of \eqref{eagle} is
\[
3^{2^{s-2}} \cdot 2^{(s+3) \cdot 2^{s-1}-2} > 2^{(2s+15/2)2^{s-2}-2},
  \]
  where we have taken $3>2^{3/2}$.  We conclude that
  $h^{-}(L_n)\leq 2^{(s-7/2)2^{s-2}+2}$.

\end{proof}

\subsection
{\texorpdfstring{The main theorem for 
\except{toc}{\boldmath{$\PSUT\big(\Z_{K_\tts}^{(2)}=\Z[\zeta_{3\cdot 2^s},1/2]\big)$}}
\for{toc}{$\PSUT\big(\Z_{K_\tts}^{(2)}=\Z[\zeta_{3\cdot 2^s},1/2]\big)$}}
{The main theorem for 
                 PSU2(Z[\unichar{"03B6}3\unichar{"B7}2\unichar{"5E}s,1/2])}}
\label{rubber1}

As in the $2^s$ case, our next task is to bound 
$E_{1,\tts}=E\big(\PSUT\big(\Z_{K_\tts}^{(2)}\big),\Delta\big)\big)$.
\begin{align}
\label{west1}
\nonumber
E_{1,\tts}&\leq v_{<1}(\gr_{\Th})
\text{ by \eqref{shade1}}\\
&\leq\big(2^{2^{s-1}-s+1}+2^{2^{s-1}}\big)h^-(K_\tts)\text{ by
Prop.~\ref{prop: three class}}\\
\nonumber
&\leq \big(2^{2^{s-1}-s+1}+2^{2^{s-1}}\big)2^{s+1+\log_23+(s-1-2\log_23)2^{s-2}}\text{ by
Thm.~\ref{class2}\eqref{class21}}\\
\nonumber
&=2^{(s-2)2^{s-2}}\left( 2^{2+\log_23}+2^{s+1+\log_23}\right)\\
\nonumber
&<2^{(s-2)2^{s-2}}( 2^{4}+2^{s+3})
\text{ since $\log_23<2$, $1-2\log_23<-2$}\\
\label{west}
&\leq 2^{(s-2)2^{s-2}}( 2^{s+4})=2^{s+4+(s-2)2^{s-2}}.
\end{align}

We now use this to bound the first Betti number
$b_1(\gr_{\Th})$ from below.

\begin{align}
\label{pig}
  b_1(\gr_{\Th})
  &= 1 +M_{\Th}-E_{1,\tts}
\text{ from Thm.~\ref{masst}\eqref{founded}}\\
\nonumber
&> 1+ 2^{(3s-12)2^{s-2}+1}-2^{s+4+(s-2)2^{s-2}}\text{ by Thm.~\ref{soup}\eqref{soup2}
and \eqref{west}}\\
\label{SU 3 2}
&=1+2^{s2^{s-2}}\big(2^{(2s-12)2^{s-2}+1}-2^{s+4-2^{s-1}}   \big)
\end{align}
For large $s$ the term $2^{s+4-2^{s-1}}$ is negligible compared
to $2^{(2s-12)2^{s-2}+1}$, and so $b_1(\gr_{\Th})$ grows as $2^{\Omega(s2^s)}$.
Additionally,
\[
(2s-12)2^{s-2}+1\geq s+4-2^{s-1}\quad\text{if}\quad s\geq 6,
\]
and hence $b_1(\gr_{\Th})\geq 1$ if $s\geq 6$.
If $s=5$, we have from Table 2 that $h^-(\Q(\zeta_{3\cdot 2^5}))=9$.
Then from \eqref{west1} and Theorem \ref{soup}\eqref{soup2}, we have
\begin{equation}
\label{nuts}
E_{1,3\cdot 2^5}\leq \big(2^{12}+2^{16}\big)9<2^{20}
\quad\text{and}\quad M_{3\cdot 2^5}>2^{25}.
\end{equation}
Hence from \eqref{pig} we have
\begin{equation}
\label{pig1}
b_1(\gr_{3\cdot 2^5})\geq 1+M_{3\cdot 2^5}-E_{1,3\cdot 2^5}
\geq 1+ 2^{25}-2^{20}.
\end{equation}
For $s=4$ the direct computation of \cite[Sect.~7.6]{IJKLZ2}
shows that $b_1(\gr_{48})=20$.

\subsection
{\texorpdfstring{The main theorem for 
\except{toc}{\boldmath{$\PUT\big(\Z_{K_\tts}^{(2)}=\Z[\zeta_{3\cdot 2^s},1/2]\big)$}}
\for{toc}{$\PUT\big(\Z_{K_\tts}^{(2)}=(\Z[\zeta_{3\cdot 2^s},1/2]\big)$}}{The main theorem for 
                 PU2(Z[\unichar{"03B6}3\unichar{"B7}2\unichar{"5E}s,1/2])}}
\label{rubber2}

Combining the class number bounds
in Theorem~\ref{class2}
and Corollary~\ref{cor:ramifying} gives
a  bound on the number $v_r(\gr_{3\cdot 2^s}/\ogr_{\Th})$
of ramified vertices of $\gr_n$ over $\ogr_n$ with $n=\tts$: 
\begin{align}
\label{cart}
\nonumber
v_r(\gr_n/\ogr_n)&\le 2h^-(K_n) +
h^-(\Z_{F_n}[\sqrt{-u_+}\,])\\
\nonumber
&\le 2\cdot3^{1-2^{s-1}}2^{s+1+(s-1)2^{s-2}} +
2^{(s-7/2)2^{s-2}+2}\\
\nonumber
&= 2^{s2^{s-2}}\big(3^{1-2^{s-1}}2^{s+2-2^{s-2}}+2^{(-7/2) 2^{s-2}+2}\big)\\
&\leq 2^{1+s2^{s-2}}\text{ if $s\geq 3$}
\end{align}
since $\log_23(1-2^{s-1})+s+2-2^{s-2}<0$ and  $(-7/2) 2^{s-2}+2<0$ if $s\geq 3$.

We can now bound $E_{0,\tts}=E\big(\PUT\big(\Z_{K_\tts}^{(2)}\big),\Delta\big)$:
\begin{align*}
E_{0,\tts}&=\big(E_{1,\tts}+
v_r(\gr_{\Th}/\ogr_{\Th})\big)/2 \text{  by \eqref{lava1}}\\
\nonumber
&<2^{s+4+(s-2)2^{s-2}}+  2^{1+s2^{s-2} }\text{ by \eqref{west} and \eqref{cart}}\\
&= 2^{1+s2^{s-2}}\big(2^{3+s-2^{s-1}}+1\big)<2^{2+s2^{s-2}}\text{ if $s\geq 4$}
\end{align*}
since $3+s-2^{s-1}<0$ if $s\geq 4$.

Lastly we bound $b_1(\ogr_{\Th})$ for $s\geq 4$:
\begin{align}
\label{supper}
\nonumber
b_1(\ogr_{\Th})&= 1+M_{\Th}/2-E_{0,\tts}\text{ by 
Thm.~\ref{masst}\eqref{founded}}\\
&>1+2^{(3s-12)2^{s-2}}-2^{2+s2^{s-2}}=1+2^{s2^{s-2}}\big(2^{(2s-12)2^{s-2}}-2^2\big)
\end{align}

We see that $b_1(\overline{\gr}_{\Th})$ grows as $2^{\Omega(s2^s)}$ as $s
\to \infty$. By Proposition~\ref{upper} we have
$b_1(\overline{\gr}_{\Th})=2^{\Theta(s2^s)}$.
Further, if $s\geq 6$, then $(2s-12)2^{s-2}\geq 2$ and so
$b_1(\ogr_{\Th})>1$ if $s\geq 6$.
If $s=5$, then from Corollary \ref{cor:ramifying} we have
\begin{align}
\label{gin}
\nonumber
v_r(\gr_{3\cdot 2^5}/\ogr_{3\cdot 2^5})&\leq 2h^-(K_{3\cdot 2^5})+
h^-\big(F_{3\cdot 2^5}(\sqrt{-u_{+}}\,)\big)\\
&\leq 2\cdot 9 + 2^{(1/2)2^3 +2}=2 \cdot 9+2^6
\end{align}
using Table 2 and Theorem \ref{class2}\eqref{class22}.
So substituting   \eqref{nuts} and \eqref{gin} into \eqref{lava1}, 
\begin{equation}
\label{sunny}
E_{0,3\cdot 2^5}\leq \big(2^{20}+18+2^6\big)/2<
2^{20}.
\end{equation}
But then combining \eqref{nuts} with Theorem \ref{masst}\eqref{founded}
gives
\[
b_1(\ogr_{3\cdot 2^5})=1 +M_{3\cdot 2^5}/2-E_{0,3\cdot 2^5}
\geq 1+2^{24}-2^{20}.
\]
  And for $s=4$ the direct 
computation of \cite[Sect.~7.6]{IJKLZ2} shows that $b_1(\ogr_{3\cdot 2^4}) = 8$.

\subsection{Summary of the cyclotomy bounds
at $\fp$ for $\mathbf{n=\tts}$}
\label{blind}

Below we summarize the cyclotomy bounds for the $\tts$ family
in this section.
\begin{theorem}
\label{main2a}
Suppose $n=3\cdot 2^s$ with $n\geq 12$.
\begin{align*}
\hspace*{-1in}\textup{(a)}\hspace*{1.95in}E_{1,\Th}&< 2^{s+4+(s-2)2^{s-2}}\\
E_{0,\Th}& <2^{2+s2^{s-2}}\quad\text{if $s\geq 4$}\\[.1in]
\hspace*{-1in}\textup{(b)}\hspace*{1.74in}
b_1(\gr_{\Th})&>1+2^{s2^{s-2}}\big(2^{(2s-12)2^{s-2}+1}-2^{s+4-2^{s-1}}\big)\\
b_1(\ogr_{\Th})&>1+2^{s2^{s-2}}\big(2^{(2s-12)2^{s-2}}-2^2\big)\quad\text{if
$s\geq 4$}.
\end{align*}
\end{theorem}

\section{A second approach to bounding the corank from below}
\label{sec:bounding-below}

In this section we give an alternate analytic approach to bounding the corank
from below.  
We will use a result of Zograf
\cite{zograf} which relies on the famous inequality $\lambda_{1}\geq 3/16$
of Selberg \cite{Selberg}
as well as work of Yang and Yau
\cite{yy}.
The formulation we need can be found
in \cite[Lemma 1.1]{voight2}.

It is possible to deduce results on $b_1(\gr_n)$ from 
analytic results on eigenvalues of the Laplacian 
because of ``{\em interchanging local invariants}\,'', 
  which shows that
the graph $\gr_n$ arises from the bad
reduction of a Shimura curve over $F_n$.
Interchanging local invariants is  originally due to \v{C}erednik \cite{cer}
with subsequent generalizations and reformulations by Varshavsky
\cite{var1}, \cite{var}.  The specific theorems for 
Shimura curves are in \cite[Sect.~5]{var} and more accessibly summarized
in \cite[Sect.~3]{jlv}.

We begin with preliminary Sections \ref{taco}, \ref{sale},
and \ref{credit} which give the general theory of
interchanging local invariants. 
Then in Sections \ref{super}, \ref{debt}, and \ref{cats}
we specialize to our case of the Hamilton quaternions
$\H_n$ over $F_n=\Q(\zeta_n)^+$ with $n=2^s$
or $n=\tts$, $n\geq 8$.

To establish notation for these 
preliminary sections, we
let $F$ be a totally real number field with a finite prime $\wp$
above the rational prime $p$
and infinite primes $\infty_1$, \ldots , $\infty_r$ with
$r=[F:\Q]$. Denote by $F_+^\times$ the totally positive elements
of $F^\times$.  Let $\Z_F$ be the integers of $F$ with
completions $\Z_{F_w}$ for primes $w$ of $F$ and with
profinite completion $\hat{\Z}_F\cong \prod_w \Z_{F_w}$.
Set 
\[
\Z_F^{\wp}=\Z_F[1/p]=\{\tau\in F^\times\mid \text{$\tau$ is integral at $v$
 for all 
finite  primes }v\neq \wp\} ,
\]
i.e., $\Z_F^{\wp}$ is the ring of $\wp$-integers of $F$.
 Denote by $\A$ the ad\`{e}les of $F$ with
$\Af$ the finite ad\`{e}les of $F$ and $\Afp$ the
finite ad\`{e}les of $F$ away from $\wp$.
Let $\I$, $\If$, $\Ifp$ be  the id\`{e}les of $F$, the finite
id\`{e}les of $F$, and the finite id\`{e}les  of $F$ away from $\wp$, 
respectively.

For a subgroup $H\leq \GLT(R)$ with $R$ a subring of a completion
of $F$, denote by 
$\Pro\!\! H$
its image in $\PGLT(R)$. (If $H\leq \SLT(R)$, then $\Pro\!\! H\leq \PSLT(R)$.)

\subsection{Archimedean uniformization of Shimura curves}
\label{taco}

Let $\Bi$ be a quaternion algebra over $F$
ramified at $\wp$, $\infty_2$, \ldots, $\infty_r$ and split at
$\infty_1$. Let $\O\subseteq \Bi$ be a maximal order.
For each finite prime $w$ of $F$, $\mathcal{O}_w\colonequals
\mathcal{O}\otimes_{\Z_F}\Z_{F_w}$ is
a maximal order in $B^{\rm{int}}_w$. 
Let $\GGi$ be the algebraic group over $F$ associated to
$B^{\rm{int},\times}$; we have $\GGi(F)=\Bit$. Put $\Gif\colonequals \GGi(\Af)$
and $G^{\rm{int},\wp}_{\rm{f}}\colonequals \GGi(\Afp)$.

For a compact open subgroup  $U\subset \Gif$ the corresponding 
complex analytic Shimura curve is
\[
\YiUt\colonequals \left(\fhpm\times (U\backslash \Gif)\right)/\Bit
\]
with $\fhpm\colonequals \BC\setminus \BR$
and $\Bit$ acting on $\fhpm\times (U\backslash \Gif)$
by $(x,g)\gamma = (\gamma^{-1}(x),g\gamma)$
 has a canonical model
$\YiU$ defined over $F$ and embedded into $\BC$ via $\infty_1$.
The geometrically irreducible components of $Y^{\rm int}_U$ are 
\begin{equation}
\label{fox}
\pi_0(Y^{\rm{int}}_U)=\pi_0(\tilde{Y}^{\rm{int}}_U)=U\backslash \Gif/
B^{\rm{int},\times}_+,
\end{equation}
where  
\[
B^{\rm{int},\times}_+=\{b\in B^{\rm{int},\times}\mid \infty_1(\nrd b)>0\}=\{b
\in B^{\rm{int},\times}\mid \nrd b\in F^\times_+\}
\]
with $\nrd: B^{\rm{int},\times}\rightarrow F^\times$
the reduced norm.
Using Strong Approximation reduced norm induces an isomorphism
\begin{equation}
\label{hare}
\nrd:\pi_0(Y^{\rm{int}}_U)=U\backslash\Gif/B^{\rm{int},\times}_+
\stackrel{\sim}{\longrightarrow}
\nrd(U)\backslash \If/F^\times_+ =\If/\nrd(U)F^\times_+\equalscolon \Cl^+_U(\Z_F) .
\end{equation}
By classfield theory, the subgroup $\nrd(U)F^\times_+\subset \If$
corresponds to an abelian extension $F_+(U)/F$ with $\Gal(F_+(U)/F)\cong
\If/\nrd(U)F^\times_+$.  Each (geometrically) 
connected component of $Y^{\rm{int}}_U$
is defined over the abelian extension $F_+(U)/F$ and $\pi_{0}(Y^{\rm{int}}_U)$ is a torsor under $\Gal(F_+(U)/F)$. Since the
connected components of $Y^{\rm{int}}_U$ are all conjugate under Galois,
they all are curves of the same genus $g(U)$ (although not necessarily
isomorphic).  For $a\in\Gif$, set $\Gamma_a(U)=a^{-1}Ua\cap B^{\rm{int},
\times}_+$.  Then each connected component of $\tilde{Y}^{\rm int}_U$
is of the form $\tilde{Y}^{\rm{int}}_U(a)\colonequals\Gamma_a(U)\backslash \fh$
for $a\in\Gif$. Put $\Gamma(U)=\Gamma_a(U)$ in case 
$a=(1,1,\ldots, 1)\in\Gif$. We have
\begin{equation*}
\label{forty}
\pi_0(\tilde{Y}^{\rm{int}}_U)=
\{\tilde{Y}^{\rm{int}}_U(a)\mid a\in U\backslash\Gif/B^{\rm{int},\times}_+\}.
\end{equation*}

The $\YiU$'s form a projective system with inverse limit
\[
\Yi=\lim_{\leftarrow}{}_{\!U}\, \YiU ,
\]
a scheme over $F$ (which is not of finite type)  acted upon
by $\Gif$. We have $U\backslash \Yi=\YiU$. Set $\Gg\colonequals
\GGi(\Afp)$ and fix a quaternion division algebra $\Btp$ over $F_\wp$.
We then have $\Gif=\Btp^\times \times \Gg$. The ring of integers
$\mathcal{O}_{\Btp}$ is normalized by $\Btp^\times$, and $\Btp^\times/
\mathcal{O}_{\Btp}^\times$ is identified with $\Z$ via the valuation of the norm.
Set $Y^{\rm{int},\wp}\colonequals \mathcal{O}_{\Btp}^\times\backslash Y^{\rm{int}}$.
The group $\Gif$ acts on $Y^{\rm{int},\wp}$ via its quotient
\[
\Gif/\mathcal{O}_{\Btp}^\times \cong \Gg\times\Z.
\]
Let $U_0\subset \Gif$ be the compact open subgroup
\[
U_0\colonequals \prod_w\mathcal{O}^\times_w =
\mathcal{O}\otimes_{\Z_F}\hat{\Z}_F\subset \Gif\quad\text{with}\quad
U^{\wp}_0\colonequals \prod_{v\neq \wp}\mathcal{O}^\times_v
\subset\Gfp=\mathcal{G} .
\]
Note that $U_0=U_0^\wp\times\mathcal{O}_{\tilde{B}_\wp}^\times$
and 
\[
\mathcal{O}_{\tilde{B}_\wp}^\times\backslash U_0=U_0^\wp\times
\{0\}\subseteq \mathcal{G}\times\Z=\mathcal{G}\times 
\left(\mathcal{O}_{\tilde{B}_\wp}^\times
\backslash \tilde{B}_{\wp}^\times\right).
\]
We have
\[
Y^{\rm{int}}_{U_0}=U_0\backslash Y^{\rm{int}}=U_0^\wp\backslash Y^{\rm{int},\wp}.
\]

By \cite[Lemma 13.4.9, Cor.~28.6.8]{voight}, we have
$\nrd(U_0)=\prod_w\Z_{F_w}^\times$.  Hence
\[ 
\Cl^+_{U_0}(\Z_F)=\I_{\rm{f}}/\nrd(U_0)F^\times_+=\Cl^+(F),
\]
the strict class group of $F$. So, by \eqref{hare},
\begin{equation}
\label{ghost}
\pi_0(Y^{\rm{int}}_{U_0})\cong \Cl^+(F)
\end{equation}
and each component of $Y^{\rm{int}}_{U_0}$ is defined over 
the strict Hilbert classfield
$\Hi^+(F)$ of $F$.  
Set
\begin{align}
\label{fat}
\mathcal{O}^\times_+&\colonequals\{\gamma\in\mathcal{O}^\times\mid
\infty_1(\nrd\gamma)>0\}=\{\gamma\in\mathcal{O}^\times\mid
\nrd\gamma\in F_+^\times\},\\
\nonumber
\mathcal{O}_1^\times&\colonequals \{\gamma\in\mathcal{O}^\times\mid
\nrd\gamma =1\}
\end{align}
and note that $U_0\cap B_+^\times=\mathcal{O}_+^\times$.
Define
\begin{align}
\label{ranch}
\Gamma_+(\O)&\colonequals \Pro\!\O^{\times}_+\subset\PGL_2^+(F)\\
\nonumber
\Gamma_1(\O)&\colonequals \Pro\!\O^\times_1\subset \PSLT(F).
\end{align}
\begin{lemma}
\label{radish}
Reduced norm induces an isomorphism
\begin{equation}
\label{rat}
\nrd:\Gamma_+(\O)/\Gamma_1(\O)\stackrel{\sim}{\longrightarrow}
\Z_{F,+}^\times/(\Z_F^\times)^2.
\end{equation}
\end{lemma}
\begin{proof}
It is immediate that the map induced by $\nrd$ in \eqref{rat}
in injective.  It is surjective by \cite[Cor.~31.1.11]{voight}.
\end{proof}

\subsection{Nonarchimedean uniformization of Shimura curves}
\label{sale}

Now we turn to the interchanged quaternion algebra.
Let $B$ be a quaternion algebra over $F$ split at $\wp$;
ramified at $\infty_1$, $\infty_2$, \ldots, $\infty_r$; 
and with $B_v\cong B^{\rm{int}}_v$ for all finite primes $v\neq \wp$.
Let $\fO\subset B$ be a maximal $\Z_F^{\wp}$-order 
with $\fO_v\colonequals \fO\otimes_
{\Z_F}\Z_{F_v}$ a maximal order in $B_v$ for each finite prime $v\neq\wp$ of $F$.
Let $\GG$ be the algebraic group over $F$ associated to
$B^\times$; we have $\GG(F)=B^\times$. Put $G_{\rm{f}}\colonequals \GG(\Af)$
and $\Gfp\colonequals \GG(\Afp)$.
We fix an isomorphism
\begin{equation}
\label{birthday}
B^{\rm{int}}\otimes \Afp\cong B\otimes \Afp.
\end{equation}
Set $\mathcal{G}=\Gfp=G_{\rm{f}}^{\rm{int},\wp}$.

Let $L/\Q_p$ be a finite extension with integers
$\Z_L$ having the unique prime ideal $v$. 
For an integer $n>0$, let $L^{(n)}$ be the unramified
extension of $L$ of degree $n$ with integers $\Z_{L^{(n)}}$.
A summary of  results on $v$-adic uniformization of 
curves follows; see \cite[Sect.~3, 4]{JL} and \cite{jlv} for more 
complete treatments
and references. 

The $v$-adic upper half plane $\mathscr{P}_L$ is a formal scheme over $\Z_L$
which admits a natural action by $\PGLT(L)$. Its special fiber
$\mathscr{P}_{L,0}$ has components which are smooth rational curves with
dual graph canonically identified
with the Bruhat-Tits building $\Delta_v$ of $\SLT(L)$.  Suppose
$\overline{\Gamma}$ is a discrete, cocompact subgroup of 
$\PGLT(L)$.  Then the quotient $\overline{\Gamma}\backslash\mathscr{P}_L$
exists and is canonically the formal completion of a scheme
$\mathscr{P}_{\overline{\Gamma}}/\Z_L$ along its closed fiber.
\begin{prop}
\label{dog}
The genus $g(\mathscr{P}_{\overline{\Gamma}})$ of the curve $\mathscr{P}_{\overline{\Gamma}}$
is the first Betti number $b_1(\overline{\Gamma}\backslash \Delta_v)$
of the graph $\overline{\Gamma}\backslash\Delta_v$.
\end{prop}
\begin{proof}
This follows from \cite[Prop.~3-2]{K} once we observe
that Kurihara's graph $(\overline{\Gamma}\backslash \Delta_v)^\ast$ 
is homotopic to $\overline{\Gamma}\backslash \Delta_v$.
\end{proof}

Let $\Z_L^{\rm{nr}}$ be a strict Henselization of $\Z_L$ and let 
$\Fr_L:\Z_L^{\rm{nr}}\rightarrow \Z_{L}^{\rm{nr}}$ be the Frobenius map.
Drinfeld \cite{drin}  constructed a projective system of \'{e}tale covers of the 
$v$-adic upper
half plane with projective limit
\[
\Omega^{\rm{nr}}_L=\mathscr{P}_L\times_{\Spf \Z_L}
\Spf \Z_L^{\rm{nr}},  
\]
viewed as a formal scheme over $\Z_L$ .
For an element $\gamma\in\GLT(L)$ denote by $[\gamma]$ the image
of $\gamma$ in $\PGLT(L)$. We let $\GLT(L)$ act on $\Omega_L^{\rm{nr}}$
by
\[
\gamma:(x,u)\rightarrow ([\gamma]x, \Fr_L^{\val_v(\nrd\gamma)}u)
\]
for $x$ a point of $\mathscr{P}_L$ and $u$ of $\Spf \Z_{L^{\rm{nr}}}$, where
$[\gamma]$ acts on $x$ via the natural action.

Let $\Gamma\subset\GLT(L)$ be a subgroup and set $Z\Gamma\colonequals
\Gamma\cap L^\times$ so that $\Pro\!\Gamma = \Gamma/Z\Gamma$.
As in \cite[Sect.~2]{jlv},
define $k=k(\Gamma)$ and $k_+=k_+(\Gamma)$ by $\val_v(Z\Gamma)=k\Z$
and $\val_v(\det\Gamma)=k_+\Z$.
Then $2k$ is a multiple of $k_+$.  
Assume that $\Pro\!\Gamma\subset\PGLT(L)$
is discrete and cocompact and also that $k(\Gamma)>0$.
Set
\begin{equation}
\label{fit}
\Gamma'\colonequals \{\gamma\in\Gamma\mid 2k|\val_v(\det\gamma)\}.
\end{equation}
Then $Z\Gamma\leq \Gamma'\unlhd\Gamma$ and $\Gamma/\Gamma'$ is
cyclic of order $2k/k_+$.
Then the  quotient 
\[
\Gamma\backslash\Omega_L^{\rm{nr}}=\Pro\!\Gamma\backslash\big(\Z\Gamma\backslash
\Omega_L^{\rm{nr}}\big)=\Pro\!\Gamma\backslash \big(\mathscr{P}_L
\times_{\Spf \Z_L}\Spf \Z_{L^{(2k)}}\big)
\]
exists, and is
canonically the completion along the closed fiber of a 
projective geometrically connected scheme
$\Omega_\Gamma$ defined over $\Z_{L^{(k_+)}}$.  More precisely
$\Pro\!\Gamma'\backslash \mathscr{P}_L$ algebraizes to a projective
geometrically connected variety $\mathscr{P}_{\Pro\!\Gamma'}/\Z_L$
and 
\[
\Omega_\Gamma=(\Gamma/\Gamma')\backslash (\mathscr{P}_{\Pro\!\Gamma'}
\times_{\Z_L}\Z_{L^{(2k)}}).
\]
So $\Omega_\Gamma$ is a \textsf{(Frobenius) twist} of the \textsf{Mumford
uniformized} $\mathscr{P}_{\Pro\!\Gamma'}$.  In particular the genus
$g(\Omega_{\Gamma})$ of the curve $\Omega_\Gamma$ given by
\begin{equation}
\label{real}
g(\Omega_{\Gamma})=g(\mathscr{P}_{\Pro\!\Gamma'})=b_1(\Pro\!\Gamma'\backslash 
\Delta_v)
\end{equation}
by Proposition \ref{dog}.

We now work over $F$ and its completion $F_\wp$.
Let $n\in\Z$ act on  $\Omega_{F_\wp}^{\rm{nr}}=
\mathscr{P}_{F_\wp}\times_{\Spf \Z_{F_\wp}}\Spf \Z_{F_\wp}^{\rm{nr}}$ 
as $\Fr_{F_\wp}^{-n}$ on $\Z_{F_\wp}^{\rm{nr}}$.
This gives an $F_\wp$-rational action of $\GLT(F_\wp)\times \Z$
on $\Omega_{F_\wp}$.  Let $B^\times$ act on $\Omega_{F_\wp}^{\rm{nr}}$ via 
$B^\times\hookrightarrow \GLT(F_\wp)$, and on $\Gg$ through the 
natural embedding.
For a compact open subgroup  $\Uu\subset\Gg$ the corresponding
$\wp$\hspace*{.01in}-analytic Shimura curve is the $F_\wp$-analytic space
\[
\tilde{Y}_\Uu\colonequals \left(\Omega_{F_\wp}^{\rm{nr}}
\times(\Uu\backslash\Gg)\right)/B^\times ,
\]
which
algebraizes canonically to a projective variety $Y_{\Uu}$ over 
$F_\wp$. The inverse limit of $Y_\Uu$ over such $\Uu$'s is a scheme
$Y$ over $F_\wp$ with an $F_\wp$-rational action by $\Gg\times\Z$.
We have $\Uu\backslash Y=Y_{\Uu}$.

\subsection{Interchanging local invariants}
\label{credit}

A special case of the main result of \cite{var1} gives the following
result (\cite[Thm.~5.3]{var}; see also \cite[Thm.~3.1]{jlv}):
\begin{theorem}
\label{sick}
\text{\rm (\v{C}erednik, Varshavsky) }
There is a \textup{(}$\Gg\times \Z$\textup{)}-equivariant, $F_\wp$-rational 
isomorphism
\[
Y^{\rm{int},\wp}\otimes_{F}F_\wp\cong Y.
\]
\end{theorem}

\subsection{Preliminaries for Sections \ref{debt}, \ref{cats}; Selberg-Zograf}
\label{super}

This section consists of common preliminaries for the two sections
which follow. We return to our standard notation from this point on.
We assume $n=2^s$ 
or $n=\tts$, 
$n\geq 8$.  Set  $F=F_{n}=\Q(\zeta_{n})^+$ with 
$\fp=\fp_n$ the unique prime above $p=2$.
The ideal $\fp_n$ is principal with totally real generator
$p_{n}=2+\zeta_n+\zeta_n^{-1}$ if $n=2^s$ and
generator $p_n'=1+\zeta_n+\zeta_n^{-1}$ if $n=\tts$,
cf.~Definition \ref{gen} and Proposition \ref{prin}.
For $n=\tts$ we have $\Nrm_{F_n/\Q}(p_n')=-2$;
all units of $F_n$ have norm $1$ and there is no element of $\Z_{F_n}$
of norm $2$.

\subsubsection{Arithmetic subgroups constructed from quaternion algebras}
\label{eagle1}

Let $\bH$ be the Hamilton quaternions
over $\Q$, and we take  $B=B_n=\bH_{n}=\bH\otimes_{\Q}F_{n}$ with
$B^{\rm{int}}=B^{\rm{int}}_n$ the quaternion algebra over $F_n$ ramified precisely
at $\fp$, $\infty_{2}$,\ldots, $\infty_{[F:\Q]}$. Set
\[
\widetilde{\mathcal{M}}\colonequals\widetilde{\mathcal{M}}_{n}=
\Z_F^{\fp}\langle 1, i, j, k\rangle
\]
to be the \textsf{standard} maximal
$\Z_F^{\fp}$-order of $B$ as in Notation \ref{chicken}.

For a maximal $\Z_F^{\fp}$-order $\fO\subset B$, set
\begin{align}
\label{pind}
\widetilde{\Gamma}_0(\fO)&\colonequals \fO^\times\subset \GLT(F_\fp)\text
{ with }\Gamma_0(\fO)=\Pro\!\widetilde{\Gamma}_0(\fO)\subset \PGLT(F_\fp)\\
\widetilde{\Gamma}_+(\fO) &\colonequals \{\gamma\in \fO^\times\mid
\val_\fp(\nrd\gamma)\text{ is even}\}
\text{ with }\Gamma_+(\fO)=\Pro\!\widetilde{\Gamma}_+(\fO)\subset
\PGLT(F_\fp)\nonumber\\
\widetilde{\Gamma}_1(\fO) &\colonequals \{\gamma\in\fO^\times\mid
\nrd\gamma =1\}
\text{ with }\Gamma_1(\fO)=\Pro\!\widetilde{\Gamma}_1(\fO)\subset \PGLT(F_\fp)
\nonumber
\end{align}
following Ihara, and in agreement with \eqref{gamma}
and \cite[Defn.~24]{IJKLZ2}.
We have 
\[
\widetilde{\Gamma}_0(\fO)\supseteq\widetilde{\Gamma}_+(\fO)
\supseteq\widetilde{\Gamma}_1(\fO)\quad\text{and}\quad
\Gamma_0(\fO)\supseteq\Gamma_+(\fO)\supseteq\Gamma_1(\fO).
\]

Apply these definitions to the maximal $\Z_F^{\fp}$-order 
$\widetilde{\mathcal{M}}=\widetilde{\mathcal{M}}_n$ in $B=B_n$.
\begin{prop}
\label{rabbit}
\begin{enumerate}[\upshape (a)]
\item
\label{rabbit1}
Let $n=2^s$.  Then we have 
\[
\Gamma_1\big(\widetilde{\mathcal{M}}\big)=
\Gamma_+\big(\widetilde{\mathcal{M}}\big)\quad\text{and}\quad 
[\Gamma_0\big(\widetilde{\mathcal{M}}\big):
\Gamma_1\big(\widetilde{\mathcal{M}}\big)=
\Gamma_+\big(\widetilde{\mathcal{M}}\big)]=2.
\]
\item
\label{rabbit2}
Let $n=\tts$.  Then we have
\[
\Gamma_+\big(\widetilde{\mathcal{M}}\big)=
\Gamma_0\big(\widetilde{\mathcal{M}}\big)\quad\text{and}\quad
[\Gamma_0\big(\widetilde{\mathcal{M}}\big)=
\Gamma_+\big(\widetilde{\mathcal{M}}\big)  
:\Gamma_1\big(\widetilde{\mathcal{M}}\big)  ]=2.
\]
\end{enumerate}
\end{prop}
\begin{proof}
As in the proof of Theorem \ref{index}, reduced norm $\nrd$
induces an isomorphism
\begin{equation}
\label{ghost1}
\nrd: \Gamma_0(\widetilde{\mathcal{M}})/
\Gamma_1(\widetilde{\mathcal{M}})\stackrel{\sim}{\longrightarrow} 
\frac{\Z_{F,+}^{\fp,\times}}
{(\Z_F^{\fp,\times})^2}\cong \Z/2\Z.
\end{equation}

\eqref{rabbit1}: Note that $\Z_F^{\fp}$ contains the totally positive
element $p_{2^s}$ with $\val_{\fp}(p_{2^s})=1$. Moreover there is an
element $x\in\widetilde{\mathcal{M}}_{2^s}=\widetilde{\mathcal{M}}$ with
$\nrd x=p_{2^s}\in\Z_F^{\fp,\times}$ by \cite[Cor.~31.1.11]{voight}
or Lemma \ref{ideal}.  But then the image of $x$ in
$\Gamma_0\big(\widetilde{\mathcal{M}}\big)$ is not in 
$\Gamma_+\big(\widetilde{\mathcal{M}}\big)$ and 
$\Gamma_0\big(\widetilde{\mathcal{M}}\big)
\supsetneq \Gamma_+\big(\widetilde{\mathcal{M}}\big)$.  Now  \eqref{ghost1}
implies the assertions in \eqref{rabbit1}.

\eqref{rabbit2}: Let $n=\tts\geq 12$. Then $\gamma =(1+\zeta_{n})
(1+\zeta_{n}^{-1})\in\Z_{F,+}^\times$ and $\gamma\notin(\Z_F^{\fp,\times})^{2}$
by Lemma \ref{notsquares}\eqref{notsquares2}. \
By \cite[Cor.~31.1.11]{voight} there
exists $x\in\widetilde{\mathcal{M}}$ with $\nrd x=\gamma$.  Then
the image of $x$ in $\Gamma_+\big(\widetilde{\mathcal{M}}\big)$
is not in $\Gamma_1\big(\widetilde{\mathcal{M}}\big)$.  Again now
\eqref{ghost1} implies the assertions in \eqref{rabbit2}.
\end{proof}

 We take the compact open subgroup
\[
\mathcal{U}=\mathcal{U}_n=
\prod_{v\neq \fp}\widetilde{\mathcal{M}}_v^\times\subset \Gg=G_{\rm{f}}^{\fp}=
G_{\rm{f}}^{\rm{int},\fp}.
\]
Since $\mathcal{U}\cap B^\times =\widetilde{\mathcal{M}}^\times$, 
we have the following.
\begin{lemma}
\label{thumb}
One connected component of $Y_{\mathcal{U}}$ is $\Omega_
{\tilde{\Gamma}_0(\widetilde
{\mathcal{M}})}$ with ${\tilde{\Gamma}_0\big(\widetilde
{\mathcal{M}}\big)}$ as in \eqref{pind}.
\end{lemma}

Let $U=U_n\subset G^{\rm{int}}_{\rm{f}}=G_{\rm{f}}^{\rm{int},\fp}\times
\tilde{B}_{\fp}^\times$
be the compact open subgroup $U=\mathcal{U}\times \tilde{\mathcal{O}}^\times_\fp$.
\begin{defn}
\label{linear}
We have  $U\cap B^{\rm{int},\times} =\mathcal{M}_{n}^\times$ for a maximal $\Z_F$-order 
$\mathcal{M}_{n}\equalscolon \mathcal{M}$
in 
$B^{\rm{int}}$. 
\end{defn}
\noindent Recall that  
$\mathcal{M}_+^\times$, $\mathcal{M}_1^\times$,
$\Gamma_+(\mathcal{M})\subseteq \PGL_2^+(F)$, and $\Gamma_1(\mathcal{M})
\subseteq \PSLT(F)$ have been defined in  \eqref{fat}, \eqref{ranch}.
In particular since $U\cap B^{\rm{int},\times}_+=\mathcal{M}_+^\times$, we
have the following.
\begin{lemma}
\label{fist}
One component of 
$Y^{\rm{int}}_U=Y^{\rm{int},\fp}_\mathcal{U}$ is $\Gamma_+(\mathcal{M})
\backslash\fh$.
\end{lemma}

Now by  Theorem \ref{sick},
\begin{equation*}
\label{girth}
Y^{\rm{int}}_U\otimes_{F}F_\wp\cong Y_\mathcal{U}.
\end{equation*}
Hence from Lemmas \ref{fist} and \ref{thumb} we see the following.
\begin{prop}
\label{yard}
The genus $g\big(\Gamma_+(\mathcal{M})\backslash\fh\big)$ of the Shimura
curve $\Gamma_+(\mathcal{M})\backslash \fh$ is equal to the genus
$g\big(\Omega_{\tilde{\Gamma}_0(\widetilde{\mathcal{M}})}\big)$ of the Drinfeld-uniformized
curve $\Omega_{\tilde{\Gamma}_0(\widetilde{\mathcal{M}})}$.
\end{prop}

\begin{prop}
\label{urn}
\begin{enumerate}[\upshape (a)]
\item
\label{urn1}
Let $n=2^s$.  Then $\Gamma_1(\mathcal{M})=
\Gamma_+(\mathcal{M})$.
\item
\label{urn2}
Let $n=\tts$.  Then $\Gamma_+(\mathcal{M})/
\Gamma_1(\mathcal{M})\cong \Z/2\Z$.
\end{enumerate}
\end{prop}
\begin{proof}
By Lemma \ref{radish} 
reduced norm induces an isomorphism
\begin{equation}
\label{mouse}
\nrd:\Gamma_+(\mathcal{M})/\Gamma_1(\mathcal{M})\stackrel{\sim}{\longrightarrow}
\Z_{F,+}^\times/(\Z_F^\times)^2.
\end{equation}
\eqref{urn1}: Let $n=2^s$.  
Weber's Theorem (cf.~Proof of Theorem \ref{fields})
asserts that all totally positive units in $\Z_{F_{2^s}}$ are squares,
proving \eqref{urn1} using \eqref{mouse}.\\
\eqref{urn2}: Let $n=\tts$. We have from Theorem \ref{fields}\eqref{fields2}
that $\Z_{F,+}^{\fp,\times}/(\Z_F^{\fp,\times})^2\cong\Z/2\Z$. But
$\Z_{F,+}^\times/(\Z_{F}^\times)^2\hookrightarrow 
\Z_{F,+}^{\fp,\times}/(\Z_F^{\fp,\times})^2$ and $(1+\zeta_n)(1+\zeta_n^{-1})
\in\Z_{F,+}^\times\setminus (\Z_F^{\times})^2$ from Lemma 
\ref{notsquares}\eqref{notsquares2}.  
By \eqref{mouse},
this proves \eqref{urn2}.
\end{proof}

\subsubsection{Selberg-Zograf}
\label{snow}

Set $\mathcal{M}=\mathcal{M}_n$ as in Definition \ref{linear}.
We now consider the area of a fundamental domain of $\Gamma_1(\mathcal{M})$
acting on the Poincar\'{e} upper half plane $\fh$, where we normalize
hyperbolic area $\mu(D)=\frac{1}{2\pi}\int\!\int_{D} \frac{dxdy}{y^2}$
so that an ideal triangle has area $1/2$. A formula of Shimizu
\cite[Eqn.~1]{voight2} gives the area $A\big(\Gamma_{1}(\mathcal{M})\big)$ 
of a fundamental domain
for the Fuchsian group $\Gamma_1(\mathcal{M})$:
\begin{equation}
\label{rant3}
A\big(\Gamma_1(\mathcal{M})\big) = \frac{4}{(2\pi)^{2[F:\Q]}} \Disc(F)^{3/2} 
\zeta_{F}(2) \Phi({\mathcal D}) \Psi({\mathcal N}),
\end{equation}
where, in our case, ${\mathcal D} = \p_2$ and so
$\Phi({\mathcal D}) = \#(\Z_F/\p_2)^\times = 1$ and, since we are
considering a maximal order in the quaternion algebra, 
the ideal $\mathcal N$
is $(1)$ and therefore $\Psi({\mathcal N})=1$ as well.
Comparing \eqref{rant3} with \eqref{stilts} and \eqref{pros}
gives the following
using the mass notation of Definition \ref{masss}.

\begin{thm}
\label{area}
\textup{ \rm{(Shimizu)} }
For $n=2^s$ or $n=\tts$, $n\geq 8$,
we have $A\big(\Gamma_{1}(\mathcal{M}_n)\big)=2M_{n}$.
\end{thm}
From this we can deduce:
\begin{thm}
\label{tub}
\begin{enumerate}[\upshape (a)]
\item
\label{tub1}
If $n=2^s$, then $A\big(\Gamma_+(\mathcal{M}_n)\big)=2M_n$.
\item
\label{tub2}
If $n=\tts$, then $A\big(\Gamma_+(\mathcal{M}_n)\big)=M_n$.
\end{enumerate}
\end{thm}
\begin{proof}
Combine Theorem \ref{area} with Proposition \ref{urn}.
\end{proof}

With $\mathcal{M}\colonequals \mathcal{M}_n$,
the area $A(\Gamma_+(\mathcal{M}))$ of the fundamental domain
 in turn is related
to the Euler-Poincar\'{e} characteristic 
$\chi(\Gamma_+(\mathcal{M}))$: Choosing a sufficiently
small congruence subgroup gives a torsion-free subgroup $H\subseteq
\Gamma_+(\mathcal{M})$ of finite index $d$
 having a fundamental domain with area
$A(H)$.  By Gauss-Bonnet,
\[
dA(\Gamma_+(\mathcal{M}))= A(H)=     -\chi(H\backslash\fh)=-\chi(H)=-d
\chi(\Gamma_+(\mathcal{M})),
\]
and hence from Theorem \ref{tub} we get
\begin{equation}
\label{paddy}
-\chi\big(\Gamma_+(\mathcal{M}_n)\big)= A\big(\Gamma_+(\mathcal{M}_n)\big)=
\begin{cases} 2M_n\text{ if $n=2^s$,}\\
M_n\text{ if $n=\tts$}.
\end{cases}
\end{equation}
The genus $g=g(\Gamma_+(\mathcal{M})\backslash \fh)$ of the Shimura
curve $\Gamma_+(\mathcal{M})\backslash \fh$ is given by
Riemann-Hurwitz (see, e.g., \cite[(2)]{voight2}):
\begin{equation}
\label{fly}
A(\Gamma_+(\mathcal{M}_n))=2g-2+\sum_{q}e_q\Big(1-\frac{1}{q}\Big)=
-\chi\big(\Gamma_+(\mathcal{M}_n)\backslash\fh\big)+\sum_{q}e_q\Big(
1-\frac{1}{q}\Big) ,
\end{equation}
where $e_q$ is the number of elliptic elements in $\Gamma_+(\mathcal{M}_n)$
of order $q\geq 2$, up to equivalence.

We now consider the invariant 
\begin{equation*}
\label{rated}
\tilde{E}_n\colonequals
E(\Gamma_+(\mathcal{M}_n),\fh)=\chi(\Gamma_+(\mathcal{M}_n)\backslash \fh)
-\chi(\Gamma_+(\mathcal{M}_n))
\end{equation*}
of Definition \ref{cream}. Combining \eqref{paddy} and \eqref{fly}
then gives
\begin{equation}
\label{pebble}
\tilde{E}_n=\sum_{q}\Big( 1-\frac{1}{q}\Big)=2-2g(\Gamma_+(\mathcal{M}_n))
+\begin{cases} 2M_n\text{ if $n=2^s$,}\\
M_n \text{ if $n=\tts$};
\end{cases}
\end{equation}
in particular, $\tilde{E}_n\geq 0$.

A result of Selberg and Zograf gives a lower bound for the genus
of a congruence arithmetic Fuchsian group in terms of the area of its fundamental domain.  It is deduced from Selberg's famous bound
$\lambda_1(\Gamma)\geq 3/16$ on the first nonzero eigenvalue of the
Laplacian of $\Gamma$ and states the following 
(see, e.g., \cite[Lemma~1.1]{voight2}).
\begin{thm}
\label{eat}
{\textup{\rm{ (Selberg-Zograf) }}}
Let $\Gamma\subset\PSLT(\BR)$ be a congruence arithmetic Fuchsian
group of genus $g=g(\Gamma\backslash \fh)$ with a fundamental 
domain in $\fh$ of area $A$.
Then $A<64(g+1)/3$.
\end{thm}

We can recast the Selberg-Zograf upper bound for $A(\Gamma_+(\mathcal{M}_n))$
in Theorem \ref{eat} as a lower bound for $\tilde{E}_n$.
Substitute
\[
g(\Gamma_+(\mathcal{M}_n)\backslash \fh)= 1-(1/2)\tilde{E}_n+
\begin{cases} M_n\quad\text{if $n=2^s$,}\\
M_n/2\quad\text{if $n=\tts$}
\end{cases}
\]
from \eqref{pebble} and \eqref{paddy} into the Selberg-Zograf bound
\[
\frac{3}{64}A(\Gamma_+(\mathcal{M}_n))-1<g(\Gamma_+(\mathcal{M}_n)
\backslash \fh)
\]
from Theorem \ref{eat} to obtain:
\begin{theorem}
\label{eclipse}
\textup{(Selberg-Zograf bounds)} 
Let $\tilde{E}_n=E(\Gamma_{+}(\mathcal{M}_n),\fh)$.
We have
\[
g(\Gamma_+(\mathcal{M}_n))>\begin{cases}
\frac{3}{32}M_n -1,\text{ $n=2^s$}\\
\frac{3}{64}M_n-1,\text{ $n=\tts$}
\end{cases}
\text{and }\quad
\tilde{E}_n<
\begin{cases}
\frac{29}{16}M_n+4,\text{ $n=2^s$}\\
\frac{29}{32}M_n+4,\text{ $n=\tts$}.
\end{cases}
\]
\end{theorem}

From \eqref{stilts2} and Theorem \ref{soup} we have
\begin{align}
\label{ride2}
2^{(3s-12.1)2^{s-3}-1/2}> M_{2^s}&=2^{1-2^{s-2}}|\zeta_{F_{2^s}}(-1)|>2^{(3s-13.7)2^{s-3}-1/2}\\
2^{(6s-19.4)2^{s-3}+1}> M_\tts&=2^{1-2^{s-1}}|\zeta_{F_\tts}(-1)|>2^{(6s-22.5)2^{s-3}+1}.
\nonumber
\end{align}
Using the bounds \eqref{ride2} in Theorem \ref{eclipse}
gives the explicit bounds:
\begin{align}
\label{shine}
g(\Gamma_+(\mathcal{M}_{2^s}))&>(3/32)M_{2^s}-1>
2^{(3s-13.7)2^{s-3}-3.92}-1\\
g(\Gamma_+(\mathcal{M}_{\tts}))&>(3/64)M_{\tts}-1>
2^{(6s-22.5)2^{s-3}-3.42}-1\nonumber\\
\tilde{E}_{2^s}&<(29/16)M_{2^s}+4<2^{(3s-12.1)2^{s-3}+0.36}+4
\nonumber\\
\tilde{E}_{\tts}&<(29/32)M_{\tts}+4<2^{(6s-19.4)2^{s-3}+0.86}+4.
\nonumber
\end{align}
In particular the genera $g(\Gamma_+(\mathcal{M}_{2^s}))$\and 
$g(\Gamma_+(\mathcal{M}_{\tts}))$
grow doubly exponentially in $s$.

\subsection{The Hamilton quaternions over $F_{2^s}=\Q(\zeta_{2^s})^+$}
\label{debt}

To begin 
we compute the invariants $k(\Gamma)$, $k_+(\Gamma)$,
and $\Gamma'$ of Section \ref{sale}
for the subgroup $\Gamma =
\widetilde{\Gamma}_0\big(\widetilde{\mathcal{M}}_{2^s}\big)\subset
\GLT(F_\fp)$.
\begin{prop}
\label{grain}
With $\Gamma=\widetilde{\Gamma}_0\big(\widetilde{\mathcal{M}}_{2^s}\big)$   
as above, we have
\[
k(\Gamma)=k_+(\Gamma)=1\quad\text{and}\quad \Gamma'=\widetilde{\Gamma}_+\big(
\widetilde{\mathcal{M}}_{2^s}\big)
\]
with notation as in \eqref{pind}.
\end{prop}
\begin{proof}
We need only note that $\Z_F^{\fp}$ contains the totally positive
element $p_{2^s}$ defined in the beginning of Section \ref{super}
with $\val_\fp(p_{2^s})=1$.  Moreover there is an element
$x\in\widetilde{\mathcal{M}}$ with $\nrd x=p_{2^s}\in\Z_F^{\fp,\times}$ by
\cite[Cor.~31.1.11]{voight} or Lemma \ref{ideal}.
\end{proof}

Finally we are in a position to connect back with coranks 
of cyclotomic unitary groups.
Recall that 
\begin{equation}
\label{putty}
\PSUT\big(\Z[\zeta_{2^s},1/2]\big)\cong
\Gamma_1\big(\widetilde{\mathcal{M}}_{2^s}\big)
\end{equation}
 by \eqref{iso}.
With $\gr_{2^s}\colonequals
\Gamma_1\big(\widetilde{\mathcal{M}}_{2^s}\big)\backslash\Delta_{\fp}$
and \eqref{putty},
the corank
 of the group $\PSUT\big(\Z[\zeta_{2^s},1/2]\big)$ is 
\begin{equation}
\label{paint}
\corank\PSUT\big(\Z[\zeta_{2^s},1/2]\big)=b_1(\gr_{2^s})
\end{equation}
by Corollary \ref{cor:surj-fg} and Remark \ref{dale}.

\begin{thm}
\label{kittya}
\begin{enumerate}[\upshape (a)]
\item
\label{kittya1}
We have 
\[
g\big(\Gamma_+(\mathcal{M}_{2^s})\backslash\fh \big)=
g\big(\Omega_{\widetilde{\Gamma}_0(
\widetilde{\mathcal{M}}_{2^s})}\big)=
b_1\big(\Gamma_+(\widetilde{\mathcal{M}}_{2^s})\backslash\Delta_{\fp}\big)=
b_1(\gr_{2^s})=\corank\PSUT\big(\Z_{K_2^s}^{(2)}\big).
\]
\item
\label{kittya2}
With $E_{1,n}=E(\Gamma_1(\widetilde{\mathcal{M}}_n),\Delta)$ as in 
Definition \textup{\ref{tail}}  ,
we have $2E_{1,2^s}=\tilde{E}_{2^s}$.
\end{enumerate}
\end{thm}
\begin{proof}
\eqref{kittya1}:
Combine \eqref{real}, \eqref{paint}, Proposition \ref{rabbit}\eqref{rabbit1},
Proposition \ref{yard}, and  
Proposition \ref{grain}.\\
\eqref{kittya2}: Combine Theorem \ref{kittya}\eqref{kittya1},\eqref{shade2}, and
\eqref{pebble}.
\end{proof}

\begin{remark}
\label{peas}
From \eqref{shine} and Theorem \ref{kitty}, we get the
explicit $\infty_1$-bounds below arising from Selberg-Zograf.
\begin{align}
\label{sonny}
b_1(\gr_{2^s})=\corank\PSUT\big(\Z[\zeta_{2^s},1/2]\big)>&(3/32)M_{2^s}-1>
2^{(3s-13.7)2^{s-3}-4.41}-1\\
\nonumber
E_{1,2^s}<& (29/32)M_{2^s}+2< 2^{(3s-12.1)2^{s-3}-0.64}+2.
\end{align}
\end{remark}

The Selberg-Zograf bounds in Remark \ref{peas} are strong
enough to prove Main Theorem \ref{main} for $\PSUT\big(\Z_{K_{2^s}}^{(2)}\big)$:
\begin{theorem}
\label{cod}
\textup{(Main Theorem \ref{main} for $\PSUT\big(\Z_{K_{2^s}}^{(2)}\big)$
via the archimedean prime $\infty_1$.) } Let $s\geq 3$.
\begin{enumerate}[\upshape (a)]
\item
\label{cod1}
We have $\corank \PSUT\big(\Z_{K_{2^s}}^{(2)}\big)=b_1(\gr_{2^s})=0$ if and only if 
$n\in\{8,16\}$.
\item
\label{cod2}
We have $\corank \PSUT\big(\Z_{K_{2^s}}^{(2)}\big)=b_1(\gr_{2^s})=2^{\Omega(s2^s)}$.
\end{enumerate}
\end{theorem}
\begin{proof}
\eqref{cod1}: From Remark \ref{peas} we get that $b_1(\gr_{2^s})>0$ if 
$M_{2^s}>32/3$. But from Figure 1 and the lower bound for $M_{2^s}$
in \eqref{ride2}, we have that $M_{2^s}>32/3$ if and only if
$s\geq 5$.\\
\eqref{cod2}: This is implied by the lower bound on $b_1(\gr_{2^s})$
in Remark \ref{peas}. 
\end{proof}
\begin{remark}
\label{fist1}
To conclude that 
$\corank \PSUT\big(\Z_{K_{2^s}}^{(2)}\big)=b_1(\gr_{2^s})=2^{\Theta(s2^s)}$
as in Main Theorem \ref{main}, one need only combine Theorem
\ref{cod}\eqref{cod2} with the much easier upper bound
on $b_1(\gr_{2^s})$ in  Proposition \ref{upper}.
\end{remark}

Our $\infty_1$-bound
in  Remark \ref{peas} is
\begin{equation}
\label{beans}
E_{1,2^s}<(29/32)M_{2^s}+2<2^{(3s-12.1)2^{s-3}-0.64}+2,
\end{equation}
whereas our $\fp$-bound in \eqref{green} is
\begin{equation}
\label{beans2}
E_{1,2^s}<2^{5.5+(s-3.2)2^{s-3}}.
\end{equation}
The dominant terms in these formulas are $2^{3s2^{s-3}}$ in
\eqref{beans} and $2^{s2^{s-3}}$ in \eqref{beans2}. Not surprisingly,
the bound at $\fp$ deduced from delicate cyclotomic arguments
specific to the $2^s$ family is asymptotically much smaller than the
general Selberg-Zograf bound at
$\infty_1$. We give a short table of values to illustrate
the difference.

\begin{figure}[ht]
\begin{center}
\label{table5}
\begin{tabular}{r|l|l}
$s$ & upper bound on $E_{1,2^s}$ at $\infty_1$
& upper bound on $E_{1,2^s}$ at $\fp$\\
\hline 
 $5$&$1994$&$6654$\\
  $6$&$1.04\cdot 10^{14}$&$2.50 \cdot 10^8$\\
  $7$&$4.72\cdot 10^{42}$&$9.08 \cdot 10^{19}$\\
  $8$&$2.75 \cdot 10^{114}$&$7.83 \cdot 10^{47}$\\  
\end{tabular}
\end{center}
\caption{Comparing our upper bounds on $E_{1,2^s}$ at $\infty_1$ 
using Selberg-Zograf and at 
$\fp$
bounding class numbers}
\end{figure}

Our $\infty_1$-bound using Selberg-Zograf is stronger than our
$\fp$-bound for $s=5$, but much
weaker for larger values of $s$.  In particular, we have shown that
$\corank \PSUT\big(\Z_{K_{2^s}}^{(2)}\big)=b_1(\gr_{2^s})$ is of the
form $(1+o(1))M_{2^s}$, and in fact the $o(1)$ diminishes very rapidly
to $0$.  Equivalently, $E_{1,2^s}/M_{2^s}$ goes to $0$ rapidly
as $s\rightarrow \infty$, since our upper bound \eqref{beans2} for $E_{1,2^s}$
has $2^{s2^{s-3}}$ as its leading term asymptotically, while
for $M_{2^s}$ the lower bound \eqref{ride2} has leading term $2^{3s2^{s-3}}$.
\subsection{The Hamilton quaternions over 
$F_\tts=\Q(\zeta_{3\cdot 2^n})^+$}
\label{cats}

As in the $2^s$ case, we begin by computing
the invariants $k(\Gamma)$, $k_+(\Gamma)$,
and $\Gamma'$ of Section \ref{sale}
for the subgroup $\Gamma =
\widetilde{\Gamma}_0(\widetilde{\mathcal{M}}_{\tts})\subset
\GLT(F_\fp)$.
\begin{prop}
\label{grain1}
With $\Gamma=\widetilde{\Gamma}_0(\widetilde{\mathcal{M}}_{\tts})$   
as above, we have
\[
k(\Gamma)=1,\,\,  k_+(\Gamma)=2,\quad\text{and}\quad \Gamma'=
\widetilde{\Gamma}_0(
\widetilde{\mathcal{M}}_{\tts})=
\widetilde{\Gamma}_+(
\widetilde{\mathcal{M}}_{\tts})
\]
in the notation \eqref{pind}.
\end{prop}
\begin{proof}
By Proposition \ref{prin}\eqref{prin2}, $\fp=(p_\tts'\colonequals
1+\zeta_\tts +\zeta_\tts^{-1})$ with $\Nrm_{F_\tts/\Q}(p_\tts')=-2$;
in particular, $p_\tts'$ is not totally real and in fact there
is no element of $\Z_{F_\tts}$ of norm $2$.
The remainder of Proposition \ref{grain1} now follows
from Proposition \ref{rabbit}\eqref{rabbit2}.
\end{proof}

Returning to coranks, recall from Theorem \ref{U}. that 
\begin{equation}
\label{grist}
\PUT(\Z[\zeta_\tts, 1/2])\cong\Gamma_0(\widetilde{\mathcal{M}}_\tts).
\end{equation}
With $\overline{\gr}_{\tts}\colonequals
\Gamma_0(\widetilde{\mathcal{M}}_{\tts})\backslash\Delta_{\fp}$
and \eqref{grist},
the corank
 of the group $\PUT(\Z[\zeta_{\tts},1/2])$ is 
\begin{equation}
\label{paint1}
\corank\PUT(\Z[\zeta_{\tts},1/2])=b_1(\overline{\gr}_{\tts})
\end{equation}
by Corollary \ref{reddish1}.

\begin{thm}
\label{kitty10}
\begin{enumerate}[\upshape (a)]
\item
\label{kitty101}
We have 
\begin{align*}
g\big(\Gamma_+(\mathcal{M}_{\tts})\backslash\fh \big)
=g\big(\Omega_{\widetilde{\Gamma}_0(
\widetilde{\mathcal{M}}_{\tts})}\big)
&=b_1\big(\Gamma_0(\widetilde{\mathcal{M}}_{\tts})\backslash\Delta_{\fp}\big)\\
&= b_1(\overline{\gr}_{\tts})=\corank\PUT(\Z[\zeta_{\tts},1/2]).
\end{align*}
\item
\label{kitty102}
With $E_{0,n}=E(\Gamma_0(\widetilde{\mathcal{M}}_n),\Delta)$ as in 
Definition \textup{\ref{tail}}  ,
we have $2E_{0,\tts}=\tilde{E}_{\tts}$.
\end{enumerate}
\end{thm}
\begin{proof}
\eqref{kitty101}: Combine \eqref{real}, 
\eqref{paint}, Proposition \ref{rabbit}\eqref{rabbit2},
Proposition \ref{yard}, and  
Proposition \ref{grain1}.\\
\eqref{kitty102}: Combine Theorem \ref{kitty10}\eqref{kitty101},
\eqref{lava1}, and \eqref{pebble}.
\end{proof}

\begin{remark}
\label{snitch}
From  \eqref{fox} and Theorem \ref{kitty10}, we get the explicit
$\infty_1$-bounds below arising from Selberg-Zograf.
\begin{align}
\label{sonny1}
b_1(\ogr_{\tts})=\corank\PUT\big(\Z[\zeta_{\tts},1/2]\big)>&(3/64)M_{\tts}-1>
2^{(6s-22.5)2^{s-3}-3.42}-1\\
\nonumber
E_{0,\tts}<& (29/64)M_{\tts}+2< 2^{(6s-19.4)2^{s-3}-0.14}+2.
\end{align}
\end{remark}
\begin{remark}
\label{snitch1}
Theorem \ref{kittya}\eqref{kittya2} and Theorem \ref{kitty10}\eqref{kitty102}
combined with \eqref{pebble} shows that $E_{1,2^s}\geq 0$\and $E_{0,\tts}\geq 0$.
Knowing that $E_{1,2^s}\geq 0$ implies that $E_{0,2^s}\geq 0$ by
\eqref{dealt4}. 
\end{remark}

Our Selberg-Zograf $\infty_1$-bounds in Remark \ref{snitch} are strong
enough to prove Main Theorem \ref{main} for $\PUT\big(\Z_{K_{\tts}}^{(2)}\big)$:
\begin{theorem}
\label{code}
\textup{(Main Theorem \ref{main} for $\PUT\big(\Z_{K_{\tts}}^{(2)}\big)$
via the archimedean prime $\infty_1$.) } Let $s\geq 2$.
\begin{enumerate}[\upshape (a)]
\item
\label{code1}
We have $\corank \PUT\big(\Z_{K_{\tts}}^{(2)}\big)=b_1(\ogr_{\tts})=0$ if and only if $n\in\{12, 24\}$.
\item
\label{code2}
We have $\corank \PUT\big(\Z_{K_{\tts}}^{(2)}\big)=b_1(\ogr_{\tts})=2^{\Omega(s2^s)}$.
\end{enumerate}
\end{theorem}
\begin{proof}
\eqref{code1}: From Remark \ref{snitch} we get that $b_1(\ogr_{\tts})>0$ if 
$M_{\tts}>64/3$. But from Figure 2 and the lower bound for $M_{\tts}$
in \eqref{ride2}, we have that $M_{\tts}>64/3$ if and only if
$s\geq 4$.\\
\eqref{code2}: This is implied by the lower bound on $b_1(\ogr_{\tts})$
in Remark \ref{snitch}. 
\end{proof}
\begin{remark}
\label{fist2}
To conclude that 
$\corank \PUT\big(\Z_{K_{\tts}}^{(2)}\big)=b_1(\gr_{\tts})=2^{\Theta(s2^s)}$
as in Main Theorem \ref{main}, combine Theorem
\ref{code}\eqref{code2} with the upper bound
on $b_1(\ogr_{\tts})$ in  Proposition \ref{upper}.
\end{remark}

\section{Deducing Theorem~\ref{rant} from 
Main Theorem~\ref{main}}
\label{bird}

In this final section we supply proofs for two reductions suggested by Sarnak
\cite{S}.  Both of our proofs use the amalgamated product structure
results of Radin and Sadun \cite{RS},
which we formulate in \cite[Thm.~4.1]{IJKLZ}
as follows:
\begin{theorem}[Radin and Sadun]
  \label{dreadful}
    $\PGn\simeq S_4\ast_{D_{4}}D_{n}$.
\end{theorem}
\noindent After proving the reductions, we easily deduce Theorem~\ref{rant}
from Main Theorem~\ref{main}.

\subsection{The first reduction}
\label{firstr}
The first reduction is stated  in \cite[Page $15^{\rm II}$]{S}:
\begin{reduction}
  \label{first}
  If $\PGn =\PUTz\big(\Z_{K_n}^{(2)}\big)$, then $\langle 2,-1\rangle = (\Z/d\Z)^{\times}$.
  \end{reduction}
\noindent The condition
$\langle 2,-1\rangle = (\Z/d\Z)^{\times}$ is equivalent to the statement
that there is one prime $\p$ above $2$ in $F_n$.

We will need a version of the Normal Subgroup Theorem of
Kazhdan and Margulis; see \cite[IV, Thm.~4.9]{KM}:
\begin{theorem}[Kazhdan and Margulis]
  \label{KM}
  Suppose $\Gamma$ is an irreducible lattice in a connected simple Lie
  group $G$ with finite center and no compact factors with $\rank(G)\geq 2$.
  Then $\Gamma/[\Gamma,\Gamma]$ is finite.
  \end{theorem}

\begin{proof}[Proof of Reduction \textup{\ref{first}}]
  We begin by showing:
  \begin{theorem}
    \label{lattice}
    The projective Clifford-cyclotomic group $\PGn$ is not an irreducible lattice
    in a Lie group $G$ as in Theorem \textup{\ref{KM}} with $\rank (G)\geq 2$.
    \end{theorem}
    \begin{proof}
      By Theorem~\ref{KM}, it suffices to exhibit a subgroup $\Gamma\subseteq\PGn$ of
      finite index with $\#\Gamma/[\Gamma,\Gamma]=\infty$.  Consider the natural action
      of $\PGn$ on a tree induced by its amalgamated product decomposition.  By basic
      properties of actions on trees each element of finite order is conjugate to an element
      in one of the factors.  Thus in particular there are only finitely many conjugacy classes
      of elements of finite order.

      Now pick a prime
      $\q\subseteq \Z_{K_n}^{(2)}$ sufficiently large that none
      of those classes is trivial modulo $\q$.  Let
      $\Gamma$ be the kernel of the map
      $
      \PGn\rightarrow \PUT(\Z_{K_n}^{(2)}/\q) .
      $
      Then $\Gamma$ has no elements of
      finite order, which since $\Gamma$ acts on a tree means it is free and thus
      $\#\Gamma/[\Gamma,\Gamma]=\infty$.
    \end{proof}
    Now suppose the prime $2$ factors in $F\colonequals F_n=K^+$ 
with $K\colonequals K_n$
as $(2)=\p_1\cdots \p_r$
    and fix $n$, dropping the subscript from our usual notation.
    Let $\UU/\Z_F$ be the unitary group over $\Z_F$
    preserving the standard Hermitian
    form
\[
h(\vec{x}, \vec{y})= x_1\overline{y}_1 +x_2\overline{y}_2
\quad\mbox{for}\quad \vec{x}=(x_1,x_2),\,\, \vec{y}=(y_1,y_2).
\]
 Then $\UU(\Z_F)=\UT(\Z_K)$  and
    $\UU(\Z_F^{(2)})=\UT(\Z_K^{(2)})$.  For $\frak{q}$ a prime in $F$
    we have
    the local group $\UU_\frak{q}:=\UU\times_{F}F_\frak{q}$.  If $\frak{q}$
    is split in $K$, then $\UU_{\frak{q}}$ is  isomorphic to $\GLT/F_{\frak{q}}$.
    If the prime $\fq$ of $F$ lies below a unique prime
    $\tilde{\fq}$ of $K$, then $\UU_\fq$ is a genuine unitary
    group over $F_{\fq}$ with corresponding field extension $K_{\tilde{\fq}}$.
    We correspondingly have the projective groups $\PU_{\p_i}$ over $\Z_{F_{\p_i}}$
    with $\PU_{\p_i}(\Z_{F_{\p_i}})=\PUT(\Z_{K_{\tilde{\p}_i}})$.
    Put $\PU_{\p_i}\colonequals \PUT(K_{\p_{i}})$
      for $1\leq i\leq r$.
    Then the diagonal embedding 
    \begin{equation*}
      \PUTz\big(\Z_K^{(2)}\big)\hookrightarrow \PU_{\p_1}\times \cdots \times \PU_{\p_r}
    \end{equation*}
    realizes $\PUTz\big(\Z_{K}^{(2)}\big)
\colonequals \PUTz\big(\Z_{K_n}^{(2)}\big)$ as a
    cocompact lattice in a group of rank $r$,
    cf. \cite[page $15^{\rm II}$]{S}. The local groups
    $\PU_{\p_i}$ are all isomorphic for $1\leq i\leq r$.  They are
    noncompact if $4|n$ and $n\geq 8$ since in this case $\H_n$
    is unramified at all primes above $2$ in $F_n$.
    Hence by Theorem~\ref{lattice} we see that
    $\PGn\neq \PUTz\big(\Z_{K_n}^{(2)}\big)$ 
if $r>1$, proving Reduction~\ref{first}.
\end{proof}
\subsection{The second reduction}
\label{secondr}
The second reduction is stated in \cite[pp.{} $15^{\rm III}$, $15^{\rm IV}$]{S}:
\begin{reduction}
  \label{secondd}
  Suppose $\langle 2,-1\rangle=(\Z/d\Z)^\times$, so that there is
  one prime $\p$ of $F_n=\Q(\zeta_n)^+$ above $2$.  If
  $\PGn=\PUTz\big(\Z_{K_n}^{(2)}\big)$, then the degree $f(\p/2)=1$.
\end{reduction}
We will show the following result:
\begin{theorem}
  \label{fixed}
  Suppose $G\colonequals \PGn\simeq S_4\ast_{D_4}D_n$ acts on a tree $X$
  without terminal vertices,
  possibly with inversions, and with finite stabilizers, so that the
  quotient graph of groups $\overline{X}:=X/G$ is finite.
  Then there exists a vertex fixed by $S_4$ with valence $3$.
\end{theorem}

\begin{lemma}\label{lem:subgroups-conjugate}
  Every $G$-conjugate of $S_4$ containing $D_4$ is $D_n$-conjugate to $S_4$.
\end{lemma}

\begin{proof} Let $g \in G$ and write $g = a_1b_1 \dots a_nb_n$
  with all $a_i \in D_n$ and all
  $b_j \in S_4$, and all $a_i, b_j \notin D_4$ except possibly $a_1, b_n$.

  If $n>1$, then choose $b \in S_4$ such that $b_n b b_n^{-1} \notin D_4$.
  Then 
\[
gbg^{-1} = a_1b_1 \dots a_n(b_nbb_n^{-1})a_n^{-1}\dots b_1^{-1}a_1^{-1}
\]
is
  an expression that cannot be shortened, except possibly by absorbing
  $a_1, a_1^{-1}$ into $b_1, b_1^{-1}$, so it is not in $D_4$.  So for
  all such $b$ we have $gbg^{-1} \notin D_4$.  Now let $b \in S_4$ be such
  that $b_n b b_n^{-1}$ is not in the intersection of the $S_4$-conjugates
  of $D_4$.  Then $c = a_nb_nbb_n^{-1}a_n^{-1}$ has the same property,
  but since $b_{n-1} \notin D_4$ it must be
  that $b_{n-1}cb_{n-1}^{-1} \notin D_4$, and again the expression $gbg^{-1}$
  cannot be shortened further, so it is not in $D_4$.  We conclude that
  if $n>1$ then at most $4 = \#\cap_{p \in S_4} pD_4p^{-1}$
  elements of $gS_4g^{-1}$ belong to $D_4$ and
  hence that $D_4 \not \subset gS_4g^{-1}$.  

  Thus if $D_4 \subset gS_4g^{-1}$, then $n=1$; this means that $g = ab$
  with $a \in D_n$ and
  $b \in S_4$.  But then $gS_4g^{-1} = aS_4a^{-1}$ and the lemma is proved.
\end{proof}

We now return to prove the theorem.
\begin{proof}
  It follows by combining \cite[Cor.~5.4.1]{S2} and \cite[Ex.~6.3.4]{S2}
  that if $G$ is a group generated by torsion elements acting on a tree $X$
  without inversions, then $G\backslash X$ is a tree.  If we replace $X$
  by its barycentric subdivision, we do not affect the truth of the
  conclusion, since we only introduce vertices of valence $2$ and do not
  change the existing ones.  Thus we do so, ensuring that $G$ acts without
  inversions and that $G\backslash X$ is a tree.

  Fix for all time an isomorphism $G\cong S_4\ast_{D_4}D_n$.
  Now choose $v, w\in X$ to be a pair of vertices stabilized by a
  $D_n$-conjugate of $S_4$ and by $D_n$ respectively
  such that $d(v,w)$ is minimal among such pairs.
  Such vertices exist since every finite group
  acting on a tree without inversions has a fixed point;
  see \cite[Example 6.3.4]{S2}.  By applying the same element of $D_n$ to
  $v, w$ we may assume that $\Stab(v) \supseteq S_4$.
  In fact $\Stab(v)=S_4$ since $S_4$  is a
  maximal finite subgroup of $G$.  Notice that $D_4$ fixes both $v$ and $w$;
  hence it fixes the path $\tau = (v = v_0, \dots, v_n = w)$ between them.
  By hypothesis $v_k$ is not stabilized by any $D_n$-conjugate of $S_4$
  for any $k>0$; it follows from the lemma, together with the observation
  just made, that it is not stabilized by any $G$-conjugate of $S_4$.

  Let $e$ be the first edge of $\tau$.  It is not stabilized by $S_4$,
  for if it were then $v_1$ would be $S_4$-stable.
  Notice that we must have $S_4\supsetneq\Stab(e)\supseteq D_4$; hence
  $\Stab(e)=D_4$.  There are precisely $3=[S_4:D_4]$ edges with initial
  vertex $v$ that are in the $S_4$-orbit of $e$.  If these are the only
  edges with initial vertex $v$, we are done.  So suppose there is another
  edge $e'$ with initial vertex $v$ and denote its terminal vertex by $v'$.
  If $e'$ were in the $G$-orbit of $e$, then for some $g$ we would either have
  $g(v,v_1) = (v,v')$ or $g(v,v_1) = (v',v)$.  The first is impossible by our
  assumption that $e'$ is not in the $S_4 = \Stab(v)$-orbit of $e$,
  and the second
  because $\Stab(v), \Stab(v_1)$ are not $G$-conjugate.

  Let $\overline{e}'$ be the image of $e'$ in $\overline{X}=X/G$ with initial
  vertex $o(\overline{e}')=\overline{v}$.  Since $\overline{X}$ is a tree,
  we can take the subgraph of $\overline{X}$ consisting of everything on the
  $\overline{v}$-side of $\overline{e}'$.  Call this subgraph $\overline{X}'$.
  Notice that the image of the entire path $\tau$ is contained in
  $\overline{X}'$. If this were not the case, then there would be some $g\in G$
  such that the edge $ge'$ is in the path $\tau$.  It cannot be that $gv = v$,
  because that would imply that $g \in S_4$ and hence that $e' \in Ge$.
  Nor can it be that $gv \ne v$, because $\Stab(gv) = g\Stab(v) g^{-1}$
  would be a subgroup of $G$ conjugate to $S_4$ and containing $D_4$,
  while $d(gv,w) < d(v,w)$, which is ruled out by our choice of $v$.
  
  Let $X'$ be the preimage of $\overline{X}'$; it is stable under $G$.
  Let $P$ be a path in $\overline{X}$ starting $\overline{v}, \overline{v}'$
  and ending at a terminal vertex of $\overline{X}$ (such vertices exist
  because $\overline{X}$ is a finite tree).  Let $P'$ be a lift of $P$ to $X$
  starting at $v$.  Since $X$ has no terminal vertices,
  at least two edges of $X$ map to the last edge of $P$.  Thus we may
  lift $P^{-1}$ to a path ${P^{-1}}''$
  whose first vertex is the same as the last vertex
  of $P'$ but whose second vertex is not the second last vertex of $P'$.
  Let $gv$ be the last vertex of ${P^{-1}}''$; this notation is justified by
  the fact that ${P^{-1}}''P'$ projects to $P^{-1}P$ in $\overline{X}$.

  The path from $v$ to $gv$ begins with $e'$, so $v, gv$ are in different
  components of $X'$.  However, all generators of $G$ fix $v$ or $w$ or both,
  so they fix the connected component of $X'$ that contains $v$.
  This contradiction shows that no such $e'$ can exist
  and that $v$ has valence $3$.
\end{proof}

\begin{remark} If $G = H \ast_K L$
  where $K$ is a maximal subgroup of $H$ and is not normal and $X$ is a tree
  without terminal vertices such that $G$ acts with finite stabilizers and
  $G \backslash X$ is finite, then $X$ has a vertex fixed by $H$ of valence
  $[H:K]$.  (Note that these hypotheses imply that $L \ne K$.)  The proof
  is essentially the same as above, replacing $S_4, D_4, D_n, 3$ by
  $H, K, L, [H:K]$ respectively.
\end{remark}

\begin{proof}[Proof of Reduction \textup{\ref{secondd}}]
  Suppose there is one prime $\p$ of $F\colonequals F_n$
  above $2$. Then $\H\colonequals \H_n$ is unramified at $\p$.
  The group $\PSUT\big(\Z_{K_n}^{(2)}\big)$ is identified with
  $\Pro\!\widetilde{\M}_{n,1}^\times$ by \eqref{iso} and hence
  is a discrete cocompact subgroup of
  $\Pro\!\H_\p^\times\cong \PGLT(F_{\fp})$; see,
  for example, \cite[Sect.~4]{K}.  This identification
on $\PSUT\big(\Z_{K_n}^{(2)}\big)$
  extends
  to an identification of $\PUT\big(\Z_{K_n}^{(2)}\big)$ with a subgroup of
  $\PGLT(\F_{\fp})$ by \eqref{var}.  Since $\PSUT\big(\Z_{K_n}^{(2)}\big)$ 
is of finite
  index in $\PUT\big(\Z_{K_n}^{(2)}\big)$, it follows that 
$\PUT\big(\Z_{K_n}^{(2)}\big)$ and
  $\Gamma:=\PUTz\big(\Z_{K_n}^{(2)}\big)$ are discrete
  and cocompact in $\PGLT(F_\p)$.
  Let $\Delta_\p$ be the
  tree for $\SL_{2}(K_\p)$ -- it is a regular tree with valence
  $2^f+1$ with $f\colonequals f(\p/2)$ and $\Delta_{\p}/\Gamma$ is finite. Hence
  Theorem~\ref{fixed} implies that if $\PGn=\PUTz\big(\Z_{K_n}^{(2)}\big)$, 
then we must have
  $f=1$.
  \end{proof}

\begin{proof}[Proof of Theorem \textup{\ref{rant}}]

We suppose $n=2^sd$ with $s\geq 2$, $d$ odd, and $n\geq 8$.
Assume $\PGn=\PUTz\big(\Z_{K_n}^{(2)}\big)$.  
By Reductions~\ref{first} and~\ref{secondd}, we have that
the prime $2$ is totally ramified in $K_n$, which means
that $d=1$ or $d=3$.  Hence we are reduced to the cases
$n=2^s$ and $n=3\cdot 2^s$, which are covered by
our Main Theorem~\ref{main}.
\end{proof}

\bibliography{CAdam3f2a1a}{}
\bibliographystyle{plain}

\end{document}